\documentclass{article}
\usepackage{listings}
\usepackage{amsmath, amsthm, amssymb}
\newtheorem{teo}{Theorem}[section]

\newtheorem{prop}[teo]{Proposition}
\newtheorem{lemma}[teo]{Lemma}

\newtheorem{remark}[teo]{Remark}
\usepackage{graphics}
\usepackage{graphicx}
\usepackage[scale=0.75]{geometry}
\usepackage[utf8]{inputenc}
\usepackage[english]{babel}

\usepackage{listings}
\graphicspath{ {./images/} }
\usepackage[output-decimal-marker={,}]{siunitx}
\usepackage{booktabs}
\usepackage{enumitem}
\usepackage{mathtools, nccmath}
\newcommand\numberthis{\addtocounter{equation}{1}\tag{\theequation}}
\usepackage{url}
\usepackage{relsize}

\usepackage{xcolor}
\usepackage{tcolorbox}
\usepackage{verbatim}
\usepackage{bm}

\newtheorem{conjecture}{Conjecture}
\newcounter{hypoconbisfl}
\newcounter{saveconbisfl}
\newcommand\debutTCL{\begin{list} {\textbf{CLT\arabic{hypoconbisfl}}}{\usecounter{hypoconbisfl}}\setcounter{hypoconbisfl}{\value{saveconbisfl}}}
	\newcommand\finTCL{\end{list}\setcounter{saveconbisfl}{\value{hypoconbisfl}}}

\newcounter{hypoconbisf}
\newcounter{saveconbisf}
\newcommand\debutI{\begin{list} {\textbf{I\arabic{hypoconbisf}}}{\usecounter{hypoconbisf}}\setcounter{hypoconbisf}{\value{saveconbisf}}}
\newcommand\finI{\end{list}\setcounter{saveconbisf}{\value{hypoconbisf}}}

\newcounter{hypoconbisx}
\newcounter{saveconbisx}
\newcommand\debutTX{\begin{list} {\textbf{TX\arabic{hypoconbisx}}}{\usecounter{hypoconbisx}}\setcounter{hypoconbisx}{\value{saveconbisx}}}
	\newcommand\finTX{\end{list}\setcounter{saveconbisx}{\value{hypoconbisx}}}

\makeatletter
\newcommand{\mathleft}{\@fleqntrue\@mathmargin0pt}
\newcommand{\mathcenter}{\@fleqnfalse}
\makeatother
\title{}
\usepackage{mathtools}

\DeclarePairedDelimiter\floor{\lfloor}{\rfloor}
\author{
	Roberta Flenghi$^1$\\
	\texttt{roberta.flenghi@enpc.fr}
	\and
	Benjamin Jourdain$^1$\\
	\texttt{benjamin.jourdain@enpc.fr}
}

\date{%
	$^1$Cermics, \'Ecole des Ponts, INRIA, Marne-la-Vall\'ee, France.\\}

\newcommand\blfootnote[1]{%
	\begingroup
	\renewcommand\thefootnote{}\footnote{#1}%
	\addtocounter{footnote}{-1}%
	\endgroup
}
\title{Central limit theorem for the stratified resampling mechanism}

\begin{document}
	\maketitle
		\blfootnote{``This work is
		supported by the french National Research Agency under the grant
		ANR-21-CE40-0006 (SINEQ).''}
              \begin{abstract}
                The stratified resampling mechanism is one of the resampling schemes commonly used in the resampling steps of particle filters. In the present paper, we prove a central limit theorem for this mechanism under the assumption that the initial positions are independent and identically distributed and the weights proportional to a positive function of the positions such that the image of their common distribution by this function has a non zero component absolutely continuous with respect to the Lebesgue measure. This result relies on the convergence in distribution of the fractional part of partial sums of the normalized weights to some random variable uniformly distributed on $[0,1]$, which is established in the companion paper \cite{CLTfract} by overcoming the difficulty raised by the coupling through the normalization. Under the conjecture that a similar convergence in distribution remains valid at the next steps of a particle filter which alternates selections according to the stratified resampling mechanism and mutations according to Markov kernels, we provide an inductive formula for the asymptotic variance of the resampled population after $n$ steps. We perform numerical experiments which support the validity of this formula.
\end{abstract}
\section{Introduction}

Particle filtering, also known as Sequential Monte Carlo methods (see Chapter $4$ in \cite{cappemoulines} and \cite{chopinandpapas} for a general introduction), is a powerful method to estimate the evolving state of a system over time, even when the state cannot be directly observed but can only be inferred through noisy measurements or observations. It has become a very popular class of numerical methods
for the solution of optimal estimation problems in non-linear non-Gaussian scenarios. This kind of
method is used in real-time applications appearing in fields such as chemical engineering,
computer vision, financial econometrics, target tracking, robotics and statistics (see among others \cite{ristic} and \cite{Doucet}). 
The use of Monte Carlo methods for nonlinear filtering problems can be traced back to Handschin \cite{handschin} and Mayne \cite{mayne}. They introduced a sequential version of the importance sampling method (see \cite{evans}, \cite{casella} as references for the importance sampling) and the corresponding algorithm is known as sequential importance sampling, often abbreviated SIS. A drawback of the latter was identified by Gordon et al. in \cite{GordonSalmond}: as the number of iterations increases, the importance weights tend to degenerate (a phenomenon usually known as \emph{weight degeneracy}). This means that after a certain number of iterations, some weights tend to become very small so that the corresponding positions no longer contribute to the estimation.
Thus, Gordon et al. introduced  the resampling step where, in view of stabilizing the Monte Carlo error over time, the key idea is to eliminate the particles having low weights and to replicate the particles having high weights. Therefore, by propagating $M$ particles through weighting, resampling  and mutation steps (each particle evolves randomly according to a given transition probability kernel), particle filters can be used to numerically estimate the state of the system given the observations
(see  also \cite{Doucet}, \cite{chopinandpapas}).\\

 Let us henceforth concentrate on the resampling step. Let $M\geq 1$. Given a sequence of $\mathbb{R}^d$-valued random vectors  $(X_{m})_{1\leq m\leq M}$ with associated  random weights $(w_m^M)_{1\leq m\leq M}$ such that $w_m^M>0$ and $\sum\limits_{m = 1}^Mw_m^M=1$, a resampling scheme defines the resampled sequence $(Y_{m})_{1\leq m\leq M}$ such that

\begin{align}\label{constra}
\mathbb{E}\left(\frac{1}{M}\sum_{m=1}^{M}\delta_{Y_{m}} \mathrel{\Big|}\mathcal{G} \right) 
=\sum_{m=1}^{M}w_m^M\delta_{X_{m}}.
\end{align} 
where $\mathcal{G}:=\sigma\left((X_{m},w^M_m)_{1\leq m\leq M}\right) $.
 Resampling schemes \cite{Douc} permit to replace the probability measure $\sum_{m=1}^{M}w_m^M\delta_{X_{m}}$ with non equal weights by some empirical measure $\frac{1}{M}\sum\limits_{m=1}^{M}\delta_{Y_{m}}$ with the same conditional expectation given $ \mathcal{G}$. Depending on the definition used for the random variables $Y_m$,  several resampling schemes can be considered. The most common resampling techniques are of the following types: multinomial, residual , stratified
 \cite{kitagawa} and systematic  \cite{Carpenter}. See \cite{Douc} for a brief description of such methods. 
 The simplest approach is the multinomial resampling. It consists in drawing, conditionally upon $ \mathcal{G}$, the new positions $\left(Y_{m}\right)_{1\leq m \leq M} $ independently from the common distribution $\sum_{m=1}^Mw^M_m\delta_{X_m}$. While the residual resampling consists in replicating  $\lfloor Mw^M_m\rfloor$-times $X_m$ for $m\in\{1,\cdots,M\}$ and the remaining $M-\sum_{m=1}^M \lfloor Mw^M_m\rfloor$ variables $Y_m$ are drawn, conditionally upon $ \mathcal{G}$, independently from the common distribution $\frac{1}{M-\sum_{m=1}^M \lfloor M w^M_m\rfloor}\sum_{\ell=1}^{M}\left\lbrace Mw^M_\ell \right\rbrace\delta_{X_{\ell}}$.
 Concerning the stratified resampling, it is not straightforward to understand its behaviour, notably due to the complicated structure in the definition of the $Y_m$ which, however, continue to be conditionally independent given $ \mathcal{G}$. The systematic resampling is even more complicated to understand since  the $Y_m$ are no longer conditionally independent given $ \mathcal{G}$.\\

Resampling schemes have been largely studied in the literature, we now present a selection of such results. The asymptotic behaviour for the multinomial resampling scheme has been extensively studied in \cite{delmoral} (see Corollary $7.4.2$ and Section $9.4.2$). Douc et al. in \cite{Douc} showed that residual and stratified resampling improve over multinomial resampling in the sense that they have a lower conditional variance (with respect to the $\sigma$-algebra generated by $(X_{m})_{1\leq m\leq M}$). They also proved, by means of a counter-example, that the same property does not hold for systematic resampling. Furthermore, they established a central limit theorem for the residual resampling approach suggesting that a similar result should be obtained for the stratified resampling scheme.
 One of the last contributions concerning resampling schemes, is given by Gerber et al. \cite{chopin}. Using the notion of negative association \cite{Proschan}, they first provided a general consistency result for resampling. An application of this theorem gives the proof of almost sure weak convergence of   $\frac{1}{M}\sum\limits_{m=1}^{M}\delta_{Y_{m}}$ in the stratified resampling method. Moreover they provide a counter-example to almost sure weak convergence for the systematic resampling method.
More recently, Chopin et al. \cite{chopin2} studied the resampling schemes for particle filters with weakly informative observations. Empirical evidence indicates that when the weights used in resampling exhibit high variability, the selection of the resampling strategy tends to have a weak impact. However, in cases where the weights are close to being uniform,  the performance differences between the different resampling methods can be substantial. By keeping $M$ fixed, they also considered the asymptotic behaviour of the resampling schemes as the weights become less and less informative. 
See also \cite{Liu}, \cite{Fearnhead}, \cite{Kunsch}, \cite{NChopin} for additional references.\\

In this paper, we focus on the stratified resampling scheme for the weights $w_m^M=g(X_m)\big{/}\sum\limits_{\ell=1}^{M}g(X_\ell)
$ where $g$ is a positive measurable function. Our purpose is to study the asymptotic behaviour of the method as the number of particles $M$ goes to $\infty$. Let $f$ be a real-valued measurable function. To compute the asympotic variance of $\frac{1}{\sqrt{M}}\sum\limits_{m=1}^{M}f\left( Y_{m}\right)$, it turns out to be essential to understand the behaviour as $M\rightarrow \infty$ of 
\begin{align}\label{altropaper}
	\left\lbrace \tfrac{M\sum_{\ell=1}^{\lceil\alpha M \rceil-1}g\left(X_\ell \right)}{\sum_{\ell=1}^{M}g\left(X_\ell \right) } \right\rbrace
\end{align}
where $0< \alpha < 1$ and where $\left\lbrace x \right\rbrace$ denotes the fractional part of $x \in \mathbb{R}$. Under the assumption that the $X_i$ are i.i.d. such that the law of $g\left(X_i\right)$ has an absolutely continuous component, we prove in the companion paper \cite{CLTfract} the convergence in distribution of (\ref{altropaper}) to a random variable uniformly distributed  on $\left[0,1 \right]$. Under this assumption, we explicit the limit of $\textrm{Var}\left(\frac{1}{\sqrt{M}}\sum\limits_{m=1}^{M}f\left( Y_{m}\right) \right) $ and prove that $\frac{1}{\sqrt{M}}\sum_{m=1}^{M}\left(f(Y_m^{M})-\frac{{\mathbb E}(f(X_1)g(X_1))}{{\mathbb E}(g(X_1))}\right)$ converges in distribution to a centered Gaussian random variable with variance equal to the limit. The proof of this Central Limit Theorem relies on the asymptotic behaviour as $M\rightarrow \infty$ of a vector composed of $\sqrt{M}\left( \dfrac{\sum_{m=1}^{M}g(X_{m})f(X_{m})}{\sum_{\ell=1}^{M}g(X_{\ell})} - \mathbb{E}\left(f\left(X_1 \right) g\left(X_1\right)  \right)\right)$ and the fractional parts (\ref{altropaper}) for $\alpha\in\{\alpha_1,\cdots,\alpha_s\}$ with $0<\alpha_1<\cdots<\alpha_s<1$. In the companion paper \cite{CLTfract}, we check that this vector converges in distribution to a vector with  centered Gaussian first component and independent $s$ last components uniformly distributed on $[0,1]^s$. Under the conjecture that a similar convergence in distribution remains valid at the next steps of a particle filter which alternates selections according to the stratified resampling mechanism and mutations according to Markov kernels, we provide an inductive formula for the asymptotic variance of the resampled population after $n$ steps.\\

The paper is organized as follows. 
In Section \ref{statmainresult}, we recall the definition of the stratified sampling scheme and the statement of the main result (Theorem \ref{risultato_principale}) is given. In Section \ref{secasymptoticvariance}, the asymptotic variance   is derived and in Section \ref{secclt}, the proof of the Central Limit Theorem is given. In Section \ref{sectionfinale}, the proof of some ausiliary results is provided. In Section \ref{nextsteps}, we consider a particle filter which alternates selections according to the stratified resampling mechanism and mutations according to Markov kernels and we provide  an inductive formula for the asymptotic variance of the resampled sequence.
We perform numerical experiments which support the validity of this formula.\\ 

	\textbf{Notation}
We denote by $\floor*{x}$ the integer $j$ such that $j\leq x < j+1$ and by $\{x\}=x-\floor*{x}$ the fractional part of $x\in \mathbb{R}$. We denote the set of real-valued bounded measurable functions on  $\mathbb{R}^d$ by $\mathcal{B}_{b}\left(\mathbb{R}^d\right)$. Given $\mu$ a positive measure on $\mathbb{R}^d$ endowed with the Borel sigma algebra and $\phi:\mathbb{R}^d \rightarrow \mathbb{R}$ a measurable function that is either positive or such that $\int_{\mathbb{R}^{d}}\left| \phi(x)\right|\mu(dx)<\infty$, we denote $\mu(\phi)= \int_{\mathbb{R}^{d}}\phi\left( x\right)\mu(dx) $. \\ 
	\section{Statement of the Main Result}\label{statmainresult}
	Given $M>0$, let $(X_{m})_{1\leq m\leq M}$ be a sequence 
	of i.i.d. $\mathbb{R}^d$-valued random vectors following the law $\eta$ and let $g:\mathbb{R}^d \rightarrow \left(0,\infty \right) $ be a measurable function such that $0<\inf_{x\in \mathbb{R}^d}g(x)\leq\sup_{x\in \mathbb{R}^d}g(x)<\infty $. In what follows we denote $\bar{g}:=\sup_{x\in \mathbb{R}^d}g(x)$ and $\underline{g}:=\inf_{x\in \mathbb{R}^d}g(x) $.\\
	 We now generate the sequence $(Y_m^{M})_{1\leq m \leq M}$ according to the selection step of the stratified sampling. We recall that starting from $M$ random variables $(U_m)_{1\leq m\leq M}$ i.i.d. distributed according to the uniform law on $\left( 0,1\right)$ and independent of $(X_{m})_{1\leq m\leq M}$, the sequence $(Y_m^{M})_{1\leq m \leq M}$ is defined in the following way 
	
	\begin{equation}\label{stratsampling}
	Y_m^{M}=\sum_{\ell=1}^{M}1_{\left\lbrace \sum\limits_{j=1}^{\ell-1}w_j^M<m-U_m\leq\sum\limits_{j=1}^{\ell}w_j^M \right\rbrace }X_{\ell} \,\, for\,m\in \left\lbrace 1,\cdots, M \right\rbrace 
	\end{equation}

	where $w_m^M=\frac{Mg(X_m)}{\sum\limits_{\ell=1}^{M}g(X_\ell)}$  for $m=1,\cdots,M$. Since the weights are preserved when multiplied by a positive constant, up to dividing $g$ by $\mathbb{E}(g(X_1))$, without loss of generality we may suppose that $$\mathbb{E}(g(X_1))=1.$$\\
	Let us observe that the random vectors $(Y_m^{M})_{1\leq m \leq M}$ are conditionally independent given $\mathcal{F}$, the $\sigma$-algebra generated by the sequence $(X_{m})_{1\leq m\leq M}$. Moreover one has 
	\begin{equation}\label{selection}
	\mathbb{E}\left( \frac{1}{M}\sum_{m=1}^{M}\delta_{Y_m^{M}}\arrowvert\mathcal{F}\right)=\dfrac{1}{\sum\limits_{m=1}^{M}g(X_{m})} \sum\limits_{m=1}^{M}g(X_{m})\delta_{X_{m}}.	
	\end{equation}
	Given  $f:\mathbb{R}^d\rightarrow \mathbb{R}$ a measurable function, our purpose is to provide a central limit theorem for $\dfrac{1}{M}\sum\limits_{m=1}^{M}f(Y_m^{M})$. If we start looking at the variance of $\dfrac{1}{\sqrt{M}}\sum\limits_{m=1}^{M}f(Y_m^{M})$, using (\ref{selection}) one has 
	\begin{align}
	\textrm{Var}\left(\dfrac{1}{\sqrt{M}}\sum_{m=1}^{M}f(Y_m^{M})\right) &= \textrm{Var}\left( \mathbb{E}\left(\dfrac{1}{\sqrt{M}}\sum_{m=1}^{M}f(Y_m^{M})\mathrel{\Big|}\mathcal{F} \right)\right)+\mathbb{E}\left( \textrm{Var}\left(\dfrac{1}{\sqrt{M}}\sum_{m=1}^{M}f(Y_m^{M}) \mathrel{\Big|}\mathcal{F}\right)\right)\label{cond_variance}\\
	&=\textrm{Var}\left( \sqrt{M}\dfrac{\sum_{m=1}^{M}g(X_{m})f(X_{m})}{\sum_{\ell=1}^{M}g(X_{\ell})} \right)+\mathbb{E}\left( \textrm{Var}\left(\dfrac{1}{\sqrt{M}}\sum_{m=1}^{M}f(Y_m^{M}) \mathrel{\Big|}\mathcal{F}\right)\right)\label{splitting}.
\end{align}
Thus, in particular, we will prove respectively the convergence of the first and second term in (\ref{splitting}). The first term is common to all the resampling schemes while the second one really depends on the considered resampling scheme.\\
	
Let us consider the following hypotheses:
\debutI
\item $\sup_{x\in \mathbb{R}^d}\left| f(x)\right| <\infty $
\item the law of $g(X_1)$ has an absolutely continuous component with respect to the Lebesgue measure on $\mathbb{R}$
\finI

Before providing the statement of the central limit theorem, to express the asymptotic variance we introduce the real-valued functions $\beta_0$ and $\beta_{1}$ respectively given by

\begin{align}\label{beta0}
\beta_0(x,y_1)&=\left\lbrace x+y_1 \right\rbrace (1-\left\lbrace x+y_1 \right\rbrace)+x(1-x)-2x(1-x-y_1)1_{\left\lbrace y_1 <1-x \right\rbrace}
\end{align}

\begin{align} 
\beta_1(x,y_1,y_2,y_{3}) & =2\bigg(\left\lbrace x+y_1 \right\rbrace\bigg(1-\left\lbrace x+y_1 \right\rbrace-y_2\bigg)  1_{\left\lbrace y_2 <1-\left\lbrace x+y_1 \right\rbrace \right\rbrace}\label{beta1}\\
&\phantom{=}-\left\lbrace x+y_1 \right\rbrace\bigg(1-\left\lbrace x+y_1 \right\rbrace-y_2-y_3\bigg) 1_{\left\lbrace y_2+ y_3<1-\left\lbrace x+y_1 \right\rbrace \right\rbrace} \nonumber\\ 
& \phantom{=}- x\bigg(1-x-y_1-y_2\bigg) 1_{\left\lbrace y_{1}+y_{2}<1-x\right\rbrace }+x\left(1-x-y_1-y_2-y_3 \right) 1_{\left\lbrace y_{1}+y_{2}+y_{3}<1-x\right\rbrace }\bigg)\nonumber.
\end{align}
\begin{remark}\label{continuitybeta}
		Let us observe that $\beta_0$ is continuous.
		Indeed, the fractional part is continuous apart from the integers where its left-hand limit is equal to $1$ and its right-hand limit is equal to $0$. The composition with the function $f\left(z\right)=z\left(1-z\right)$ which is such that $f\left(0 \right)= f\left(1 \right) =0$, allows us to conclude that $\left( x,y_1\right)\mapsto \left\lbrace x+y_1 \right\rbrace (1-\left\lbrace x+y_1 \right\rbrace) $ is continuous. Moreover $\left( x,y_1\right)\mapsto x(1-x-y_1)1_{\left\lbrace y_1 <1-x \right\rbrace}$ is continuous  since the function $\bar{f}(x,y_1)= 1-x-y_1$ is equal to $0$ on the set $\left\lbrace (x,y_1)\in \mathbb{R}^2 : y_1=1-x \right\rbrace$ of discontinuity points  of the indicator function. Similarly it is possible to prove that also $\beta_1$ is a continuous function on $\mathbb{R}\times\mathbb{R}\times\mathbb{R}^+\times\mathbb{R}^+ $. 
\end{remark}

The following result holds.
	
	\begin{teo}\label{risultato_principale}
		Under the notation introduced above and under \textbf{I1-2} 
		we have 
		
		\begin{equation}
			\lim_{M\rightarrow \infty} \textrm{Var}\left(\dfrac{1}{\sqrt{M}}\sum_{m=1}^{M}f(Y_m^{M})\right) =\sigma^2_1(f)+\sigma^2_2(f)
		\end{equation}
		where 
		\begin{equation}\label{sigma1}
		\sigma^2_1(f)= \eta\left(\left(  g\left(f-\eta(fg) \right)\right)^2\right) \end{equation}
		 and 

		\begin{align}\label{sigma2}
		\sigma_2^2(f) := \sum\limits_{k= 0}^{\left\lceil\frac{\bar{g}}{\underline{g}}\right\rceil}\mathbb{E}\left(F_{k}\right)
		\end{align}
		with $\left( F_{k}\right)_{k\in \mathbb{N}} $  given by 
		\begin{equation} \label{defpk}
		F_{k} = \begin{cases}
		f^2(X_{1})\beta_{0}\left( U_1,g(X_1)\right) &\,k=0\\
		-f( X_{1})f( X_{k+1})\beta_{1}\left( U_1,g(X_1),\sum\limits_{\ell=2}^{k}g(X_{\ell}),g(X_{k+1})\right) &\, k>0
		\end{cases}
		\end{equation}  
		where $U_1\sim \mathcal{U}(0,1)$ is independent of $X_1,\cdots,X_{k+1}$. 
		
		Moreover the following convergence in distribution holds
				\begin{equation}
		\sqrt{M}\left( \frac{1}{M}\sum_{m=1}^{M}f(Y_m^{M})-\eta(fg) \right)\overset{d}{\Longrightarrow} \mathcal{N}\left( 0,\sigma^2_1(f)+\sigma^2_2(f)\right).
		\end{equation} 
		
	\end{teo}

	We will split the proof of the theorem into two parts: in Section \ref{secasymptoticvariance}, we are going to prove the result about the asymptotic variance and in Section \ref{secclt}, the proof of the Central Limit Theorem is provided.

\section{Asymptotic Variance}\label{secasymptoticvariance}

 The following proposition provides the asymptotic behaviour of the first term in (\ref{splitting}). We provide its proof for the sake of completeness.
	
	\begin{prop} Under \textbf{I1}, the following convergence in distribution holds
		\begin{equation} \label{tcl}
		\sqrt{M}\left(\dfrac{\sum_{m=1}^{M}g(X_{m})f(X_{m})}{\sum_{\ell=1}^{M}g(X_{\ell})} -\eta(fg) \right)\overset{d}{\Longrightarrow} \mathcal{N}\left( 0,\sigma^2_1(f)\right)
		\end{equation}
		where $\sigma^2_1(f)$ has been defined in (\ref{sigma1}). Moreover
		\begin{equation}\label{purpose}
		\lim_{M\rightarrow \infty}\textrm{Var}\left( \sqrt{M}\dfrac{\sum_{m=1}^{M}g(X_{m})f(X_{m})}{\sum_{\ell=1}^{M}g(X_{\ell})} \right) =\sigma^2_1(f).
		\end{equation}
	\end{prop}
	\begin{proof}
		 We recall that we suppose that $\eta(g)=1$. Let us define $h:=g\left(f-\eta(fg) \right)$. By observing that $\eta(h)=0$ and \begin{equation} \label{riscrittura}
		\sqrt{M}\dfrac{\sum_{m=1}^{M}g(X_{m})f(X_{m})}{\sum_{\ell=1}^{M}g(X_{\ell})} -\sqrt{M}\eta(fg)=\dfrac{M}{\sum\limits_{m=1}^{M}g(X_{m})}\times\frac{1}{\sqrt{M}}\sum\limits_{m=1}^{M}h(X_{m}),
		\end{equation} by Slutsky's theorem we have
		$$ \sqrt{M}\left(\dfrac{\sum_{m=1}^{M}g(X_{m})f(X_{m})}{\sum_{\ell=1}^{M}g(X_{\ell})} -\eta(fg) \right)\overset{d}{\Longrightarrow} \mathcal{N}\left( 0,\eta(h^2)\right).$$
		Let us now prove (\ref{purpose}). Let us preliminary study the following quantity 
		\begin{align*}
		\mathbb{E}\left(\left( \frac{1}{\sqrt{M}}\sum\limits_{m=1}^{M}h(X_m)\right)^4  \right) = \frac{1}{M^2}\sum\limits_{m_1,m_2,m_3,m_4=1}^{M}\mathbb{E}\left(\prod_{i=1}^{4} h(X_{m_i})\right).
		\end{align*}
		
		Among the $M^4$ expectations appearing in the sum on the right-hand side, the ones where one index is different from the other three are equal to $0$ by the independence of the $X_i$ and the fact that $\eta(h)=0$. Therefore 
		
		\begin{align}
		\mathbb{E}\left(\left( \frac{1}{\sqrt{M}}\sum\limits_{m=1}^{M}h(X_m)\right)^4  \right) = \frac{1}{M}\eta(h^4)+\frac{3(M-1)}{M}\eta(h^2)^2\leq D, \label{boundness}
		\end{align}
		where $D$ is a finite constant depending on $g$ and $f$ but not on $M$.\\
		Since (\ref{riscrittura}) holds, to prove (\ref{purpose}) we can study the convergence as $M$ goes to infinity of the variance of  $\sqrt{M}\dfrac{1}{\bar{g}_M} \bar{h}_M$ given by 
		$$\mathbb{E}\left(\left(\sqrt{M}\dfrac{\bar{h}_M}{\bar{g}_M}\  \right)^2  \right) - \mathbb{E}\left(\sqrt{M}\dfrac{\bar{h}_M}{\bar{g}_M} \right)^2$$ 
		where we denote by $\bar{g}_M$ the empirical mean $\dfrac{1}{M} \sum\limits_{m=1}^{M}g(X_{m})$  and by $\bar{h}_M$ the empirical mean $\dfrac{1}{M} \sum\limits_{m=1}^{M}h(X_{m})$.\\
		
		By the Cauchy–Schwarz inequality we have 
		
		\begin{align*}
		\mathbb{E}\left(
		\left(\sqrt{M}\frac{\bar{h}_M}{\bar{g}_M} -\sqrt{M}\bar{h}_M \right)^2 \right) &=\mathbb{E}\left(
		\left( \sqrt{M}\bar{h}_M\right) ^2\left(\frac{1}{\bar{g}_M} -1 \right)^2 \right)\\
		& \leq \mathbb{E}^\frac{1}{2}\left(\left( \sqrt{M}\bar{h}_M\right) ^4 \right) \mathbb{E}^\frac{1}{2}\left(\left(\frac{1}{\bar{g}_M} -1 \right)^4 \right)\leq D^\frac{1}{2} \cdot \mathbb{E}^\frac{1}{2}\left(\left(\frac{1}{\bar{g}_M} -1 \right)^4 \right),
		\end{align*}
		
		where the right-hand side converges to $0$ as $M$ goes to infinity by the Strong Law of Large Numbers and by Lebesgue's theorem. With $\eta(h)=0$, this in particular implies that  
		
		\begin{align*}
		0=\lim_{M\rightarrow\infty} \left| \mathbb{E}^\frac{1}{2}\left( \left(\sqrt{M}\frac{\bar{h}_M}{\bar{g}_M} \right)^2 \right)-\mathbb{E}^\frac{1}{2}\left( \left(\sqrt{M}\bar{h}_M \right)^2 \right)  \right|  =  \lim_{M\rightarrow\infty} \left| \mathbb{E}^\frac{1}{2}\left( \left(\sqrt{M}\frac{\bar{h}_M}{\bar{g}_M} \right)^2 \right)-\eta(h^2)^\frac{1}{2} \right|
		\end{align*}
		
		and
		\begin{align*}
		0=\lim_{M\rightarrow\infty} \left| \mathbb{E}\left( \sqrt{M}\frac{\bar{h}_M}{\bar{g}_M}  \right)-\mathbb{E}\left( \sqrt{M}\bar{h}_M \right)  \right| = \lim_{M\rightarrow\infty} \left| \mathbb{E}\left( \sqrt{M}\frac{\bar{h}_M}{\bar{g}_M}  \right) \right| 
		\end{align*}
		and so the proof is complete.
	\end{proof}

	Let us now study the second term in the expression (\ref{splitting}). Our purpose is to prove the following result:
	\begin{teo}\label{mainresult}
	Under	\textbf{I1-2}, the following convergence holds
		\begin{align}
		\lim_{M\rightarrow\infty}&	\mathbb{E}\left( Var\left(\dfrac{1}{\sqrt{M}}\sum_{m=1}^{M}f(Y_m^{M}) |\mathcal{F}\right)\right)\label{final_result0}=\sigma_2^2(f)
		\end{align}
		where $\sigma_2^2(f)$ has been defined in (\ref{sigma2}).

	\end{teo}

	To prove Theorem \ref{mainresult} we need some preliminary results the proofs of which are given in Section \ref{sectionfinale}. By the conditional independence of the random vectors $(Y_m^{M})_{1\leq m \leq M}$ with respect to the $\sigma$-algebra $\mathcal{F}$, we have
	\begin{align}
	\textrm{Var}\left(\dfrac{1}{\sqrt{M}}\sum_{m=1}^{M}f(Y_m^{M}) \mathrel{\Big|}\mathcal{F}\right)&=\dfrac{1}{M} \sum_{m=1}^{M} \textrm{Var}\left(f(Y_m^{M}) \mathrel{\Big|}\mathcal{F}\right)\\
	&=\dfrac{1}{M}\sum_{m=1}^{M} \mathbb{E}\left(f^2(Y_m^{M}) \mathrel{\Big|}\mathcal{F}\right)- \dfrac{1}{M}\sum_{m=1}^{M} \mathbb{E}\left(f(Y_m^{M}) \mathrel{\Big|}\mathcal{F}\right)^2\label{toberewritten2}.
	\end{align}

	We introduce the following notation that will be useful in what follows: for $i=1,\cdots,M$, let us denote $u_{i}^M:= \left\lbrace w_1^M+\cdots+w_{i}^M \right\rbrace ,$ $\mu_{i}^M:= \lfloor w_1^M+\cdots+w_{i}^M \rfloor+1$  where $w_{i}^M=\frac{Mg(X_i)}{\sum\limits_{\ell=1}^{M}g(X_\ell)}$ has been introduced in (\ref{stratsampling}) and by convention  $u_{0}^M=0$, $\mu_{0}^M=1$.

We now rewrite in a more explicit way the conditional variance  $\textrm{Var}\left(\frac{1}{\sqrt{M}}\sum\limits_{m=1}^{M}f(Y_m^{M}) |\mathcal{F}\right)$. The proof is given in Section \ref{sectionfinale}.
	\begin{prop}\label{rewriting_term}
		 Under \textbf{I1}, we have
		\begin{align}
		& \textrm{Var}\bigg(\dfrac{1}{\sqrt{M}}\sum_{m=1}^{M}f(Y_m^{M}) \mathrel{\Big|}\mathcal{F}\bigg) \label{key} \\
		&\phantom{=}=\frac{1}{M}\sum\limits_{i=1}^{M} f^2( X_{i})\beta_{0}(u_{i-1}^M,w_i^M)  -\frac{1}{M}\sum_{k=1}^{(M-1)\wedge \left\lceil\frac{\bar{g}}{\underline{g}}\right\rceil}\sum\limits_{i=1}^{M-k}f( X_{i})f( X_{i+k})\beta_{1}(u_{i-1}^M,w_i^M,\sum\limits_{\ell=2}^{k}w_{i+\ell-1}^M,w_{i+k}^M) \label{key2}
		\end{align}
		where $\beta_{0}$ and $\beta_{1}$ are respectively defined in (\ref{beta0}) and (\ref{beta1}).
	\end{prop}
	In what follows, without specifying it every time,  we will suppose that $M\geq 1+\left\lceil\frac{\bar{g}}{\underline{g}}\right\rceil $ so that $(M-1)\wedge \left\lceil\frac{\bar{g}}{\underline{g}}\right\rceil =  \left\lceil\frac{\bar{g}}{\underline{g}}\right\rceil$.
	
	Let us observe that it is possible to rewrite  (\ref{key2}) in the following way
	\begin{align}
 \label{part2} \textrm{Var}\bigg(\dfrac{1}{\sqrt{M}}\sum_{m=1}^{M}f(Y_m^{M}) \mathrel{\Big|}\mathcal{F}\bigg) =\sum\limits_{k=0}^{\left\lceil\frac{\bar{g}}{\underline{g}}\right\rceil}\int_{0}^{1}F_{k,\alpha_1}^M d\alpha_1
	\end{align} 
	where 
	$$ F^M_{k,\alpha_1} = \begin{cases}
	f^2(X_{\lceil \alpha_1 M \rceil})\beta_{0}(u_{\lceil \alpha_1 M \rceil-1}^M,w_{\lceil \alpha_1 M \rceil}^M) &\,k=0\\
	-1_{\left\lbrace \lceil \alpha_1 M \rceil\leq M-k\right\rbrace}f( X_{\lceil \alpha_1 M \rceil})f( X_{\lceil \alpha_1 M \rceil+k})\beta_{1}(u_{\lceil \alpha_1 M \rceil-1}^M,w_{\lceil \alpha_1 M \rceil}^M,\sum\limits_{\ell=2}^{k}w_{\lceil \alpha_1 M \rceil+\ell-1}^M,w_{\lceil \alpha_1 M \rceil+k}^M) &\, k>0.
	\end{cases} $$ 
	
	Given $0< \alpha_1<1 $ and $k\in\mathbb{N}$, $ F^M_{k,\alpha_1}$  is a function of the vector 
	
	\begin{align*} 
	T_{\alpha_1,k}^M:=\left( f\left( X_{\lceil \alpha_1 M \rceil}\right) , f\left( X_{\lceil \alpha_1 M \rceil+k}\right) , u_{\lceil \alpha_1 M \rceil-1}^M,w_{\lceil \alpha_1 M \rceil}^M,\sum\limits_{\ell=2}^{k}w_{\lceil \alpha_1 M \rceil+\ell-1}^M,w_{\lceil \alpha_1 M \rceil+k}^M\right)
	\end{align*}
	which is well defined for $M$ big enough.
Therefore to study the limit behavior as $M\rightarrow \infty$ of the expectation of $ F^M_{k,\alpha_1}$, we are going to study the limit behavior of the vector $T_{\alpha_1,k}^M$. To obtain a much more general formulation useful in the next section, we introduce also the following family of vectors well defined for $M$ big enough.
Let $s\geq 1$ and $0< \alpha_1 < \cdots < \alpha_s< 1 $. For $k:=(k_1,\cdots,k_s)\in\mathbb{N}^s$, let us define $\forall i=1,\cdots,s$

\begin{equation} \label{tildeXvector}
T_{\alpha_i,k_i}^M:=\left( f\left( X_{\lceil \alpha_i M \rceil}\right) , f\left( X_{\lceil \alpha_i M \rceil+k_i}\right) , u_{\lceil \alpha_i M \rceil-1}^M,w_{\lceil \alpha_i M \rceil}^M,\sum\limits_{\ell=2}^{k_i}w_{\lceil \alpha_i M \rceil+\ell-1}^M,w_{\lceil \alpha_i M \rceil+k_i}^M\right).
\end{equation}

Moreover let 	
	
	\begin{equation} \label{tcl1}
	H^M:= \sqrt{M}\left(\dfrac{\sum\limits_{m=1}^{M}(fg)(X_m)}{\sum\limits_{m=1}^{M}g(X_m)}-\eta(fg)\right)
	\end{equation}
that thanks to (\ref{tcl})  converges in distribution to a centered Gaussian random variable with variance $\sigma_1^2(f).$
The purpose of the next proposition is therefore to study the convergence in distribution of  $\left(H^M,T_{\alpha_1,k_1}^M, \cdots , T_{\alpha_s,k_s}^M\right)$.

	\begin{prop}\label{conv_dist}
	Assume \textbf{I1-2}.
		Let $s\in \mathbb{N}^*$  and $k:=(k_1,\cdots,k_s)\in\mathbb{N}^s$. We set $K_i:=k_1+\cdots+k_{i-1}+i-1$ for $i\in \left\lbrace 2,\cdots,s \right\rbrace $ and $K_1:=0$. Let $\left(U_j \right)_{1\leq j\leq s} $ be a sequence of i.i.d random variables distributed according to the uniform law on $(0,1)$ and independent of $H\sim\mathcal{N}\left( 0,\sigma^2_1(f)\right)$ with $\left(U_1,\cdots,U_s,H \right)$ independent from $\left(X_j \right)_{1\leq j \leq K_{s+1}} $.
For $0< \alpha_1 < \cdots < \alpha_s< 1$, as $M$ goes to infinity, the following convergence in distribution holds
\begin{equation}\label{conv_dis_vector}
\left(H^M,T_{\alpha_1,k_1}^M, \cdots , T_{\alpha_s,k_s}^M\right) \overset{d}{\Longrightarrow} \left(H, T_{1,k},\cdots ,T_{s,k}\right) 
\end{equation}	
where
		\begin{equation} \label{Xvector}	 	
		T_{i,k}  := \left(f\left( X_{K_i+1}\right) ,f\left( X_ {K_{i+1}}\right) ,U_i,g(X_{K_i+1}),\sum\limits_{\ell=K_i+2}^{K_{i+1}-1}g(X_{\ell}),g(X_{K_{i+1}})\right)\quad \forall i=1,\cdots,s.
		\end{equation}
  \end{prop}

	\begin{proof}[Proof of Proposition \ref{conv_dist}]

	Let $M\geq \max\left(\frac{2}{\alpha_{1}}, \frac{k_s}{1-\alpha_{s}},\max\limits_{1\leq i\leq s-1}\frac{k_i+1}{\alpha_{i+1}-\alpha_{i}}\right)$ so that $2\leq \lceil\alpha_i M \rceil\leq M-k_i$  $\forall\,i=1,\cdots,s$ and $k_i < \lceil\alpha_{i+1} M \rceil -\lceil\alpha_{i} M \rceil $ $\forall\,i=1,\cdots,s-1$. The last condition allows us to separate the variables:  $\left(X_{\lceil \alpha_i M \rceil},\cdots,X_{\lceil \alpha_i M \rceil+k_i}\right)$  is independent of $\left(X_{\lceil \alpha_j M \rceil},\cdots,X_{\lceil \alpha_j M \rceil+k_j}\right)$ for $i,j$ distinct elements of $\left\lbrace1,\cdots,s \right\rbrace$. \\
		By defining $J_{\alpha_i,k_i}^M:= \left\lbrace \lceil \alpha_i M\rceil,\cdots,\lceil \alpha_i M\rceil+k_i \right\rbrace $ for $i=1,\cdots,s$, let $J^M:= \bigcup_{i=1}^sJ_{\alpha_i,k_i}^M$.\\ Moreover let $(\tilde{X}_{m})_{m\geq 1}$ be a sequence of i.i.d. $\mathbb{R}^d$-valued random vectors following the law $\eta$ independent of $(X_{m})_{1\leq m\leq M}$.  Given   $x:=(x_{j})_{1\leq j \leq K_{s+1}}\in\mathbb{R}^{d(K_{s+1})}$, we define

		$$ H_{x}^M := \sqrt{M}\left(\dfrac{    \sum\limits_{j=1}^{K_{s+1}}(fg)(x_{j})+ \sum\limits_{j\notin J_{s}^M}(fg)(X_j)}{\sum\limits_{j=1}^{K_{s+1}}g(x_{j})+\sum\limits_{j\notin J_{s}^M}g(X_j)}- \eta(fg)\right)$$
		and 
		$$ \tilde{H}_{x}^M:= \sqrt{M}\left(\dfrac{    \sum\limits_{j=1}^{K_{s+1}}(fg)(x_{j})+ \sum\limits_{j=1}^{M-K_{s+1}}(fg)(\tilde{X}_j)}{\sum\limits_{j=1}^{K_{s+1}}g(x_{j})+\sum\limits_{j=1}^{M-K_{s+1}}g(\tilde{X}_j)}- \eta(fg)\right).$$

		Similarly for $i=1,\cdots,s$  we define the vector $ T_{\alpha_i,k_i,x}^M$ by 

		\begin{align*}
		{\scriptstyle  \left(f\left( x_{K_i+1}\right) ,f\left( x_{K_{i+1}}\right) , \left\lbrace M\times\frac{\sum\limits_{\substack{j=1\\j\notin J_{s}^M}}^{\lceil\alpha_i M\rceil-1}g(X_j)+\sum\limits_{j=1}^{K_{i}}g(x_{j})}{\sum\limits_{j=1}^{K_{s+1}}g(x_{j})+\sum\limits_{j\notin J_{s}^M}g(X_j)} \right\rbrace,\frac{Mg(x_{K_i+1})}{\sum\limits_{j=1}^{K_{s+1}}g(x_{j})+\sum\limits_{j\notin J_{s}^M}g(X_j)},\frac{M\sum\limits_{\ell=K_i+2}^{K_{i+1}-1}g(x_{\ell })}{\sum\limits_{j=1}^{K_{s+1}}g(x_{j})+\sum\limits_{j\notin J_{s}^M}g(X_j)},\frac{Mg(x_{K_{i+1}})}{\sum\limits_{j=1}^{K_{s+1}}g(x_{j})+\sum\limits_{j\notin J_{s}^M}g(X_j)}
			\right) }
		\end{align*}
		and the vector $\tilde T_{\alpha_i,k_i,x}^M$ by 
		\begin{align*}
		{\scriptstyle  \left(f\left( x_{K_i+1}\right) ,f\left( x_{K_{i+1}}\right) , \left\lbrace M\times \frac{\sum\limits_{j=1}^{\lceil\alpha_i M\rceil-K_i-1}g(\tilde{X}_j)+\sum\limits_{j=1}^{K_{i}}g(x_{j})}{\sum\limits_{j=1}^{K_{s+1}}g(x_{j})+\sum\limits_{j=1}^{M-K_{s+1}}g(\tilde{X}_j)} \right\rbrace,\frac{Mg(x_{K_i+1})}{\sum\limits_{j=1}^{K_{s+1}}g(x_{j})+\sum\limits_{j=1}^{M-K_{s+1}}g(\tilde{X}_j)},\frac{M\sum\limits_{\ell=K_i+2}^{K_{i+1}-1}g(x_{\ell })}{\sum\limits_{j=1}^{K_{s+1}}g(x_{j})+\sum\limits_{j=1}^{M-K_{s+1}}g(\tilde{X}_j)},\frac{Mg(x_{K_{i+1}})}{\sum\limits_{j=1}^{K_{s+1}}g(x_{j})+\sum\limits_{j=1}^{M-K_{s+1}}g(\tilde{X}_j)}
			\right) }.
		\end{align*}
		
		We notice that 
		\begin{align}
			\left(H_{x}^M , T_{\alpha_1,k_1,x}^M,\cdots,T_{\alpha_s,k_s,x}^M\right) \overset{\mathcal{L}}{=} \left(\tilde{H}_{x}^M , \tilde{T}_{\alpha_1,k_1,x}^M,\cdots,\tilde{T}_{\alpha_s,k_s,x}^M\right).\label{eqinlegge}
		\end{align}

		Let us observe that by hypothesis  \textbf{I2}  the law of $ g\left(\tilde{X}_i \right)$ has a component absolutely continuous. Moreover since $\int_{\mathbb{R}}\left| 1+y\right|^{s+1}e^{-y^2}dy < \infty$ $\forall s\geq 1$, we can apply Theorem 1 of \cite{CLTfract} with $q=s$ and 	
		
		\begin{itemize}
			\item $\left(Y_i,Z_i \right)$ equal to $\left(g\left(\tilde{X}_i \right),(fg)(\tilde{X}_i)\right)  $ 
			\item $\left( \beta_M^1,\cdots,\beta_M^s,\beta_M^{s+1}\right) $ equal to $\left(\lceil\alpha_1 M\rceil-K_1-1,\cdots,\lceil\alpha_s M\rceil-K_s-1,M-K_{s+1} \right) $ 
			\item $\phi(x)=\frac{1}{x}$

			\item $x$ equal to $\sum\limits_{j=1}^{K_{s+1}}g(x_{j})$, $z$ equal to $\sum\limits_{j=1}^{K_{s+1}}(fg)(x_{j})$, $(y_1,\cdots,y_s)$ equal to $\left(\sum\limits_{j=1}^{K_{1}}g(x_{j}), \cdots,\sum\limits_{j=1}^{K_{s}}g(x_{j}) \right) $ and $\theta$ equal to $\eta(fg)$
			\item $T:=H \sim \mathcal{N}\left( 0,\sigma^2_1(f)\right)$ by (\ref{tcl}) 
		\end{itemize} 
		and deduce the following convergence in distribution as $M\rightarrow \infty$
	
		\begin{equation}\label{conv_corollary}
			\left(\tilde H_{x}^M ,\left(  \left\lbrace M\times \tfrac{\sum\limits_{j=1}^{\lceil\alpha_i M\rceil-K_i-1}g(\tilde{X}_j)+\sum\limits_{j=1}^{K_{i}}g(x_{j})}{\sum\limits_{j=1}^{K_{s+1}}g(x_{j})+\sum\limits_{j=1}^{M-K_{s+1}}g(\tilde{X}_j)} \right\rbrace\right)_{1\leq i \leq s} \right)  \overset{d}{\Longrightarrow} \left(H,U_1,\cdots,U_s \right) .
		\end{equation}

		Moreover by the Strong Law of Large Numbers \begin{align*}
			\lim\limits_{M\rightarrow \infty}\tfrac{Mg(x_{K_i+1})}{\sum\limits_{j=1}^{K_{s+1}}g(x_{j})+\sum\limits_{j=1}^{M-K_{s+1}}g(\tilde{X}_j)} =g(x_{K_i+1}),\quad \lim\limits_{M\rightarrow \infty}\frac{M\sum\limits_{\ell=K_i+2}^{K_{i+1}-1}g(x_{\ell })}{\sum\limits_{j=1}^{K_{s+1}}g(x_{j})+\sum\limits_{j=1}^{M-K_{s+1}}g(\tilde{X}_j)}= \sum\limits_{\ell=K_i+2}^{K_{i+1}-1}g(x_{\ell })
		\end{align*} and 
		\begin{align*}\lim\limits_{M\rightarrow \infty}\frac{Mg(x_{K_{i+1}})}{\sum\limits_{j=1}^{K_{s+1}}g(x_{j})+\sum\limits_{j=1}^{M-K_{s+1}}g(\tilde{X}_j)}=g(x_{K_{i+1}}).
		\end{align*}

		 Therefore by Slutsky's theorem, for each  $x:=(x_{j})_{1\leq j\leq K_{s+1}}\in\mathbb{R}^{d(K_{s+1})}$ we can conclude that the following random vector 
		
		\begin{align*}
	\left(\tilde H_{x}^M,\tilde T_{\alpha_1,k_1,x}^M,\cdots,\tilde T_{\alpha_s,k_s,x}^M
		\right)
		\end{align*}
		converges in distribution as $M$ goes to infinity to
		\begin{align}\label{conv_slutsky}
\left(H,  T_{1,k,x},\cdots , T_{s,k,x}\right)
		\end{align}	 
		 where the vector $T_{i,k,x}$ is given by
		 \begin{align*}
		 {\scriptstyle  \left(f\left( x_{K_i+1}\right) ,f\left( x_{K_{i+1}}\right) , U_i,g(x_{K_i+1}),\sum\limits_{\ell=K_i+2}^{K_{i+1}-1}g(x_{\ell }),g(x_{K_{i+1}})
		 	\right) }.
		 \end{align*}

		In particular, given $\phi:\mathbb{R}^{1+6s} \rightarrow \mathbb{R}$ a continuous bounded function one has
		\begin{equation}\label{conv_of_expectation}
		\lim_{M\rightarrow \infty}\mathbb{E}\left(\phi\left(	\tilde H_{x}^M,\tilde T_{\alpha_1,k_1,x}^M,\cdots,\tilde T_{\alpha_s,k_s,x}^M
		\right)\right) = \mathbb{E}\left(\phi\left(H,  T_{1,k,x},\cdots , T_{s,k,x}\right)\right).
		\end{equation}
		We are now ready to prove (\ref{conv_dis_vector}). By conditioning with respect to $ \sigma\left(\left( X_i\right)_{i\in J_s^M}\right)$ and by applying the Freezing Lemma, (\ref{eqinlegge}) and (\ref{conv_of_expectation}), one has
		\begin{align*}
		\mathbb{E}&\left(\phi\left( {H}^M,{T}_{\alpha_1,k_1}^M, \cdots , {T}_{\alpha_s,k_s}^M  \right)  \right)\\
		&\phantom{===}= \int_{\mathbb{R}^{d(k_1+\cdots +k_s+s)}}\mathbb{E}\left(\phi\left( H_{x}^M, T_{\alpha_1,k_1,x}^M,\cdots, T_{\alpha_s,k_s,x}^M\right)\right)  \prod_{i=1}^{K_{s+1}}\eta(dx_{i})\\
		&\phantom{===}= \int_{\mathbb{R}^{d(k_1+\cdots +k_s+s)}}\mathbb{E}\left(\phi\left( \tilde H_{x}^M, \tilde T_{\alpha_1,k_1,x}^M,\cdots, \tilde T_{\alpha_s,k_s,x}^M\right)\right)  \prod_{i=1}^{K_{s+1}}\eta(dx_{i})\underset{M\rightarrow \infty}{\longrightarrow} \mathbb{E}\left(\phi\left(H, {T}_{1,k},\cdots ,{T}_{s,k} \right)\right)
		\end{align*}

		where ${T}_{i,k}, \, i=1,\cdots,s$ is defined in (\ref{Xvector}) and so the proof is complete.
	\end{proof}

	We are now ready to prove Theorem \ref{mainresult} .

	\begin{proof}[Proof of Theorem \ref{mainresult}]
		By (\ref{part2}), it is enough to study the convergence as $M$ goes to infinity of 
		
		\begin{equation} \label{step1}
			\sum\limits_{k=0}^{\left\lceil\frac{\bar{g}}{\underline{g}}\right\rceil} \int_{0}^{1}\mathbb{E}\left(F_{k,\alpha_1}^M \right) d\alpha_1.
		\end{equation} 
		 
		Let us observe that we can replace $\beta_{0}$ and $\beta_{1}$ (introduced in (\ref{beta0}) and (\ref{beta1})) in the definition of $F_{k,\alpha_1}^M$ respectively by the bounded functions $\beta_0\left(0\vee x\wedge1,0\vee y_1 \right) $ and $\beta_1\left(0\vee x\wedge1,0\vee y_1 ,0\vee y_2,0\vee y_3\right) $ that to simplify the notation we will keep calling  $ \beta_{0}$ and $\beta_{1}$.
		
For what has been said there exists a constant $\tilde{C} < \infty$ such that 
		
		\begin{equation}\label{limitato}
		\sup_{(x,y_1)\in \mathbb{R}^{2}}\left|\beta_{0}(x,y_1) \right|\vee\sup_{(x,y_1,y_{2},y_{3})\in \mathbb{R}^{4}}\left|\beta_{1}(x,y_1,y_{2},y_{3}) \right|	 \leq \tilde{C}.
		\end{equation}
		
		
		By Proposition \ref{conv_dist} applied with $s=1$ and by 
		Remark \ref{continuitybeta},  we have that for each $\alpha_1 \in (0,1)$
		\begin{align*}
			F_{0,\alpha_1}^M \overset{d}{\Longrightarrow} F_{0}
		\end{align*}
		with $F_{0}$ defined in (\ref{defpk}). Moreover by the hypothesis of boundedness  of $f$  and (\ref{limitato}) we have
		\begin{align*}
		\lim_{M\rightarrow \infty}&  \mathbb{E}\left(F_{0,\alpha_1}^M \right)=\mathbb{E}\left( F_{0} \right).
		\end{align*}
		
	By Lebesgue's theorem, $\int_{0}^{1}\mathbb{E}\left(F_{0,\alpha_1}^M \right) d\alpha_1$  will converge to the same limit.\\
		Let us now study the convergence of $\sum\limits_{k=1}^{\left\lceil\frac{\bar{g}}{\underline{g}}\right\rceil} \int_{0}^{1}\mathbb{E}\left(F_{k,\alpha_1}^M \right) d\alpha_1$.	
	By Proposition \ref{conv_dist} applied with $s=1$ and by Remark \ref{continuitybeta}, we have that for each $\alpha_1 \in (0,1)$ and $k\geq 1$  
		\begin{align*}
	F_{k,\alpha_1}^M \overset{d}{\Longrightarrow}F_{k}
		\end{align*}
		with $F_{k}$ defined in (\ref{defpk}). Moreover by the hypothesis of boundedness of $f$ and (\ref{limitato}),\\
				\begin{align*}
		\lim_{M\rightarrow \infty}&\mathbb{E}\left(F_{k,\alpha_1}^M\right)=\mathbb{E}\left(F_{k}\right) .
		\end{align*}
		By Lebesgue's theorem, $\lim_{M\rightarrow \infty}\int_{0}^{1}\mathbb{E}\left( F_{k,\alpha_1}^M\right) d\alpha_1 = \mathbb{E}\left( F_{k}\right) $ for each $k= 1,\cdots,\left\lceil\frac{\bar{g}}{\underline{g}}\right\rceil$ and so the proof is complete.
	\end{proof}
We are now ready to provide a central limit theorem for $\frac{1}{M}\sum\limits_{m=1}^{M}f(Y_m^{M})$ knowing its asymptotic variance.
	\section{Central Limit Theorem}\label{secclt}

	\begin{teo}\label{mainresulttcl}
		 Let assume \textbf{I1-2}. The following convergence in distribution holds
		\begin{equation}\label{tclfin}
			\sqrt{M}\left( \frac{1}{M}\sum_{m=1}^{M}f(Y_m^{M})-\eta(fg) \right)\overset{d}{\Longrightarrow} \mathcal{N}\left( 0,\sigma^2_1(f)+\sigma^2_2(f)\right)
		\end{equation}
		where $\sigma^2_1(f)$ and $\sigma^2_2(f)$ are respectively defined in (\ref{sigma1})  and (\ref{sigma2}). 
	\end{teo}

\begin{proof}[Proof of Theorem \ref{mainresulttcl}]
Let $u\in\mathbb{R}$. By introducing the notation
\begin{equation} \label{am}
a_m^M := e^{\frac{iu}{\sqrt{M}}\left(f(Y_m^{M})-\mathbb{E}\left(f(Y_m^{M})\mid \mathcal{F} \right) \right)}\,\,for\,\,m=1,\cdots,M
\end{equation}
and by applying (\ref{selection}), we have 
\begin{align} 
	\mathbb{E}\left( e^{\frac{iu}{\sqrt{M}}\sum\limits_{m=1}^{M}\left( f(Y_m^{M})-\eta(fg) \right) } \right) &= \mathbb{E}\left(e^{iu\sqrt{M}\left( \mathbb{E}\left( \frac{1}{M}\sum\limits_{m=1}^{M}f(Y_m^{M})\mathrel{\Big|}\mathcal{F}\right) -\eta(fg)\right) } \mathbb{E}\left(\prod\limits_{m=1}^{M}a_m^M\mathrel{\Big|}\mathcal{F}\right)\right)\nonumber \\
	&= \mathbb{E}\left( e^{iu{H}^M}\mathbb{E}\left(\prod\limits_{m=1}^{M}a_m^M\mathrel{\Big|}\mathcal{F}\right)\right)  \label{char.function}
\end{align}
where ${H}^M$ is defined in (\ref{tcl1}).
The purpose of what follows is therefore to study the convergence as $M$ goes to infinity of  (\ref{char.function}).\\

\textbf{First step}\\
We are first going to study the conditional expectation $\mathbb{E}\left(\prod\limits_{m=1}^{M}a_m^M\mathrel{\Big|}\mathcal{F}\right)$.\\
 By using that for each $u,y \in \mathbb{R}$ $\left|e^{iuy}-1-iuy+\frac{(uy)^2}{2} 	 \right|\leq \frac{\left| uy\right|^3 }{6},$  one has 

\begin{align*}
&\left|a_m^M-1-\frac{iu}{\sqrt{M}}\left(f(Y_m^{M})-\mathbb{E}\left(f(Y_m^{M})\mid \mathcal{F} \right) \right)+\frac{u^2 \left(f(Y_m^{M})-\mathbb{E}\left(f(Y_m^{M})\mid \mathcal{F} \right) \right)^2 }{2M}\right|\\
&\phantom{===}\leq \frac{\left|u\right|^3\left|f(Y_m^{M})-\mathbb{E}\left(f(Y_m^{M})\mid \mathcal{F} \right) \right| ^3 }{6M^{\frac{3}{2}}}\leq  \frac{4\left\| f\right\| _\infty^3\left|u\right|^3 }{3M^{\frac{3}{2}}}.
\end{align*}

Therefore for $m=1,\cdots,M$, $a_m^M$ can be rewritten as
\begin{align*}
a_m^M = 1+\frac{iu}{\sqrt{M}}\left(f(Y_m^{M})-\mathbb{E}\left(f(Y_m^{M})\mid \mathcal{F} \right) \right)-\frac{u^2 \left(f(Y_m^{M})-\mathbb{E}\left(f(Y_m^{M})\mid \mathcal{F} \right) \right)^2 }{2M}+r_m^M 
\end{align*}
with $\left| r_m^M\right| \leq  \frac{4\left\| f\right\| _\infty^3\left|u\right|^3 }{3M^{\frac{3}{2}}}.$\\
 By defining for $m=1,\cdots,M$
 \begin{equation}\label{bm}
 	b_m^M:= 1-\frac{u^2 \left(f(Y_m^{M})-\mathbb{E}\left(f(Y_m^{M})\mid \mathcal{F} \right) \right)^2 }{2M},
 \end{equation} let us now study the difference
\begin{align}\label{difference}
&\left|\mathbb{E}\left(\prod_{m=1}^{M}a_m^M\mathrel{\Big|} \mathcal{F} \right) -\mathbb{E}\left(\prod_{m=1}^{M}b_m^M\mathrel{\Big|} \mathcal{F} \right) \right|.
\end{align} 
By rewriting $\prod\limits_{m=1}^{M}a_m^M - \prod\limits_{m=1}^{M}b_m^M = \sum\limits_{m=1}^{M} \prod\limits_{\ell=1}^{m-1}a_\ell^M\left(a_m^M-b_m^M \right) \prod\limits_{\ell=m+1}^{M}b_\ell^M$, where by convention the empty product is equal to one,  by using that the random vectors $(Y_m^{M})_{1\leq m \leq M}$ are conditionally independent given $\mathcal{F}$ and  that $\left| r_m^M\right| \leq  \frac{4\left\| f\right\| _\infty^3\left|u\right|^3 }{3M^{\frac{3}{2}}}$, $\left|a_\ell^M \right| \leq 1$ and $\left|b_\ell^M \right| \leq 1+\frac{2\left\|f \right\|_\infty^2 u^2}{M}$, (\ref{difference}) becomes

 \begin{align*}
&\left|\mathbb{E}\left(\prod_{m=1}^{M}a_m^M\mathrel{\Big|} \mathcal{F} \right) -\mathbb{E}\left(\prod_{m=1}^{M}b_m^M\mathrel{\Big|} \mathcal{F} \right) \right|\\ &\phantom{=}=\left|\sum\limits_{m=1}^{M}\mathbb{E}\left( \prod\limits_{\ell=1}^{m-1}a_\ell^M\left(\frac{iu}{\sqrt{M}}\left(f(Y_m^{M})-\mathbb{E}\left(f(Y_m^{M})\mid \mathcal{F} \right) \right)+r_m^M \right) \prod\limits_{\ell=m+1}^{M}b_\ell^M \mathrel{\Big|} \mathcal{F} \right)\right|\\
 &\phantom{=}=\left|\sum\limits_{m=1}^{M}\mathbb{E}\left( \prod\limits_{\ell=1}^{m-1}a_\ell^M\mathrel{\Big|} \mathcal{F} \right)\mathbb{E}\left(\frac{iu}{\sqrt{M}}\left(f(Y_m^{M})-\mathbb{E}\left(f(Y_m^{M})\mid \mathcal{F} \right) \right)+r_m^M \mathrel{\Big|} \mathcal{F} \right)\mathbb{E}\left( \prod\limits_{\ell=m+1}^{M}b_\ell^M \mathrel{\Big|} \mathcal{F} \right)\right|\\
 &\phantom{=}=\left|\sum\limits_{m=1}^{M}\mathbb{E}\left( \prod\limits_{\ell=1}^{m-1}a_\ell^M\mathrel{\Big|} \mathcal{F} \right)\mathbb{E}\left(r_m^M \mathrel{\Big|} \mathcal{F} \right)\mathbb{E}\left( \prod\limits_{\ell=m+1}^{M}b_\ell^M \mathrel{\Big|} \mathcal{F} \right)\right| \leq M\times \frac{4\left\| f\right\| _\infty^3\left|u\right|^3 }{3M^{\frac{3}{2}}} \times \left( 1+\frac{2\left\|f \right\|_\infty^2 u^2}{M}\right)^M \underset{M \rightarrow \infty}{\longrightarrow} 0.
 \end{align*}

This implies that 
	\begin{align}
		\limsup_{M\rightarrow \infty}\left|  \mathbb{E}\left( e^{iu{H}^M}\mathbb{E}\left(\prod\limits_{m=1}^{M}a_m^M\mathrel{\Big|}\mathcal{F}\right)\right) - \mathbb{E}\left( e^{iu{H}^M}\mathbb{E}\left(\prod\limits_{m=1}^{M}b_m^M\mathrel{\Big|}\mathcal{F}\right)\right) \right| =0.\label{eqaeb}
	\end{align} 
	\textbf{Second step}\\
	
	By the previous step, it is enough to study the limit behaviour as $M\rightarrow \infty$ of $\mathbb{E}\left( e^{iuH^M}\mathbb{E}\left(\prod\limits_{m=1}^{M}b_m^M\mathrel{\Big|}\mathcal{F}\right)\right)$ where $b_m^M$ is defined in (\ref{bm}) . \\
	By using the general fact that $\prod\limits_{m=1}^{M} (1-c_m)= \sum\limits_{s=0}^{M}(-1)^s\sum\limits_{1\leq m_1<\cdots <m_s\leq M}\prod\limits_{j=1}^{s}c_{m_j}$ for any real sequence $\left( c_m\right)_{1\leq m \leq M} $, we have 

	\begin{align}
		\mathbb{E}\left(\prod\limits_{m=1}^{M}b_m^M\mathrel{\Big|}\mathcal{F}\right)&=\sum\limits_{s=0}^{M}\left( \frac{-u^2}{2M}\right) ^s\sum\limits_{1\leq m_1<\cdots <m_s\leq M} \mathbb{E}\left(\prod\limits_{j=1}^{s} \left(f(Y_{m_j}^{M})-\mathbb{E}\left(f(Y_{m_j}^{M})\mid \mathcal{F} \right) \right)^2 \mathrel{\Big|}\mathcal{F}\right)\nonumber\\
		&=\sum\limits_{s=0}^{M}\left( \frac{-u^2}{2M}\right) ^s\sum\limits_{1\leq m_1<\cdots <m_s\leq M} \prod\limits_{j=1}^{s}\mathbb{E}\left( \left(f(Y_{m_j}^{M})-\mathbb{E}\left(f(Y_{m_j}^{M})\mid \mathcal{F} \right) \right)^2 \mathrel{\Big|}\mathcal{F}\right)\label{distinct}
	\end{align}
	where for the last equality we use that the random vectors $(Y_m^{M})_{1\leq m \leq M}$ are conditionally independent given $\mathcal{F}$. 
	Let us introduce $B_{m_1,\cdots,m_s}:=\prod\limits_{j=1}^{s}\mathbb{E}\left( \left(f(Y_{m_j}^{M})-\mathbb{E}\left(f(Y_{m_j}^{M})\mid \mathcal{F} \right) \right)^2 \mathrel{\Big|}\mathcal{F}\right)$. Since $B_{m_1,\cdots,m_s}$ is symmetric in its indexes and satisfies $\left| B_{m_1,\cdots,m_s}\right|\leq \left( 4\left\| f\right\|_\infty ^{2}\right) ^s $,  for each $s\geq 1$ we have
	\begin{align*}
		\left| \frac{1}{M^s}\left(  \frac{1}{s!}\sum\limits_{ m_1,\cdots ,m_s= 1}^{M} B_{m_1,\cdots,m_s}-\sum\limits_{1\leq m_1<\cdots <m_s\leq M} B_{m_1,\cdots,m_s}\right)\right| &\leq \frac{\left( 4\left\| f\right\|_\infty ^{2}\right) ^s}{M^s}\left( \frac{M^s}{s!}-\binom{M}{s}\right).
	\end{align*}
	It is therefore possible to rewrite (\ref{distinct}) as  
	\begin{align*}
		\mathbb{E}\left(\prod\limits_{m=1}^{M}b_m^M\mathrel{\Big|}\mathcal{F}\right)&=\sum\limits_{s=0}^{M}\left(- \frac{u^2}{2M}\right) ^s\sum\limits_{1\leq m_1<\cdots <m_s\leq M} \prod\limits_{j=1}^{s}\mathbb{E}\left( \left(f(Y_{m_j}^{M})-\mathbb{E}\left(f(Y_{m_j}^{M})\mid \mathcal{F} \right) \right)^2 \mathrel{\Big|}\mathcal{F}\right)\\
		&= \sum\limits_{s=0}^{M}\left( -\frac{u^2}{2M}\right) ^s \frac{1}{s!}\sum\limits_{ m_1,\cdots ,m_s= 1}^{M}\prod\limits_{j=1}^{s}\mathbb{E}\left( \left(f(Y_{m_j}^{M})-\mathbb{E}\left(f(Y_{m_j}^{M})\mid \mathcal{F} \right) \right)^2 \mathrel{\Big|}\mathcal{F}\right)+ \bar{r}^M\\
		&=\sum\limits_{s=0}^{M}\left( -\frac{u^2}{2}\right) ^s \frac{1}{s!}\left(\frac{1}{M} \sum\limits_{ m= 1}^{M}\mathbb{E}\left( \left(f(Y_{m}^{M})-\mathbb{E}\left(f(Y_{m}^{M})\mid \mathcal{F} \right) \right)^2 \mathrel{\Big|}\mathcal{F}\right)\right) ^s+ \bar{r}^M\\
		&= \sum\limits_{s=0}^{M}\left( -\frac{u^2}{2}\right) ^s \frac{1}{s!}\left(  \textrm{Var}\left( \frac{1}{\sqrt{M}}\sum\limits_{m=1}^{M}f(Y_{m}^{M})\mathrel{\Big|}\mathcal{F}\right)\right)^s+ \bar{r}^M
	\end{align*}
	with $\bar{r}^M$ such that 
	\begin{align*}
		\left|\bar{r}^M \right| \leq \sum_{s\geq1} \frac{\left( 2u^2\left\| f\right\|_\infty^2 \right)^s}{s!}  \left(1-\frac{M(M-1)\cdots (M-s+1)}{M^s} \right) \underset{M\rightarrow \infty}{\longrightarrow}0.
	\end{align*}
	 With $\left|e^{iu{H}^M}\right|\leq 1$, (\ref{char.function}) and (\ref{eqaeb}), this implies that
	\begin{align*}
	&\left|  \mathbb{E}\left( e^{\frac{iu}{\sqrt{M}}\sum\limits_{m=1}^{M}\left( f(Y_m^{M})-\eta(fg) \right) } \right) - \mathbb{E}\left( e^{iu{H}^M}\sum\limits_{s=0}^{M}\left( -\frac{u^2}{2}\right) ^s \frac{1}{s!}\left(  \textrm{Var}\left( \frac{1}{\sqrt{M}}\sum\limits_{m=1}^{M}f(Y_{m}^{M})\mathrel{\Big|}\mathcal{F}\right)\right)   ^s\right) \right| \underset{M\rightarrow \infty}{\longrightarrow} 0.
	\end{align*}
	
\subsubsection*{Third step}
	The purpose of what follows is therefore to study the convergence as $M$ goes to infinity of 
	
	\begin{equation}\label{newpurpose}
		\mathbb{E}\left( e^{iu{H}^M}\sum\limits_{s=0}^{M}\left( -\frac{u^2}{2}\right) ^s \frac{1}{s!}\left(  \textrm{Var}\left( \frac{1}{\sqrt{M}}\sum\limits_{m=1}^{M}f(Y_{m}^{M})\mathrel{\Big|}\mathcal{F}\right)\right)^s\right).
	\end{equation}

	By (\ref{part2}), (\ref{newpurpose}) can be rewritten as
	
	 \begin{align}
	 	\mathbb{E}&\left(e^{iu{H}^M}\right) +\sum\limits_{s=1}^{M}\left( -\frac{u^2}{2}\right)^s \frac{1}{s!}\sum\limits_{ k_1,\cdots,k_s= 0}^{\left\lceil\frac{\bar{g}}{\underline{g}}\right\rceil} \int_{\left[0,1 \right]^s }\mathbb{E}\left(e^{iu{H}^M}F_{k_1,\alpha_1}^M\cdots F_{k_s,\alpha_s}^M \right) d\alpha_1\cdots d\alpha_s \nonumber\\
	 	&=\mathbb{E}\left(e^{iu{H}^M}\right) + \sum\limits_{s\geq 1} 1_{\left\lbrace s\leq M \right\rbrace}z(s,M) \label{conv}
	 \end{align}
	 where for $s\geq 1$ and $M\geq 1$
	
	\begin{align*}
	z(s,M)=	\left( -\frac{u^2}{2}\right)^s \sum\limits_{ k_1,\cdots,k_s= 0}^{\left\lceil\frac{\bar{g}}{\underline{g}}\right\rceil} \int_{\left[0,1 \right]^s }1_{\left\lbrace 0<\alpha_1< \cdots <\alpha_s < 1\right\rbrace }\mathbb{E}\left(e^{iu{H}^M}F^M_{k_1,\alpha_1}\cdots F^M_{k_s,\alpha_s} \right) d\alpha_1\cdots d\alpha_s
	\end{align*} 
and where we recall that for  $i=1, \cdots, s$ 
 
 $$ F^M_{k_i,\alpha_i} = \begin{cases}
 f^2(X_{\lceil \alpha_i M \rceil})\beta_{0}(u_{\lceil \alpha_i M \rceil-1}^M,w_{\lceil \alpha_i M \rceil}^M) &\,k_i=0\\
 -1_{\left\lbrace \lceil \alpha_i M \rceil\leq M-k_i\right\rbrace}f( X_{\lceil \alpha_i M \rceil})f( X_{\lceil \alpha_i M \rceil+k_i})\beta_{1}(u_{\lceil \alpha_i M \rceil-1}^M,w_{\lceil \alpha_i M \rceil}^M,\sum\limits_{\ell=2}^{k_i}w_{\lceil \alpha_i M \rceil+\ell-1}^M,w_{\lceil \alpha_i M \rceil+k_i}^M) &\, k_i>0.
 \end{cases} $$ 
 Let us now study the convergence as $M\rightarrow \infty$ of (\ref{conv}): $\lim\limits_{M\rightarrow \infty}\mathbb{E}\left(e^{iu{H}^M}\right)=e^{-\frac{u^2\sigma_1^2(f)}{2}}$ by (\ref{tcl})  and to study the convergence of the sum we will apply Lebesgue's theorem.\\
Given $s\geq 1$, we first compute the pointwise convergence as $M \rightarrow \infty$ of $z(s,M)$ and then we show the existence of a function $\tilde{z}$ independent of $M$ that dominates $\left| z\right| $  and such that  $\sum\limits_{ s\geq 1}\tilde{z}(s)< \infty$.
 
Let us start by computing the pointwise convergence. Let $s\geq 1$ and $ 0\leq  k_1,k_2,\cdots,k_s\leq \left\lceil\frac{\bar{g}}{\underline{g}}\right\rceil $.\\
By applying the reasoning based on Proposition \ref{conv_dist} done in the proof of Theorem \ref{mainresult}, $\forall\,\, 0<\alpha_{1}<\cdots < \alpha_{s}<1$ we check that
\begin{align*}
	e^{iu{H}^M}F^M_{k_1,\alpha_1}\cdots F^M_{k_s,\alpha_s}\overset{d}{\Longrightarrow}e^{iuH} F_{1,k}\cdots F_{s,k}
\end{align*}
where, recalling that  $k$ denotes the multiindex $\left(k_1,\cdots,k_s \right)$ and $K_i:=k_1+\cdots+k_{i-1}+i-1$ $\forall i=1,\cdots,s$ with the convention $K_1:=0$, $F_{i,k} $ for $i=1,\cdots,s$ is given by
\begin{align*} 
F_{i,k} = \begin{cases}
f^2(X_{K_i+1})\beta_{0}\left( U_i,g(X_{K_i+1})\right) &\,k_i=0\\
-f( X_{K_i+1})f( X_ {K_{i+1}})\beta_{1}\left( U_i,g(X_{K_i+1}),\sum\limits_{\ell=K_i+2}^{K_{i+1}-1}g(X_{\ell}),g(X_{K_{i+1}})\right) &\, k_i>0
\end{cases}
\end{align*} 

and $H \sim \mathcal{N}\left(0,\sigma_1^2(f) \right)$ independent of $\left( U_j\right)_{1\leq j \leq s}$ and  $\left( X_{j}\right)_{1\leq j\leq K_{s+1}}$. \\
Moreover by the hypothesis of boundedness  of $f$, (\ref{limitato}) and the fact that $\left| e^{iu{H}^M}\right| \leq 1$,
\begin{align*}
\lim_{M\rightarrow \infty}& \mathbb{E}\left(e^{iu{H}^M}F^M_{k_1,\alpha_1}\cdots F^M_{k_s,\alpha_s} \right)  = \mathbb{E}\left(e^{iuH} F_{1,k}\cdots F_{s,k}\right) .
\end{align*}

By Lebesgue's theorem and by observing that $F_{1,k},\cdots,F_{s,k}$ are independent, respectively distributed as $F_{k_1},\cdots,F_{k_s}$ and independent of $H$, $z(s,M)$ converges to 
 
\begin{align*} 
& \left( -\frac{u^2}{2}\right)^s \sum\limits_{ k_1,\cdots,k_s= 0}^{\left\lceil\frac{\bar{g}}{\underline{g}}\right\rceil} \int_{\left[0,1 \right]^s }1_{\left\lbrace 0<\alpha_1< \cdots <\alpha_s < 1\right\rbrace }\mathbb{E}\left(e^{iuH} F_{1,k}\cdots F_{s,k}\right) d\alpha_1\cdots d\alpha_s \\
 &\phantom{====}= e^{-\frac{u^2\sigma_1^2(f)}{2}} \frac{1}{s! }\left( -\frac{u^2}{2}\right)^s\sum\limits_{ k_1,\cdots,k_s= 0}^{\left\lceil\frac{\bar{g}}{\underline{g}}\right\rceil} \mathbb{E}\left(F_{k_1}\right)\cdots \mathbb{E}\left(F_{k_s}\right)\\
 &\phantom{====}=  e^{-\frac{u^2\sigma_1^2(f)}{2}} \frac{1}{s!}\left( -\frac{u^2}{2}\right)^s \left( \sum_{k_1=0}^{\left\lceil\frac{\bar{g}}{\underline{g}}\right\rceil} \mathbb{E}\left(F_{k_1}\right)\right)^s = e^{-\frac{u^2\sigma_1^2(f)}{2}} \frac{1}{s! }\left( -\frac{u^2}{2}\right)^s \left( \sigma_2^2(f)\right)^s.
\end{align*} 
 
Let us now show that $z$ is dominated by a summable function $\tilde{z}$. By using that $f$, $\beta_{0}$ and $\beta_{1}$ are bounded functions, for $s\geq1$ one has 
 
 \begin{align*}
 	\left|z(s,M)\right| &= \left|\left( -\frac{u^2}{2}\right)^s \sum\limits_{ k_1,\cdots,k_s= 0}^{\left\lceil\frac{\bar{g}}{\underline{g}}\right\rceil} \int_{\left[0,1 \right]^s }1_{\left\lbrace 0<\alpha_1< \cdots <\alpha_s < 1\right\rbrace }\mathbb{E}\left(e^{iuH} F_{1,k}\cdots F_{s,k}\right) d\alpha_1\cdots d\alpha_s \right| \\
 	&\leq C\frac{1}{s!} \left(\frac{u^2\left( \left\lceil\frac{\bar{g}}{\underline{g}}\right\rceil+1\right) }{2}\right)^s
 \end{align*}
 for a given finite constant $C$.
 It is now sufficient to observe that $\sum\limits_{ s\geq 1}\frac{1}{s!}\left(\left( \left\lceil\frac{\bar{g}}{\underline{g}}\right\rceil+1\right)  \times\frac{u^2}{2}\right)^s = e^{\left( \left\lceil\frac{\bar{g}}{\underline{g}}\right\rceil +1\right) \frac{u^2}{2}}<\infty.$
 In conclusion we have proved that (\ref{conv}) converges as $M\rightarrow \infty$ to 
\begin{align*}
	e^{-\frac{u^2\sigma_1^2(f)}{2}}\sum_{s\geq 0} \frac{1}{s! }\left( -\frac{u^2}{2}\right)^s \left( \sigma_2^2(f)\right)^s=e^{-\frac{u^2\left( \sigma_1^2(f)+\sigma_2^2(f)\right) }{2}}.
\end{align*}

\end{proof}

\section{Proof of Proposition \ref{rewriting_term} }\label{sectionfinale}
We recall that for $i=1,\cdots,M$ we denote $u_{i}^M:= \left\lbrace w_1^M+\cdots+w_{i}^M \right\rbrace$ and $\mu_{i}^M:= \lfloor w_1^M+\cdots+w_{i}^M \rfloor+1,$  where by convention  $u_{0}^M=0$, $\mu_{0}^M=1$. The following technical result holds.
\begin{lemma}\label{ausiliary_remark}
	
If $1\leq p<\ell\leq M$ then

	\begin{enumerate}[ref=\thelemma(\arabic*)]
		\item $\mu_p^M<\mu_{\ell}^M$ if and only if $w_{p+1}^M+\cdots+w_{\ell}^M \geq 1-u_p^M$ \label{sublemmaOne}
		\item $\mu_p^M=\mu_{\ell}^M$ if and only if $w_{p+1}^M+\cdots+w_{\ell}^M <1-u_p^M$ \label{sublemmaDue}
		\item  $\mu_p^M=\mu_{\ell}^M$ if and only if $u_{\ell}^M=u_p^M+w_{p+1}^M+\cdots+w_{\ell}^M.$
	\end{enumerate}
\end{lemma}

\begin{proof}[Proof of Lemma \ref{ausiliary_remark}]
	Let us observe that 2. follows directly from 1. and 3. is a direct consequence of the definition of the integer and fractional part. Therefore it is enough to prove 1.\\
	If $\mu_p^M<\mu_{\ell}^M$, then\\
	$$w_1^M+\cdots+w_{p}^M-u_p^M+1=\mu_p^M\leq\mu_{\ell}^M-1\leq w_1^M+\cdots+w_{\ell}^M$$ 
	that implies
	$$1-u_p^M \leq w_{p+1}^M+\cdots+w_{\ell}^M.$$
	Let us now prove the other implication. If $1-u_p^M \leq w_{p+1}^M+\cdots+w_{\ell}^M$, then 
	$$\mu_{p}^M=w_{1}^M+\cdots+w_{p}^M+1-u_p^M \leq w_{1}^M+\cdots+w_{\ell}^M < \mu_{\ell}^M.$$
\end{proof}
 The following lemma provides an explicit expression for the conditional expectation $\mathbb{E}\left(f( Y_m^{M}) \mid \mathcal{F} \right)$ for $m=1,\cdots,M$ which appears in (\ref{toberewritten2}) and so it allows to prove Proposition \ref{rewriting_term}.

\begin{lemma}\label{cond_expectation} 
	Given  $\varrho\in \mathcal{B}_b\left( \mathbb{R}^d\right)$,
	for $m=1,\cdots,M$ we have
	$$\mathbb{E}\left(\varrho( Y_m^{M}) \mid \mathcal{F} \right)= \sum_{i=1}^{M}\varrho( X_{i})  q_{m,i}^M 1_{\left\lbrace\mu_{i-1}^M\leq m \leq \mu_{i}^M \right\rbrace }$$ 
	
	where for $i=1,\cdots,M$
	
		\begin{equation}\label{coeff}
	q_{m,i}^M = 1_{\left\lbrace \mu_{i-1}^M<\mu_{i}^M\right\rbrace }(1_{\left\lbrace \mu_{i-1}^M<m<\mu_{i}^M\right\rbrace }+1_{\left\lbrace m=\mu_{i-1}^M\right\rbrace }(1-u_{i-1}^M)+1_{\left\lbrace m=\mu_{i}^M\right\rbrace }u_i^M)+1_{\left\lbrace \mu_{i-1}^M=m=\mu_{i}^M\right\rbrace }w_i^M.
	\end{equation}
\end{lemma}
\begin{proof}[Proof of Lemma \ref{cond_expectation}]
For $m=1,\cdots,M$,	by (\ref{stratsampling}), by observing that $m-U_m$ is uniformly distributed on $\left[m-1,m \right] $ and  by the Freezing Lemma, we have
	\begin{align}
	\mathbb{E}\left(\varrho(Y_m^{M})\mid \mathcal{F} \right)&=\sum_{i=1}^{M}\varrho(X_{i}) \int_{m-1}^{m}1_{\left\lbrace \sum\limits_{j=1}^{i-1}w_j^M<u\leq \sum\limits_{j=1}^{i}w_j^M\right\rbrace}du. \label{integral}
	\end{align}

	By observing that for $i=1,\cdots,M$ 
	
	$$1=1_{\left\lbrace  m< \mu_{i-1}^M \right\rbrace }+1_{ \left\lbrace m> \mu_{i}^M \right\rbrace  }+1_{\left\lbrace \mu_{i-1}^M<\mu_{i}^M\right\rbrace }\left( 1_{ \left\lbrace m= \mu_{i-1}^M \right\rbrace } +1_{\left\lbrace m= \mu_{i}^M \right\rbrace } + 1_{\left\lbrace \mu_{i-1}^M <m< \mu_{i}^M \right\rbrace } \right)  +1_{ \left\lbrace m=\mu_{i-1}^M=\mu_{i}^M\right\rbrace },$$
	
	let us study the value of the integral in (\ref{integral}) according to this partition.
	If $m < \mu_{i-1}^M$ then $m\leq \mu_{i-1}^M-1\leq w_1^M+\cdots+w_{i-1}^M$ and if $m > \mu_{i}^M$ then $m\geq  \mu_{i}^M+1> w_1^M+\cdots+w_i^M+1$ and so in both cases the integral is zero. Let us now suppose that $\mu_{i-1}^M<\mu_{i}^M$.
	If $\mu_{i-1}^M<m<\mu_{i}^M$, then $m\geq \mu_{i-1}^M+1> w_1^M+\cdots+w_{i-1}^M+1$ and $m\leq w_1^M+\cdots+w_{i}^M$. Therefore
	
	$$ \int_{m-1}^{m}1_{\left\lbrace \sum\limits_{j=1}^{i-1}w_j^M<u\leq \sum\limits_{j=1}^{i}w_j^M\right\rbrace}du=1.$$
	If $m=\mu_{i-1}^M$:
	$$ \int_{m-1}^{m}1_{\left\lbrace \sum\limits_{j=1}^{i-1}w_j^M<u\leq \sum\limits_{j=1}^{i}w_j^M\right\rbrace}du=\int_{w_1^M+\cdots+w_{i-1}^M}^{\mu_{i-1}^M}du =1-u_{i-1}^M.$$
	If  $m=\mu_{i}^M$:
	$$ \int_{m-1}^{m}1_{\left\lbrace \sum\limits_{j=1}^{i-1}w_j^M<u\leq \sum\limits_{j=1}^{i}w_j^M\right\rbrace}du=\int_{\mu_{i}^M-1}^{w_1^M+\cdots+w_{i}^M}du =u_i^M. $$
	Finally if $m=\mu_{i}^M=\mu_{i-1}^M$:
	$$ \int_{m-1}^{m}1_{\left\lbrace \sum\limits_{j=1}^{i-1}w_j^M<u\leq \sum\limits_{j=1}^{i}w_j^M\right\rbrace}du=\int_{w_1^M+\cdots+w_{i-1}^M}^{w_1^M+\cdots+w_{i}^M}du = w_i^M.$$
	To sum up we have obtained that for $i=1,\cdots, M$
	\begin{align*}
	\int_{m-1}^{m}&1_{\left\lbrace \sum\limits_{j=1}^{i-1}w_j^M<u\leq \sum\limits_{j=1}^{i}w_j^M\right\rbrace}du\\
	&= 1_{\left\lbrace \mu_{i-1}^M<\mu_{i}^M\right\rbrace }(1_{\left\lbrace \mu_{i-1}^M<m<\mu_{i}^M\right\rbrace }+1_{\left\lbrace m=\mu_{i-1}^M\right\rbrace }(1-u_{i-1}^M)+1_{\left\lbrace m=\mu_{i}^M\right\rbrace }u_i^M)+1_{\left\lbrace \mu_{i-1}^M=m=\mu_{i}^M\right\rbrace }w_i^M.\\
	\end{align*}
\end{proof}
\begin{proof}[Proof of Proposition \ref{rewriting_term}]
	
	Let us rewrite   (\ref{toberewritten2}).
	By Lemma \ref{cond_expectation} we have
	\begin{align*}
	\dfrac{1}{M}& \sum_{m=1}^{M} \mathbb{E}\left(f^2(Y_m^{M}) |\mathcal{F}\right) -\frac{1}{M} \sum_{m=1}^{M} \mathbb{E}\left(f(Y_m^{M}) |\mathcal{F}\right)^2 \\
	&=\dfrac{1}{M} \sum_{m=1}^{M}\sum_{i=1}^{M}f^2( X_{i}) q_{m,i}^M 1_{\left\lbrace\mu_{i-1}^M\leq m \leq \mu_{i}^M \right\rbrace }-\dfrac{1}{M} \sum_{m=1}^{M}\sum_{i=1}^{M}f^2( X_{i}) (q_{m,i}^M)^2 1_{\left\lbrace\mu_{i-1}^M\leq m \leq \mu_{i}^M \right\rbrace }\\
	&\phantom{==}-\dfrac{2}{M}\sum_{m=1}^{M}\sum_{1\leq i< j\leq M}^{}f( X_{i})f( X_{j}) q_{m,i}^M q_{m,j}^M1_{\left\lbrace\mu_{i-1}^M\leq m \leq \mu_{i}^M \right\rbrace } 1_{\left\lbrace\mu_{j-1}^M\leq m \leq \mu_{j}^M \right\rbrace }\\ 
	&=\dfrac{1}{M} \sum_{m=1}^{M}\sum_{i=1}^{M}f^2( X_{i}) \left( q_{m,i}^M-(q_{m,i}^M)^2\right)  1_{\left\lbrace\mu_{i-1}^M\leq m \leq \mu_{i}^M \right\rbrace }-\dfrac{2}{M}\sum_{m=1}^{M}\sum_{1\leq i< j\leq M}^{}f( X_{i})f( X_{j}) q_{m,i}^M q_{m,j}^M1_{\left\lbrace m = \mu_{i}^M = \mu_{j-1}^M\right\rbrace}.\numberthis \label{eqn}
	\end{align*}
	
	We are now going first to rewrite the first component of the right-hand side and then the second one.
	\subsubsection*{First term:} 
	Let us observe that by Lemma \ref{ausiliary_remark} applied to the couple $\left( p,l\right) $ equal to  $(i-1,i)$, we have  $1_{\left\lbrace \mu_{i}^M > \mu_{i-1}^M \right\rbrace}=1_{\left\lbrace w_i^M\geq 1-u_{i-1}^M  \right\rbrace }$ and $1_{\left\lbrace \mu_{i}^M = \mu_{i-1}^M \right\rbrace}=1_{\left\lbrace w_i^M< 1-u_{i-1}^M  \right\rbrace }$ for $i=1\cdots,M$. 
	Therefore by Lemma \ref{cond_expectation} and by using that $\mu_{i}^M-\mu_{i-1}^M =u_{i-1}^M-u_{i}^M+w_i^M  $, we have
	\begin{align*}
	&\sum_{m=1}^{M} q_{m,i}^M 1_{\left\lbrace\mu_{i-1}^M\leq m \leq \mu_{i}^M \right\rbrace } \\
	&= \sum_{m=1}^{M}\left( 1_{\left\lbrace w_i^M\geq 1-u_{i-1}^M  \right\rbrace }(1_{\left\lbrace \mu_{i-1}^M<m<\mu_{i}^M\right\rbrace }+1_{\left\lbrace m=\mu_{i-1}^M\right\rbrace }(1-u_{i-1}^M)+1_{\left\lbrace m=\mu_{i}^M\right\rbrace }u_i^M)+1_{\left\lbrace w_i^M< 1-u_{i-1}^M  \right\rbrace }1_{\left\lbrace \mu_{i-1}^M=m=\mu_{i}^M\right\rbrace }w_i^M\right) \\
	&= 1_{\left\lbrace w_i^M\geq 1-u_{i-1}^M  \right\rbrace }\left( \mu_{i}^M-\mu_{i-1}^M-u_{i-1}^M+u_i^M\right)+ 1_{\left\lbrace w_i^M< 1-u_{i-1}^M  \right\rbrace }w_i^M = w_i^M1_{\left\lbrace  w_i^M  \geq 1-u_{i-1}^M \right\rbrace}+w_i^M1_{\left\lbrace w_i^M  < 1-u_{i-1}^M  \right\rbrace} = w_i^M.
	\end{align*}
Let us observe that since, by (\ref{selection}), $\frac{1}{M}\sum\limits_{m=1}^{M} \mathbb{E}\left(f^2(Y_m^{M}) |\mathcal{F}\right) =\frac{1}{M}\sum\limits_{i=1}^{M} f^2( X_{i})w_i^M$, we got the  expected result. Similarly, by using that $\mu_{i}^M-\mu_{i-1}^M =u_{i-1}^M-u_{i}^M+w_i^M  $, we have
	\begin{align*}
	\sum_{m=1}^{M} &(q_{m,i}^M)^2 1_{\left\lbrace\mu_{i-1}^M\leq m \leq \mu_{i}^M \right\rbrace }\\
	&= 1_{\left\lbrace w_i^M\geq 1-u_{i-1}^M  \right\rbrace }(\mu_{i}^M-\mu_{i-1}^M-1+(1-u_{i-1}^M)^2+(u_i^M)^2)+1_{\left\lbrace w_i^M< 1-u_{i-1}^M  \right\rbrace }(w_i^M)^2 \\
	&=(w_i^M-u_i^M(1-u_i^M)-u_{i-1}^M(1-u_{i-1}^M))1_{\left\lbrace w_i^M\geq 1-u_{i-1}^M \right\rbrace}+(w_i^M)^21_{\left\lbrace w_i^M < 1-u_{i-1}^M \right\rbrace}\\
	&= w_i^M-u_i^M(1-u_i^M)-u_{i-1}^M(1-u_{i-1}^M) - (w_i^M(1-w_i^M)-u_i^M(1-u_i^M)-u_{i-1}^M(1-u_{i-1}^M))1_{\left\lbrace w_i^M < 1-u_{i-1}^M \right\rbrace}.
	\end{align*}
	Substrating the two quantities and using Lemma \ref{ausiliary_remark} to rewrite $u_i^M$ as $ w_i^M+u_{i-1}^M$ if $w_i^M <1-u_{i-1}^M$ , we have
	\begin{align*}
	&\sum_{m=1}^{M} q_{m,i}^M 1_{\left\lbrace\mu_{i-1}^M\leq m \leq \mu_{i}^M \right\rbrace }-	\sum_{m=1}^{M} (q_{m,i}^M)^2 1_{\left\lbrace\mu_{i-1}^M\leq m \leq \mu_{i}^M \right\rbrace}\\
	&=u_i^M(1-u_i^M)+u_{i-1}^M(1-u_{i-1}^M)+(w_i^M(1-w_i^M)-u_i^M(1-u_i^M)-u_{i-1}^M(1-u_{i-1}^M))1_{\left\lbrace w_i^M <1-u_{i-1}^M \right\rbrace}\\
	&=u_i^M(1-u_i^M)+u_{i-1}^M(1-u_{i-1}^M)+(w_i^M(1-w_i^M)-(w_i^M+u_{i-1}^M)(1-w_i^M-u_{i-1}^M)-u_{i-1}^M(1-u_{i-1}^M))1_{\left\lbrace  w_i^M <1-u_{i-1}^M \right\rbrace}\\
	&=u_i^M(1-u_i^M)+u_{i-1}^M(1-u_{i-1}^M)-(2u_{i-1}^M(1-u_{i-1}^M)-2w_i^Mu_{i-1}^M)1_{\left\lbrace  w_i^M <1-u_{i-1}^M \right\rbrace}.
	\end{align*}
	
	Therefore by observing that for $i=1\cdots,M$ $ u_i^M =\left\lbrace u_{i-1}^M + w_i^M \right\rbrace $,  the first term can be rewritten as
	\begin{align}\label{firstterm}
	\dfrac{1}{M}& \sum_{m=1}^{M}\sum_{i=1}^{M}f^2(X_{i}) \left( q_{m,i}^M-(q_{m,i}^M)^2\right) 1_{\left\lbrace\mu_{i-1}^M\leq m \leq \mu_{i}^M \right\rbrace}  = \dfrac{1}{M}\sum_{i=1}^{M}f^2( X_{i})\beta_{0}(u_{i-1}^M,w_i^M)
	\end{align}
	with $\beta_{0}$ defined in (\ref{beta0}).
	\subsubsection*{Second term:} 
	By Lemma \ref{cond_expectation} and by rewriting $1_{\left\lbrace \mu_{i-1}^M<\mu_{i}^M\right\rbrace }=1-1_{\left\lbrace \omega_{i}^M<1-u_{i-1}^M\right\rbrace }$, we can rewrite the right-hand side of (\ref{eqn}) in the following way 
	\begin{align*} 
	&\frac{2}{M}\sum_{m=1}^{M}\sum_{1\leq i< j\leq M}^{}f( X_{i})f( X_{j}) q_{m,i}^M q_{m,j}^M1_{\left\lbrace m = \mu_{i}^M = \mu_{j-1}^M  \right\rbrace}\\
	&=\frac{2}{M}\sum_{1\leq i< j\leq M}^{}f( X_{i})f( X_{j}) \left(1_{\left\lbrace \mu_{i-1}^M<\mu_{i}^M\right\rbrace }u_i^M+1_{\left\lbrace \mu_{i-1}^M=\mu_{i}^M\right\rbrace }w_i^M \right)\\
	&\phantom{=========}\times  \left( 1_{\left\lbrace \mu_{j-1}^M<\mu_{j}^M\right\rbrace}(1-u_{j-1}^M)+1_{\left\lbrace \mu_{j-1}^M=\mu_{j}^M\right\rbrace }w_j^M\right) 1_{\left\lbrace  \mu_{i}^M = \mu_{j-1}^M  \right\rbrace}\\
	&=  \frac{2}{M} \sum_{k=1}^{M-1}\sum_{i=1}^{M-k}f( X_{i})f( X_{i+k}) \left( \rho_{i,k}^M+\tilde\rho_{i,k}^M\right) 
	\end{align*}

	where	for $k=1,\cdots,M-1$ and $i=1,\cdots,M-k$
	\begin{align*}
		&\rho_{i,k}^M = u_i^M  \left( 1_{\left\lbrace \mu_{i+k-1}^M<\mu_{i+k}^M\right\rbrace}(1-u_{i+k-1}^M)+1_{\left\lbrace \mu_{i+k-1}^M=\mu_{i+k}^M\right\rbrace }w_{i+k}^M\right) 1_{\left\lbrace  \mu_{i}^M = \mu_{i+k-1}^M  \right\rbrace}\\
			&\tilde{\rho}_{i,k}^M = \left( w_i^M -u_i^M \right)  \left( 1_{\left\lbrace \mu_{i+k-1}^M<\mu_{i+k}^M\right\rbrace}(1-u_{i+k-1}^M)+1_{\left\lbrace \mu_{i+k-1}^M=\mu_{i+k}^M\right\rbrace }w_{i+k}^M\right) 1_{\left\lbrace \mu_{i-1}^M=\mu_{i}^M\right\rbrace }1_{\left\lbrace  \mu_{i}^M = \mu_{i+k-1}^M  \right\rbrace}.
	\end{align*}

	Let us now rewrite $\rho_{i,k}^M$ and $\tilde\rho_{i,k}^M $. To simplify the notation, we will denote  $\sum\limits_{\ell=j_1}^{j_2}w_\ell^M$ for $1 \leq j_1\leq j_2 \leq M$ by $s_{j_1}^{j_2}$. We can apply Lemma \ref{ausiliary_remark} to the couples \\
	
	\begin{enumerate}
		\item $(p,\ell)=(i,i+k-1)$ so that $1_{\left\lbrace \mu_{i}^M = \mu_{i+k-1}^M\right\rbrace }=1_{\left\lbrace s_{i+1}^{i+k-1} <1-u_i^M\right\rbrace }$
		\item  $(p,\ell)=(i-1,i)$ so that $1_{\left\lbrace \mu_{i-1}^M=\mu_{i}^M\right\rbrace }=1_{\left\lbrace \omega_{i}^M<1-u_{i-1}^M\right\rbrace }$
		\item $(p,\ell)=(i+k-1,i+k)$ so that $1_{\left\lbrace \mu_{i+k-1}^M=\mu_{i+k}^M\right\rbrace }=1_{\left\lbrace\omega_{i+k}^M<1-u_{i+k-1}^M \right\rbrace } $  and $1_{\left\lbrace \mu_{i+k-1}^M<\mu_{i+k}^M\right\rbrace }=1_{\left\lbrace\omega_{i+k}^M\geq1-u_{i+k-1}^M \right\rbrace } \\$
	\end{enumerate}
	so that, by observing that when $s_{i+1}^{i+k-1} <1-u_i^M$ we can rewrite $u_{i+k-1}^M$ as $u_i^M+s_{i+1}^{i+k-1}$, one has 
	\begin{align*}
		\rho_{i,k}^M &= u_i^M  \left( 1_{\left\lbrace\omega_{i+k}^M\geq1-u_{i+k-1}^M \right\rbrace }(1-u_{i+k-1}^M)+w_{i+k}^M1_{\left\lbrace\omega_{i+k}^M<1-u_{i+k-1}^M \right\rbrace } \right) 1_{\left\lbrace s_{i+1}^{i+k-1} <1-u_i^M\right\rbrace } \\
		&= u_i^M  \left( 1_{\left\lbrace\omega_{i+k}^M\geq1-u_i^M-s_{i+1}^{i+k-1} \right\rbrace }(1-u_i^M-s_{i+1}^{i+k-1})+w_{i+k}^M1_{\left\lbrace\omega_{i+k}^M<1-u_i^M-s_{i+1}^{i+k-1} \right\rbrace } \right) 1_{\left\lbrace s_{i+1}^{i+k_1-1} <1-u_i^M\right\rbrace }\\
		&=u_i^M (1-u_i^M-s_{i+1}^{i+k-1})1_{\left\lbrace s_{i+1}^{i+k-1} <1-u_i^M\leq s_{i+1}^{i+k} \right\rbrace}+u_i^Mw_{i+k}^M1_{\left\lbrace s_{i+1}^{i+k}<1-u_i^M \right\rbrace }\\
		&= u_i^M (1-u_i^M-s_{i+1}^{i+k-1})1_{\left\lbrace s_{i+1}^{i+k-1} <1-u_i^M \right\rbrace}-u_i^M (1-u_i^M-s_{i+1}^{i+k})1_{\left\lbrace s_{i+1}^{i+k} <1-u_i^M\right\rbrace}
	\end{align*}
	
and, by observing that when $w_i^M < 1- u_{i-1}^M$ and $s_{i+1}^{i+k-1} <1-u_i^M$ we can rewrite $u_{i}^M=u_{i-1}^M+w_i^M $ and  $u_{i+k-1}^M = u_{i-1}^M + s_{i}^{i+k-1}$, one has 
	\begin{align*}
	\tilde\rho_{i,k}^M&=\left( w_i^M -u_i^M\right) \left( 1_{\left\lbrace\omega_{i+k}^M\geq1-u_{i+k-1}^M \right\rbrace }(1-u_{i+k-1}^M)+w_{i+k}^M1_{\left\lbrace\omega_{i+k}^M<1-u_{i+k-1}^M \right\rbrace } \right)1_{\left\lbrace \omega_{i}^M<1-u_{i-1}^M\right\rbrace}1_{\left\lbrace s_{i+1}^{i+k-1} <1-u_i^M\right\rbrace }\\
			&= -\left( u_{i-1}^M+w_i^M\right) \left( 1_{\left\lbrace s_{i}^{i+k}\geq1-u_{i-1}^M \right\rbrace }(1-u_{i-1}^M - s_{i}^{i+k-1})+w_{i+k}^M1_{\left\lbrace s_{i}^{i+k}<1-u_{i-1}^M   \right\rbrace } \right)1_{\left\lbrace s_{i}^{i+k-1} <1-u_{i-1}^M\right\rbrace }\\ 
	&\phantom{=}+w_i^M \left( 1_{\left\lbrace s_{i}^{i+k}\geq1-u_{i-1}^M \right\rbrace }(1-u_{i-1}^M - s_{i}^{i+k-1})+w_{i+k}^M1_{\left\lbrace s_{i}^{i+k}<1-u_{i-1}^M   \right\rbrace } \right)1_{\left\lbrace s_{i}^{i+k-1} <1-u_{i-1}^M\right\rbrace }\\
	&= - u_{i-1}^M (1-u_{i-1}^M - s_{i}^{i+k-1})1_{\left\lbrace s_{i}^{i+k-1} <1-u_{i-1}^M \leq s_{i}^{i+k} \right\rbrace }- u_{i-1}^Mw_{i+k}^M1_{\left\lbrace s_{i}^{i+k}<1-u_{i-1}^M   \right\rbrace }\\
	&=- u_{i-1}^M (1-u_{i-1}^M - s_{i}^{i+k-1})1_{\left\lbrace s_{i}^{i+k-1} <1-u_{i-1}^M \right\rbrace }+ u_{i-1}^M (1-u_{i-1}^M - s_{i}^{i+k})1_{\left\lbrace s_{i}^{i+k} <1-u_{i-1}^M \right\rbrace}.
	\end{align*}

	
	Therefore we have obtained that 
	\begin{align*}
	&\frac{1}{M}\sum_{m=1}^{M}\sum_{1\leq i< j\leq M}^{}f( X_{i})f( X_{j}) q_{m,i}^M q_{m,j}^M1_{\left\lbrace m = \mu_{i}^M = \mu_{j-1}^M \right\rbrace}\\
	& \phantom{==}= \frac{1}{M} \sum_{k=1}^{M-1}\sum_{i=1}^{M-k}f( X_{i})f( X_{i+k}) \beta_{1}\left( u_{i-1}^M,w_i^M,\sum\limits_{\ell=2}^{k}w_{i+\ell-1}^M,w_{i+k}^M\right)   
	\end{align*}
	with $\beta_{1}$ defined as in  (\ref{beta1}). Let us now observe that there exists a finite constant $C_1$ such that for $k=1,\cdots,M-1$, $i=1,\cdots,M-k$
	
	\begin{align}
		\left| \beta_{1}\left( u_{i-1}^M,w_i^M,\sum\limits_{\ell=2}^{k}w_{i+\ell-1}^M,w_{i+k}^M\right)\right|\leq C_1 1_{\left\lbrace \sum\limits_{\ell=2}^{k}w_{i+\ell-1}^M<1\right\rbrace }  \leq  C_1 1_{\left\lbrace k <1+\dfrac{\bar{g}}{\underline{g}}\right\rbrace }.
	\end{align}
	Therefore
		\begin{align*}
	& \frac{1}{M} \sum_{k=1}^{M-1}\sum_{i=1}^{M-k}f( X_{i})f( X_{i+k}) \beta_{1}\left( u_{i-1}^M,w_i^M,\sum\limits_{\ell=2}^{k}w_{i+\ell-1}^M,w_{i+k}^M\right)  \\
	&=  \frac{1}{M} \sum_{k=1}^{(M-1)\wedge \left\lceil\frac{\bar{g}}{\underline{g}}\right\rceil }\sum_{i=1}^{M-k}f( X_{i})f( X_{i+k}) \beta_{1}\left( u_{i-1}^M,w_i^M,\sum\limits_{\ell=2}^{k}w_{i+\ell-1}^M,w_{i+k}^M\right).
	\end{align*}
\end{proof}

\section{Asymptotic Variance for the Next Steps}\label{nextsteps}
What we have seen so far is the study of the stratified sampling selection step: it is actually part of a more general algorithm where each step consists of a selection part and a mutation part that we are now going to describe more in details (see for instance \cite{delmoral} for a more general version of the algorithm).\\

For each step $n\geq 0$, we are going to recursively define the selection sequence $\left(Y_n^{M,m}\right)_{1\leq m\leq M}$ and the mutation sequence  $\left(X_n^{M,m}\right)_{1\leq m\leq M}$.
Let $g_n:\mathbb{R}^{d}\rightarrow \left(0,\infty \right), n\geq 0$ be a family of measurable functions such that $0< \inf_{x\in \mathbb{R}^{d}}g_n(x)\leq \sup_{x\in \mathbb{R}^{d}}g_n(x)< \infty$. In what follows we denote $\bar{g}_n:=\sup_{x\in \mathbb{R}^d}g_n(x)$ and $\underline{g}_n:=\inf_{x\in \mathbb{R}^d}g_n(x) $. Let for $n\geq 0$ and $k\geq 0$

\begin{align}
\phi_{n}(k):=\left\lceil \dfrac{\bar{g}_n}{\underline{g}_n}\left(1+k \right) \right\rceil.\label{deffunzionen}
\end{align} Moreover let $\left(Z_k\right)_{k\geq 0}$ be a $\mathbb{R}^d$-valued time-inhomogeneous Markov chain that is for each bounded measurable function $h:\mathbb{R}^d\rightarrow \mathbb{R}$, $\mathbb{E}\left(h\left( Z_k\right)|\sigma\left(Z_0,\cdots,Z_{k-1} \right)   \right)= \int_{\mathbb{R}^d}h(x)P_k\left( Z_{k-1},dx\right) $ for a transition kernel  $P_k$ possibly depending on $k$. We denote the law of $Z_{0}$ by $\eta$.\\

\begin{itemize}
	\item[\textbf{Initialization:}] we generate a sequence $(X_{0}^{M,m})_{1\leq m\leq M}$ 
	of i.i.d. $\mathbb{R}^d$-valued random vectors following the law $\eta$. By convention $Y_{0}^{M,m} :=X_{0}^{M,m} $ for $1\leq m \leq M$.
	\item[\textbf{From n to n+1:}] the transition from $\left( Y_n^{M,m},X_n^{M,m}\right)_{1\leq m \leq M} $ to  $\left( Y_{n+1}^{M,m},X_{n+1}^{M,m}\right)_{1\leq m \leq M} $ for $n\geq 0$	consists of two steps.
	\begin{itemize}
		\item[Selection:]we generate the random vectors $\left(Y_{n+1}^{M,m} \right)_{1\leq m \leq M}$ conditionally independent given \\
		$\mathcal{F}^{n}:= \sigma\left( \left(X_i^{M,m},Y_i^{M,m}\right)_{1\leq m\leq M,0\leq i\leq n}\right)$ by 
		\begin{align}
		Y_{n+1}^{M,m}=\sum_{\ell=1}^{M}1_{\left\lbrace \frac{M\sum\limits_{j=1}^{\ell-1}g_n\left(X_n^{M,j} \right)}{\sum\limits_{j=1}^{M}g_n\left(X_n^{M,j} \right)}<m-U_n^{m}\leq \frac{M\sum\limits_{j=1}^{\ell}g_n\left(X_n^{M,j} \right)}{\sum\limits_{j=1}^{M}g_n\left(X_n^{M,j} \right)}\right\rbrace } X_n^{M,\ell} \quad m=1,\cdots,M \label{defYn}
		\end{align}
		where $\left( U_n^m\right)_{1\leq m \leq M} $ is a sequence of i.i.d. random variables  independent of $\mathcal{F}^n$ and distributed according to the uniform law on $(0,1)$.
		\item[Mutation:] given $\mathcal{G}^{n+1}:=    \sigma\left( \left(X_i^{M,m},Y_i^{M,m}\right)_{1\leq m\leq M,0\leq i\leq n} ,\left( Y_{n+1}^{M,m} \right)_{1\leq m\leq M}\right)$ we generate the random vectors $X_{n+1}^{M,m},1\leq m \leq M$ conditionally independent and respectively distributed according to the probability measure $P_{n+1}\left(Y_{n+1}^{M,m},\cdot \right)$\\
	\end{itemize}
\end{itemize}

For each $n\geq 0$ we denote the weights by
\begin{align*}
w_n^{M,m}=\frac{Mg_n(X_n^{M,m})}{\sum\limits_{\ell=1}^{M}g_n(X_n^{M,\ell})}\quad m=1,\cdots,M.
\end{align*}

Moreover let us observe that the selection property for $n=0$ (\ref{selection})
 remains valid for each $n$: $\forall n\geq 0$

\begin{align}\label{selection_n}
\mathbb{E}\left( \frac{1}{M}\sum_{m=1}^{M}\delta_{Y_{n+1}^{M,m}}\arrowvert\mathcal{F}^{n}\right)= \sum\limits_{m=1}^{M}\frac{w_n^{M,m}}{M}\delta_{ X_n^{M,m}}.	
\end{align}

The study of the asymptotic variance of the selection part of the first step that is $\textrm{Var}\left( \frac{1}{\sqrt{M}}\sum_{m=1}^{M}f\left( Y_1^{M,m}\right) \right)$  for a given bounded measurable function $f:\mathbb{R}^{ d}\rightarrow \mathbb{R}$ was the object of Section \ref{secasymptoticvariance}.
Our purpose now is to generalize this result by studying the asymptotic variance of $\frac{1}{\sqrt{M}}\sum\limits_{m=1}^{M}f\left(Y_{n+1}^{M,m}\right)$ for  $n\geq 1$.\\ Before doing it let us begin with some notation: for $n\geq 0$, $M\geq 1$ and $k= 0,\cdots, M-1$ let 

\begin{align*}
\widetilde\eta_{n}^{k,M}:= \frac{1}{M}\sum\limits_{i=1}^{M-k} \delta_{\left(Y_n^{M,i} ,\cdots,Y_{n}^{M,i+k} \right) }
\end{align*}
\begin{align*}
\bar\eta_{n}^{k,M} := \frac{1}{M}\sum\limits_{i=1}^{M-k} \delta_{\left( X_n^{M,i},\cdots ,X_{n}^{M,i+k}\right)}.
\end{align*}

\begin{remark}\label{stronglaw}
Let $f\in \mathcal{B}_b\left( \mathbb{R}^d\right) $. If $k=0$ and $n=0$, by the classical Strong Law of Large Numbers we have 
  
  \begin{align*}
  	\lim_{M\rightarrow \infty}\widetilde\eta_{0}^{0,M}(f) = \lim_{M\rightarrow \infty}\bar\eta_{0}^{0,M}(f) = \bar\eta_{0}^{0}(f):= \eta(f) =  \mathbb{E}\left(f\left(Z_0 \right)  \right)\quad a.s. 
  \end{align*}
  
  It is possible to prove that a  Strong Law of Large Numbers holds (see \cite[Corollary 7.4.2]{delmoral}) also in the case $k=0$ and $n\geq 1$:
	\begin{align*}
		\lim_{M\rightarrow \infty}\widetilde\eta_{n}^{0,M}(f) = \lim_{M\rightarrow \infty} \frac{1}{M}\sum\limits_{i=1}^{M } f\left(Y_n^{M,i}\right)= \widetilde\eta_{n}^{0}(f) := \frac{\mathbb{E}\left(f\left(Z_{n-1} \right)\prod_{p=0}^{n-1}g_p\left(Z_p \right)  \right)}{\mathbb{E}\left(\prod_{p=0}^{n-1}g_p\left(Z_p \right)  \right)}\quad a.s.  
	\end{align*}
	
	\begin{align*}
		\lim_{M\rightarrow \infty}\bar\eta_{n}^{0,M}(f) = 		\lim_{M\rightarrow \infty} \frac{1}{M}\sum\limits_{i=1}^{M}f\left(X_n^{M,i}  \right)= \bar\eta_{n}^{0}(f) := \frac{\mathbb{E}\left(f\left(Z_{n} \right)\prod_{p=0}^{n-1}g_p\left(Z_p \right)  \right)}{\mathbb{E}\left(\prod_{p=0}^{n-1}g_p\left(Z_p \right)\right)}\quad a.s.  .
	\end{align*}
	
\end{remark}
	We are now going to prove that given $h:\mathbb{R}^{2 d}\rightarrow \mathbb{R}$ a bounded measurable function, the study of the asymptotic variance of $\frac{1}{\sqrt{M}}\sum\limits_{m=1}^{M}h\left(Y_{n+1}^{M,m}, X_{n+1}^{M,m}\right) $ depends on the study of the asymptotic variance of $\frac{1}{\sqrt{M}}\sum\limits_{m=1}^{M}f\left(Y_{n+1}^{M,m}\right)$ for $f:\mathbb{R}^{ d}\rightarrow \mathbb{R}$ a bounded measurable function. 
\begin{prop}\label{mutesel}
	Let $h\in \mathcal{B}_b\left( \mathbb{R}^{2 d}\right) $. The following convergence holds
	\begin{align*}
		\left| \textrm{Var}\left(\frac{1}{\sqrt{M}}\sum\limits_{m=1}^{M}h\left(Y_{n+1}^{M,m}, X_{n+1}^{M,m}\right) \right)- \textrm{Var}\left(\frac{1}{\sqrt{M}}\sum\limits_{m=1}^{M} P_{n+1}h\left(Y_{n+1}^{M,m}\right)\right)-\frac{\bar\eta_{n}^{0}\left( g_n\left( P_{n+1}h^2-\left( P_{n+1}h\right) ^2\right) \right)  }{\bar\eta_{n}^{0}(g_n)} \right| \underset{M\rightarrow \infty}{\longrightarrow} 0 .
	\end{align*}
\end{prop}
\begin{proof}[Proof of Proposition \ref{mutesel}]
By using that the $X_{n+1}^{M,m}$ are conditionally independent given $\mathcal{G}^{n+1}$ and that the $Y_{n+1}^{M,m}$ are $\mathcal{G}^{n+1}$ measurable, one has\\

	\begin{align*}
	&\textrm{Var}\left(\frac{1}{\sqrt{M}}\sum\limits_{m=1}^{M}h\left(Y_{n+1}^{M,m}, X_{n+1}^{M,m}\right) \right)\nonumber \\
	& = \textrm{Var}\left(\frac{1}{\sqrt{M}}\sum\limits_{m=1}^{M} \mathbb{E}\left(h\left(Y_{n+1}^{M,m}, X_{n+1}^{M,m}\right)\mathrel{\Big|}\mathcal{G}^{n+1} \right)\right)+\mathbb{E}\left(\frac{1}{M}\sum\limits_{m=1}^{M} \textrm{Var}\left(h\left(Y_{n+1}^{M,m}, X_{n+1}^{M,m}\right) \mathrel{\Big|}\mathcal{G}^{n+1}\right)\right)\nonumber\\
	& = \textrm{Var}\left(\frac{1}{\sqrt{M}}\sum\limits_{m=1}^{M} \mathbb{E}\left(h\left(Y_{n+1}^{M,m}, X_{n+1}^{M,m}\right)\mathrel{\Big|}\mathcal{G}^{n+1} \right)\right)+ \mathbb{E}\left(\frac{1}{M}\sum\limits_{m=1}^{M} \mathbb{E}\left(h^2\left(Y_{n+1}^{M,m}, X_{n+1}^{M,m}\right) \mathrel{\Big|}\mathcal{G}^{n+1}\right)\right)\nonumber\\
	&\phantom{==}- \mathbb{E}\left(\frac{1}{M}\sum\limits_{m=1}^{M} \mathbb{E}\left(h\left(Y_{n+1}^{M,m}, X_{n+1}^{M,m}\right) \mathrel{\Big|}\mathcal{G}^{n+1}\right)^2\right).\nonumber\\
	\end{align*}
	
	Since $\mathcal{L}\left( X_{n+1}^{M,m}\mathrel{|}\mathcal{G}^{n+1} \right)\sim P_{n+1}\left(Y_{n+1}^{M,m},\cdot\right)$, for each bounded measurable function $\tilde{h}:\mathbb{R}^{2 d}\rightarrow \mathbb{R}$ one has
	
	\begin{align} \label{law}
	\mathbb{E}\left( \tilde{h}(Y_{n+1}^{M,m},X_{n+1}^{M,m})\mathrel{|}\mathcal{G}^{n+1} \right) = P_{n+1}\tilde{h}\left(Y_{n+1}^{M,m}\right):= \int_{\mathbb{R}^d} \tilde{h}(Y_{n+1}^{M,m},x)P_{n+1}\left(Y_{n+1}^{M,m},dx\right).
	\end{align}
	
	Therefore we obtain that 
	
	\begin{align*}
	&\textrm{Var}\left(\frac{1}{\sqrt{M}}\sum\limits_{m=1}^{M}h\left(Y_{n+1}^{M,m}, X_{n+1}^{M,m}\right) \right)\nonumber \\
	&= \textrm{Var}\left(\frac{1}{\sqrt{M}}\sum\limits_{m=1}^{M} P_{n+1}h\left(Y_{n+1}^{M,m}\right)\right)+ \mathbb{E}\left(\frac{1}{M}\sum\limits_{m=1}^{M} P_{n+1}h^2\left(Y_{n+1}^{M,m}\right)\right)- \mathbb{E}\left(\frac{1}{M}\sum\limits_{m=1}^{M} \left( P_{n+1}h\left(Y_{n+1}^{M,m}\right)\right) ^2\right)\\
	&= \textrm{Var}\left(\frac{1}{\sqrt{M}}\sum\limits_{m=1}^{M} P_{n+1}h\left(Y_{n+1}^{M,m}\right)\right)+ \mathbb{E}\left( \sum\limits_{m=1}^{M}\frac{w_n^{M,m}}{M}P_{n+1}h^2\left( X_n^{M,m}\right) \right) -  \mathbb{E}\left( \sum\limits_{m=1}^{M}\frac{w_n^{M,m}}{M}\left( Ph\left( X_n^{M,m}\right)\right) ^2 \right)
	\end{align*}
	where to obtain the last equality we use the selection property  (\ref{selection_n}). By Remark \ref{stronglaw},
	
	\begin{align*}
		&\lim_{M\rightarrow \infty} \frac{1}{M}\sum\limits_{m=1}^Mg_n\left(X_n^{M,m}\right) P_{n+1}h^2\left( X_n^{M,m}\right) =  \bar\eta_{n}^{0}(g_nP_{n+1}h^2)\quad a.s.,\\
		& \lim_{M\rightarrow \infty} \frac{1}{M}\sum\limits_{m=1}^Mg_n\left(X_n^{M,m}\right) \left( P_{n+1}h\left( X_n^{M,m}\right) \right) ^2=   \bar\eta_{n}^{0}(g_n\left( P_{n+1}h\right) ^2)\quad a.s.,\\
		& \lim_{M\rightarrow \infty} \frac{1}{M}\sum\limits_{m=1}^Mg_n\left(X_n^{M,m}\right) =  \bar\eta_{n}^{0}(g_n)\quad a.s.
	\end{align*}
	Thus by Lebesgue theorem, 
	\begin{align*}
	&	\lim_{M\rightarrow \infty}\mathbb{E}\left( \sum\limits_{m=1}^{M}\frac{w_n^{M,m}}{M}P_{n+1}h^2\left( X_n^{M,m}\right) \right)= \frac{\bar\eta_{n}^{0}(g_nP_{n+1}h^2) }{\bar\eta_{n}^{0}(g_n)},\\
	& 	\lim_{M\rightarrow \infty}\mathbb{E}\left( \sum\limits_{m=1}^{M}\frac{w_n^{M,m}}{M}\left( P_{n+1}h\left( X_n^{M,m}\right)\right) ^2 \right) =\frac{\bar\eta_{n}^{0}(g_n\left( P_{n+1}h\right) ^2) }{\bar\eta_{n}^{0}(g_n)}
	\end{align*}
	and this concludes the proof.
\end{proof}
We are now ready to study for $n\geq1$  and $f\in \mathcal{B}_b\left(\mathbb{R}^d \right) $ the convergence as $M\rightarrow \infty$ of $$V_{n+1}^M\left(f \right):=\textrm{Var}\left(\frac{1}{\sqrt{M}}\sum\limits_{m=1}^{M}f\left(Y_{n+1}^{M,m}\right)\right).$$

With the aim of studying $\lim_{M\rightarrow \infty}V_{n+1}^M\left(f \right)$,
we will assume the following conjectures:

\begin{conjecture}\label{conj1} For each $f\in \mathcal{B}_b\left(\mathbb{R}^d \right) $ one has 
	\begin{align*} 
		\left|\textrm{Var}\left( \sqrt{M}\dfrac{\bar{\eta}^{0,M}_{n}\left(g_nf \right)}{\bar{\eta}^{0,M}_{n}\left(g_n \right)} \right) - \frac{V_{n}^M\left(P_nf_n\right)}{\left( \bar{\eta}_n^0\left(g_n\right)\right) ^4 }-\frac{\bar{\eta}_{n-1}^0\left(g_{n-1}\left( P_n f_n^2-\left(P_n f_n \right)^2 \right) \right) }{\left( \bar{\eta}_n^0\left(g_n\right)\right)^4\bar{\eta}_{n-1}^0\left(g_{n-1}\right)}\right|   \underset{M\rightarrow \infty}{\longrightarrow} 0 
	\end{align*}
where
\begin{align*}
f_n:=g_n\left( \bar{\eta}_n^0\left(g_n\right)f-\bar{\eta}_n^0\left(g_nf\right)\right).
\end{align*}
\end{conjecture}
\begin{remark}
	In Section \ref{secasymptoticvariance} we proved that in the case $n=0$, by observing that $\textrm{Var}\left( \sqrt{M}\dfrac{\bar{\eta}^{0,M}_{0}\left(g_0f \right)}{\bar{\eta}^{0,M}_{0}\left(g_0 \right)} \right) = \textrm{Var}\left( \sqrt{M}\dfrac{\bar{\eta}^{0,M}_{0}\left(\frac{g_0}{\eta\left(g_0 \right)}f \right)}{\bar{\eta}^{0,M}_{0}\left(\frac{g_0}{\eta\left(g_0 \right)}\right)} \right)$,  one has
	$\left|\textrm{Var}\left( \sqrt{M}\dfrac{\bar{\eta}^{0,M}_{0}\left(g_0f \right)}{\bar{\eta}^{0,M}_{0}\left(g_0\right)} \right) -\frac{\eta\left(f_0^2 \right)}{\eta\left(g_0 \right)^4}  \right|\underset{M\rightarrow \infty}{\longrightarrow} 0.$ \\
	
	For $n\geq 1$ the idea underlying the Conjecture \ref{conj1} is the following: by considering
	\begin{align*}
		\textrm{Var}\left( \sqrt{M}\left( \dfrac{\bar{\eta}^{0,M}_{n}\left(g_nf \right)}{\bar{\eta}^{0,M}_{n}\left(g_n \right)}- \frac{\bar{\eta}^{0}_{n}\left(g_nf \right)}{\bar{\eta}^{0}_{n}\left(g_n \right)} \right) \right) = \textrm{Var}\left( \sqrt{M} \frac{\bar{\eta}^{0,M}_{n}\left(f_n \right)}{\bar{\eta}^{0,M}_{n}\left(g_n \right)\bar{\eta}^{0}_{n}\left(g_n \right)}  \right),
	\end{align*} 
	we asymptotically replace in the denominator of the right-hand side $\bar{\eta}^{0,M}_{n}\left(g_n \right)$ by $\bar{\eta}^{0}_{n}\left(g_n \right)$. Thus we have
	\begin{align*}
		\left|\textrm{Var}\left( \sqrt{M} \frac{\bar{\eta}^{0,M}_{n}\left(f_n \right)}{\bar{\eta}^{0,M}_{n}\left(g_n \right)\bar{\eta}^{0}_{n}\left(g_n \right)}  \right) - \frac{1}{\left( \bar{\eta}_n^0\left(g_n\right)\right) ^4} \textrm{Var}\left(\sqrt{M}\bar{\eta}^{0,M}_{n}\left(f_n \right)\right)  \right|\underset{M\rightarrow \infty}{\longrightarrow}0.
	\end{align*}
	We then rewrite
	 $\textrm{Var}\left(\sqrt{M}\bar{\eta}^{0,M}_{n}\left(f_n \right)\right) $ following the same
	reasoning used in the proof of Proposition \ref{mutesel}. 
\end{remark}
\begin{conjecture}\label{conj2} Let  $t\in \mathbb{N}$. Given $\psi\in \mathcal{B}_{b}\left(\mathbb{R}^{t+2}\right)$ continuous and $h\in \mathcal{B}_{b}\left(\mathbb{R}^{d(t+1)}\right)$, the following convergence holds 
	\begin{align*}
		&\bigg\lvert \frac{1}{M}\sum_{m=1}^{M-t} \mathbb{E}\left(h\left(X_n^{M,m},\cdots, X_n^{M,m+t} \right)\psi\left(u_{n}^{M,m-1},w_n^{M,m},\cdots,w_n^{M,m+t} \right)   \right)\\
		&\phantom{==}-\frac{1}{M}\sum_{m=1}^{M-t} \mathbb{E}\left(h\left(X_n^{M,m},\cdots, X_n^{M,m+t} \right)\int_{0}^{1}\psi\left(u,\tilde{g}_n(X_n^{M,m}),\cdots,\tilde{g}_n(X_n^{M,m+t})\right)  du \right) \bigg\lvert \underset{M\rightarrow \infty}{\longrightarrow} 0
	\end{align*}
	where $u_{n}^{M,m-1} = \left\lbrace\frac{M\sum\limits_{i=1}^{m-1}g_n\left(X_n^{M,i} \right)}{\sum\limits_{i=1}^{M}g_n\left(X_n^{M,i} \right) } \right\rbrace $ and $\tilde{g}_n(x)= \frac{g_n(x)}{\bar\eta_{n}^{0}(g_n)}$.	
\end{conjecture}

\begin{remark}
	In Section \ref{secasymptoticvariance} we have provided a formal proof of Conjecture $2$ in the case $n=0$. In that case we strongly used the fact that the random variables we are working with $\left( \left(X_0^{M,m} \right)_{m\geq1}\right) $ are i.i.d. and the fact that the law of $g_0\left(X_0^{M,m} \right) $ has an absolutely continuous component to prove that is possible to asymptotically replace $u_{0}^{M,i-1}$ with a uniformly distributed random variable independent of $\mathcal{F}^{0}$  and $w_{0}^{M,i}$ with $\frac{g_0\left(X_{0}^{M,i} \right)}{\bar\eta_{0}^{0}(g_0)}$.
\end{remark}

Before providing the main result of this section, let us introduce some notation.\\ For  $k\geq 0$ let $P^{\otimes (k+1)}_n:\mathcal{B}_b\left( \mathbb{R}^{d(k+1)}\right)\rightarrow \mathcal{B}_b\left( \mathbb{R}^{d(k+1)}\right)$ be defined by 

\begin{align*}
	P^{\otimes (k+1)}_nh\left(x_0,\cdots,x_k \right):= \int_{\mathbb{R}^{d(k+1)}} h\left(y_0,\cdots,y_k \right)P_n\left(x_0,dy_0 \right) \cdots P_n\left(x_k,dy_k \right)
\end{align*}

and let  for $s_k =0,\cdots, \phi_{{n-1}}(k)$, 
$T_{n-1}^{k\rightarrow s_k }:\mathcal{B}_b\left(\mathbb{R}^{d(k+1)} \right) \rightarrow \mathcal{B}_b\left(\mathbb{R}^{d(s_k+1)} \right) $  given by

 \begin{align*}
&T_{n-1}^{k\rightarrow s_k }f\left(x_0,x_1,\cdots,x_{s_{k}} \right)=\sum_{0\leq s_1\leq s_2\leq\cdots \leq s_{k}} f\left( x_0,x_{s_1}, \cdots,x_{s_{k}} \right)\int_{0}^{1}\psi_{s_{1}:s_{k}}\left(u,\tilde{g}_{n-1}(x_{0}),\tilde{g}_{n-1}(x_{1}),\cdots ,\tilde{g}_{n-1}(x_{s_k})\right)du
\end{align*}
and where given  $0\leq s_k \leq \phi_{{n-1}}(k)$ and $s_0:=0\leq s_1\leq s_2\leq\cdots \leq s_k$:

\begin{align*}
&\psi_{s_{1}:s_{k}}\left(u,y_{0},y_{1},\cdots ,y_{s_k}\right) =\sum\limits_{i=^ 1}^{\left\lceil 1+\bar{g}_{n} /\underline{g}_{n}\right\rceil} \prod_{q=0}^{k}  \int_{i+q-1}^{i+q}1_{\left\lbrace u+\sum\limits_{j=0}^{s_{q}-1}y_{j} <u'\leq u+\sum\limits_{j=0}^{s_{q}}y_{j} \right\rbrace}du'.
\end{align*}

Finally let $\mathcal{T}_{n-1}^{k }:\mathcal{B}_b\left( \mathbb{R}^{d(k+1)}\right)\rightarrow  \mathcal{B}_b\left( \mathbb{R}^{d(\phi_{{n-1}}\left(k \right) +1)}\right)
$ defined by 
\begin{align*}
\mathcal{T}_{n-1}^{k }h\left(x_0,\cdots,x_{\phi_{{n-1}}\left(k \right)}\right)  = \sum\limits_{ s_k = 0}^{\phi_{{n-1}}(k)} T_{n-1}^{k\rightarrow s_k}\left( P^{\otimes (k+1)}_nh\right) \left(x_0,x_1,\cdots,x_{s_k}\right) .
\end{align*}

\begin{teo}\label{finalresults} Let us assume Conjecture \ref{conj1} and Conjecture \ref{conj2}. Then  $\forall n\geq 1$ and $\forall f\in \mathcal{B}\left(\mathbb{R}^{d} \right) $, $V_{n}^M\left(f \right) $ converges as $M$ goes to infinity and we denote its limit by $V_{n}\left(f \right)$. Moreover, by defining $\phi_{m:n}\left(0 \right) := \phi_{m}\left(\phi_{{m+1}}\left(\cdots \left( \phi_{{n}}\left( 0\right)\right) \right) \right) $ for  $m\leq n$, $m,n\in \mathbb{N}$, $V_{n}\left(f \right)$ is defined by the following recursive formula:

	\begin{align*}
	 V_{n+1}\left(f \right) &= \frac{V_{n}\left(P_nf_n\right)}{\left( \bar{\eta}_n^0\left(g_n\right)\right) ^4 }+\frac{\bar{\eta}_{n-1}^0\left(g_{n-1}\left( P_n f_n^2-\left(P_n f_n \right)^2 \right) \right) }{\left( \bar{\eta}_n^0\left(g_n\right)\right)^4\bar{\eta}_{n-1}^0\left(g_{n-1}\right)}\\
	 &\phantom{=}+ \int_{}\mathcal{T}_{0}^{\phi_{1:n}\left(0 \right)}\cdots \mathcal{T}_{n-2}^{\phi_{n-1:n}\left(0 \right)} \mathcal{T}_{n-1}^{\phi_{n:n}\left(0 \right)}\bar{f}_n\left( x\right) \eta^{\otimes \phi_{0:n}\left(0 \right) }\left(dx\right)
	\end{align*}
	
	where
	$\bar{f}_n\in \mathcal{B}_b\left( \mathbb{R}^{d(1+\phi_{{n}}\left(0\right))}\right) $ is given by
	$$\bar{f}_n\left( x_0,\cdots,x_{\phi_{{n}}\left(0\right)}\right) := \sum\limits_{ k = 0}^{\phi_{{n}}\left(0\right) }f\left(x_0 \right)f\left(x_k \right) \int_{0}^{1}\bar\beta_k\left(u,\tilde{g}_{n}\left( x_0 \right) ,\cdots,\tilde{g}_{n}\left( x_k \right) \right) du$$
	with $\bar\beta_k\left(u,y_0,\cdots,y_k \right) = \beta_{0}(u,y_0)1_{\left\lbrace k=0\right\rbrace }- \beta_{1}\left( u,y_0,\sum\limits_{\ell=1}^{k-1}y_{\ell},y_k\right) 1_{\left\lbrace k\neq 0\right\rbrace }  $ where  $\beta_{0}$ and $\beta_{1}$ are respectively defined in (\ref{beta0}) and (\ref{beta1}).
\end{teo}
\begin{remark}
	Let us observe that by Theorem \ref{risultato_principale}, $\lim\limits_{M\rightarrow \infty}V_{1}^M\left(h\right)= V_{1}\left(h\right) = \sigma^2_1(h)+\sigma^2_2(h)$ $\forall h \in \mathcal{B}_b\left(\mathbb{R}^d \right) $.
\end{remark}
The proof of Theorem \ref{finalresults} relies on the following proposition the proof of which is provided after the proof of the theorem.
\begin{prop}\label{mainconv}  Let us assume Conjecture \ref{conj2}. Given $n\geq 1$ and $h\in \mathcal{B}_b\left(\mathbb{R}^{d\left( \phi_{n}(0)+1\right) } \right)$ one has
	\begin{align*}
		\left|\mathbb{E}\left(\bar\eta_{n}^{\phi_{n}(0),M}(h) \right)- \int_{}\mathcal{T}_{0}^{\phi_{1:n}\left(0 \right)}\cdots \mathcal{T}_{n-2}^{\phi_{n-1:n}\left(0 \right)} \mathcal{T}_{n-1}^{\phi_{n:n}\left(0 \right)}h\left( x\right) \eta^{\otimes \phi_{0:n}\left(0 \right) }\left(dx\right)  \right| \underset{M\rightarrow \infty}{\longrightarrow} 0.	\end{align*}
\end{prop}

\begin{proof}[Proof of Theorem \ref{finalresults}]
Similarly to what we have done in (\ref{cond_variance}) and (\ref{splitting}), we can rewrite the variance in the following way
\begin{align}
 V_{n+1}^M\left(f \right) & = \textrm{Var}\left(\dfrac{1}{\sqrt{M}}\sum_{m=1}^{M}f(Y_{n+1}^{M,m})\right)\nonumber \\
&= \textrm{Var}\left( \mathbb{E}\left(\dfrac{1}{\sqrt{M}}\sum_{m=1}^{M}f(Y_{n+1}^{M,m})\mathrel{\Big|}\mathcal{F}^{n} \right)\right)+\mathbb{E}\left( \textrm{Var}\left(\dfrac{1}{\sqrt{M}}\sum_{m=1}^{M}f(Y_{n+1}^{M,m}) \mathrel{\Big|}\mathcal{F}^{n}\right)\right)\nonumber\\
&=\textrm{Var}\left( \sqrt{M}\dfrac{\sum_{m=1}^{M}g_n(X_n^{M,m})f(X_n^{M,m})}{\sum_{m=1}^{M}g_n({X_n^{M,m}})} \right)+\mathbb{E}\left( \textrm{Var}\left(\dfrac{1}{\sqrt{M}}\sum_{m=1}^{M}f(Y_{n+1}^{M,m}) \mathrel{\Big|}\mathcal{F}^{n}\right)\right) \label{rewriting-stepn}
\end{align}
where to obtain the last equality we use the selection property (\ref{selection_n}).
Conjecture \ref{conj1} gives the asymptotic behavior  of the first term of (\ref{rewriting-stepn}).\\ 
Let us then study the second term of (\ref{rewriting-stepn}).
Using the definition of $Y_{n+1}^{M,m}$ and following the same reasoning done to prove Proposition \ref{rewriting_term}, we can rewrite the expression inside the expectation: for $M\geq 1+\phi_{n}(0)$
\begin{align}
& \textrm{Var}\bigg(\dfrac{1}{\sqrt{M}}\sum_{m=1}^{M}f(Y_{n+1}^{M,m}) \mathrel{\Big|}\mathcal{F}^{n}\bigg) \label{expression_var_n} \\
 &=\frac{1}{M}\sum_{k=0}^{\phi_{n}(0)}\sum\limits_{i=1}^{M-k}f( X_{n}^{M,i})f( X_{n}^{M,i+k})\bar\beta_k\left(u_{n}^{M,i-1},w_{n}^{M,i},\cdots,w_{n}^{M,i+k} \right)   \label{form2}
\end{align}
with $\bar\beta_k\left(u,y_0,\cdots,y_k \right): = \beta_{0}(u,y_0)1_{\left\lbrace k=0\right\rbrace }- \beta_{1}\left( u,y_0,\sum\limits_{\ell=1}^{k-1}y_{\ell},y_k\right) 1_{\left\lbrace k\neq 0\right\rbrace }  $
where we recall that $\beta_{0}$ and $\beta_{1}$ are respectively defined in (\ref{beta0}) and (\ref{beta1}).
 Given $k=0,\cdots,\phi_{n}(0)$, we can apply Conjecture \ref{conj2} with $t=k$, $h\left(x_0,\cdots,x_k \right)= f\left(x_0\right)f\left(x_k\right)$ and $\psi\left(u,y_0,\cdots,y_k \right)= \bar\beta_k\left(u,y_0,\cdots,y_k \right)$ and obtain that 
 \begin{align}
 	&\left| \mathbb{E}\left( \textrm{Var}\left(\dfrac{1}{\sqrt{M}}\sum_{m=1}^{M}f(Y_{n+1}^{M,m}) \mathrel{\Big|}\mathcal{F}^{n}\right)\right) -\sum_{k=0}^{\phi_{n}(0)}
 	\mathbb{E}\left(\bar{\eta}_n^{k,M} \left( \tilde{f}_k\right) \right)  \right| \underset{M\rightarrow \infty}{\longrightarrow} 0 \label{conv1}
 \end{align}

where for $k=0,\cdots,\phi_{n}(0)$, $\tilde{f}_k\left(x_0,\cdots,x_k \right)= f\left(x_0 \right)f\left(x_k \right) \int_{0}^{1}\bar\beta_k\left(u,\tilde{g}_{n}\left( x_0 \right) ,\cdots,\tilde{g}_{n}\left( x_k \right) \right) du$. \\
We now observe that by defining $\bar{f}_n\left(x_0,\cdots,x_{\phi_{n}(0)} \right) = \sum\limits_{k=0}^{\phi_{n}(0)} \tilde{f}_k\left(x_0,\cdots,x_k \right) $, one has

\begin{align*}
	&\left|\sum_{k=0}^{\phi_{n}(0)}
	\bar{\eta}_n^{k,M} \left( \tilde{f}_k\right)  - \bar{\eta}_n^{\phi_{n}(0),M} \left(\bar{f}_n\right)  \right|\\
	&= \frac{1}{M} \left|\sum_{k=0}^{\phi_{n}(0)}\left(\sum\limits_{ i = 1}^{M-k}\tilde{f}_k\left(X_n^{M,i},\cdots,X_n^{M,i+k} \right) -\sum\limits_{ i = 1}^{M-\phi_{n}(0)}\tilde{f}_k\left(X_n^{M,i},\cdots,X_n^{M,i+k} \right)  \right)  \right| \\
	&\leq \frac{C}{M}\sum_{k=0}^{\phi_{n}(0)}\left(\phi_{n}(0)-k \right) \underset{M\rightarrow \infty}{\longrightarrow} 0
\end{align*}
for a finite constant $C$. Thus  we can combine it with (\ref{conv1}) and  obtain that 
\begin{align*}
&\left| \mathbb{E}\left( \textrm{Var}\left(\dfrac{1}{\sqrt{M}}\sum_{m=1}^{M}f(Y_{n+1}^{M,m}) \mathrel{\Big|}\mathcal{F}^{n}\right)\right) -
\mathbb{E}\left(\bar{\eta}_n^{\phi_{n}(0),M} \left(\bar{f}_n\right) \right)  \right| \underset{M\rightarrow \infty}{\longrightarrow} 0 .
\end{align*}
We can now apply Proposition \ref{mainconv} with $h=\bar{f}_n $ and obtain that 
\begin{align*}
	\left| \mathbb{E}\left(\bar{\eta}_n^{\phi_{n}(0),M} \left(\bar{f}_n\right) \right)  - \int_{}\mathcal{T}_{0}^{\phi_{1:n}\left(0 \right)}\cdots \mathcal{T}_{n-2}^{\phi_{n-1:n}\left(0 \right)} \mathcal{T}_{n-1}^{\phi_{n:n}\left(0 \right)}\bar{f}_n\left( x\right)  \eta^{\otimes \phi_{0:n}\left(0 \right) }\left(dx\right) \right|\underset{M\rightarrow \infty}{\longrightarrow} 0. 
\end{align*}
Thus we have obtained that $\forall n\geq 1$, $\forall f\in \mathcal{B}_b\left(\mathbb{R}^d\right) $
\begin{align*}
	 &\bigg{\lvert}V_{n+1}^M\left(f \right)- \frac{V_{n}^M\left(P_nf_n\right)}{\left( \bar{\eta}_n^0\left(g_n\right)\right) ^4 }-\frac{\bar{\eta}_{n-1}^0\left(g_{n-1}\left( P_n f_n^2-\left(P_n f_n \right)^2 \right) \right) }{\left( \bar{\eta}_n^0\left(g_n\right)\right)^4\bar{\eta}_{n-1}^0\left(g_{n-1}\right)}\\
	 &\phantom{====}- \int_{}\mathcal{T}_{0}^{\phi_{1:n}\left(0 \right)}\cdots \mathcal{T}_{n-2}^{\phi_{n-1:n}\left(0 \right)} \mathcal{T}_{n-1}^{\phi_{n:n}\left(0 \right)}\bar{f}_n\left( x\right)  \eta^{\otimes \phi_{0:n}\left(0 \right) }\left(dx\right) \bigg{\lvert}\underset{M\rightarrow \infty}{\longrightarrow} 0.
\end{align*}
To conclude the proof, it is now sufficient to observe that by Theorem \ref{risultato_principale}, $\forall h \in \mathcal{B}_b\left( \mathbb{R}^d\right) $ $V_{1}^M\left(h\right)$ converges as $M$ goes to infinity  and its limit is given by $V_{1}\left(h\right)=\sigma^2_1(h)+\sigma^2_2(h) $.
\end{proof}

The proof of  Proposition  \ref{mainconv} is a direct consequence of the following Lemma.

\begin{lemma}\label{ricorrenza}
	Let $n\geq 0$ and $ k\geq 0$.  Given $h\in \mathcal{B}_b\left( \mathbb{R}^{d(k+1)} \right)$ one has 	
\begin{align*}
\left|\mathbb{E}\left(\bar\eta_{n+1}^{k,M}(h) \right)- \mathbb{E}\left( \bar\eta_{n}^{\phi_{{n}}(k),M}(\mathcal{T}_{n}^{k }h)\right) \right|\underset{M\rightarrow \infty}{\longrightarrow} 0.
\end{align*}
 
\end{lemma}
\begin{proof}[Proof of Lemma \ref{ricorrenza}]
	Since the $X_{n+1}^{M,m}$ are conditionally independent given $\mathcal{G}^{n+1}$ and $\mathcal{L}\left( X_{n+1}^{M,m}\mathrel{|}\mathcal{G}^{n+1} \right)\sim P_{n+1}\left(Y_{n+1}^{M,m},\cdot\right)$,
	 for each bounded measurable function $h:\mathbb{R}^{ d(k+1)}\rightarrow \mathbb{R}$ one has
	
	\begin{align} \label{lawg}
	&\mathbb{E}\left( h(X_{n+1}^{M,i} ,\cdots,X_{n+1}^{M,i+k})\mathrel{|}\mathcal{G}^{n+1} \right) = P^{\otimes (k+1)}_{n+1}h\left(Y_{n+1}^{M,i},\cdots,Y_{n+1}^{M,i+k}\right)\\
	&\phantom{=}:= \int_{\mathbb{R}^{d(k+1)}} h(x_0,\cdots,x_k)P_{n+1}\left(Y_{n+1}^{M,i},dx_0\right)\cdots P_{n+1}\left(Y_{n+1}^{M,i+k},dx_k\right).
	\end{align}
	
	Therefore for $M\geq k+1$

\begin{align}
		&\mathbb{E}\left(\bar\eta_{n+1}^{k,M}(h) \right)= \frac{1}{M}\sum\limits_{i=1}^{M-k}\mathbb{E}\left(h\left(X_{n+1}^{M,i},\cdots,X_{n+1}^{M,i+k}\right) \right) \\
		&= \frac{1}{M}\sum\limits_{i=1}^{M-k}\mathbb{E}\left(P^{\otimes (k+1)}_{n+1}h\left(Y_{n+1}^{M,i},\cdots,Y_{n+1}^{M,i+k}\right)\right) = \mathbb{E}\left(\tilde\eta_{n+1}^{k,M}(P^{\otimes (k+1)}_{n+1}h) \right). \label{from-mut-to-sel}
\end{align}

Thus our purpose now becomes to study the asymptotic behaviour of $\mathbb{E}\left(\tilde\eta_{n+1}^{k,M}(f) \right)$ for a given $f\in \mathcal{B}_b\left(\mathbb{R}^{d(k+1)} \right)$. 
 	Recalling the definition $\left(\ref{defYn}\right)$ of $Y_{n+1}^{M,m}$, let us observe that if we denote by $\ell_{m}$ the random index in $\left\lbrace 1,\cdots,M\right\rbrace $ such that $1_{\left\lbrace \sum\limits_{j=1}^{\ell_m-1}w_{n}^{M,j}<m-U_{n}^{m}\leq\sum\limits_{j=1}^{\ell_m}w_{n}^{M,j} \right\rbrace}=1$ so that $Y_{n+1}^{M,m} = X_{n}^{M,\ell_{m}}$, one has that $m\mapsto\ell_{m}$ is non decreasing.\\
Therefore given $k\geq 0$, $M\geq k+1$ and $1\leq i \leq M-k$ one has:

\begin{align*}
&\left( Y_{n+1}^{M,i},\cdots,Y_{n+1}^{M,i+k}\right)\\
&=\sum_{\ell_0=1}^{M}\cdots\sum_{\ell_{k}=\ell_{k-1}}^{M} \left( X_{n}^{M,\ell_0}, \cdots,X_{n}^{M,\ell_{k}} \right) \prod_{q=0}^{k}1_{\left\lbrace \sum\limits_{j=1}^{\ell_q-1}w_{n}^{M,j}<i+q-U_{n}^{i+q}\leq\sum\limits_{j=1}^{\ell_q}w_{n}^{M,j} \right\rbrace}.
\end{align*}
We can now apply the following change of variables $m=\ell_{0},s_1=\ell_{1}-\ell_{0},s_2= \ell_{2}-\ell_{0},\cdots,s_k =\ell_{k}-\ell_{0}$ and  set $s_0:= 0$ so that the above expression becomes
\begin{align*}
 \sum_{0\leq s_k\leq M-1}\sum_{0\leq s_1\leq s_2\leq\cdots \leq s_k}
\sum_{m=1}^{M-s_{k}} \left( X_{n}^{M,m},X_{n}^{M,m+s_{1}}, \cdots,X_{n}^{M,m+s_{k}} \right) \prod_{q=0}^{k} 1_{\left\lbrace \sum\limits_{j=1}^{m+s_{q}-1}w_{n}^{M,j}<i+q-U_{n}^{i+q}\leq\sum\limits_{j=1}^{m+s_{q}}w_{n}^{M,j} \right\rbrace}.
\end{align*}

	Therefore one has

\begin{align*}
\mathbb{E}\left(\tilde\eta_{n+1}^{k,M}(f) \right) &= \frac{1}{M}\sum\limits_{i=1}^{M-k}\mathbb{E}\left( f\left( Y_{n+1}^{M,i},\cdots,Y_{n+1}^{M,i+k}\right)\right) = \frac{1}{M}\sum\limits_{i=1}^{M-k}\mathbb{E}\left(\mathbb{E}\left( f\left( Y_{n+1}^{M,i},\cdots,Y_{n+1}^{M,i+k}\right)\mid \mathcal{F}^{n}\right)\right) \\
&	=\frac{1}{M}\sum_{0\leq s_k\leq M-1}\sum_{0\leq s_1\leq s_2\leq\cdots \leq s_k}
\sum_{m=1}^{M-s_{k}}\mathbb{E}\Bigg( f\left( X_{n}^{M,m},X_{n}^{M,m+s_{1}}, \cdots,X_{n}^{M,m+s_{k}} \right)\\
&\phantom{====================}\times\sum\limits_{i=1}^{M-k} \prod_{q=0}^{k}  \int_{i+q-1}^{i+q}1_{\left\lbrace \sum\limits_{j=1}^{m+s_{q}-1}w_{n}^{M,j}<u\leq\sum\limits_{j=1}^{m+s_{q}}w_{n}^{M,j} \right\rbrace}du \Bigg)
\end{align*}
where to obtain the last equality we apply the Freezing Lemma and the fact that the sequence $\left(U_{n}^{m} \right)_{1\leq m \leq M}$ is independent of $\mathcal{F}^{n}$ and the sequence $\left(X_{n}^{M,m} \right)_{1\leq m \leq M}$ is $\mathcal{F}^{n}$ - measurable.\\
Given $k \geq 0$, $M\geq k+1$, $0\leq s_k\leq M-1$, $0\leq s_1\leq s_2\leq\cdots \leq s_k$ and $1\leq m\leq M-s_{k}$, let us now focus on the sum over $i$ appearing in the above expression:
\begin{align*}
\sum\limits_{i=1}^{M-k} \prod_{q=0}^{k}  \int_{i+q-1}^{i+q}1_{\left\lbrace \sum\limits_{j=1}^{m+s_{q}-1}w_{n}^{M,j}<u\leq\sum\limits_{j=1}^{m+s_{q}}w_{n}^{M,j} \right\rbrace}du.
\end{align*}
Let us first observe that if the quantity inside the sum is different from zero, then in particular\\ $\int_{i-1}^{i}1_{\left\lbrace \sum\limits_{j=1}^{m-1}w_{n}^{M,j}<u\leq\sum\limits_{j=1}^{m}w_{n}^{M,j} \right\rbrace}du\neq 0$ and $\int_{i+k-1}^{i+k}1_{\left\lbrace \sum\limits_{j=1}^{m+s_{k}-1}w_{n}^{M,j}<u\leq\sum\limits_{j=1}^{m+s_{k}}w_{n}^{M,j} \right\rbrace}du\neq 0$. This  implies that $\sum\limits_{j=1}^{m}w_{n}^{M,j}>i-1$ and $\sum\limits_{j=1}^{m+s_k-1}w_{n}^{M,j}<i+k$. Thus

\begin{align*}
\dfrac{\underline{g}_{n}}{\bar{g}_{n}}\left(s_k-1 \right) 	\leq\sum\limits_{j=m+1}^{m+s_k-1}w_{n}^{M,j}= \sum\limits_{j=1}^{m+s_k-1}w_{n}^{M,j}-\sum\limits_{j=1}^{m}w_{n}^{M,j}<i+k - i+1=k+1
\end{align*}
and so we have obtained an upper bound for $s_k$: $s_k \leq \left\lceil \dfrac{\bar{g}_{n}}{\underline{g}_{n}}\left(k+1 \right)\right\rceil=\phi_{{n}}(k)$ where we use the notation introduced in (\ref{deffunzionen}).  \\
Moreover we observe that we can replace the finite sum over $i$ with an infinite sum: if $i\geq M-k+1$ $$\int_{i+k-1}^{i+k}1_{\left\lbrace \sum\limits_{j=1}^{m+s_{k}-1}w_{n}^{M,j}<u\leq\sum\limits_{j=1}^{m+s_{k}}w_{n}^{M,j} \right\rbrace}du=0$$ 
since $\sum\limits_{j=1}^{m+s_{k}}w_{n}^{M,j}\leq\sum\limits_{j=1}^{M}w_{n}^{M,j} =  M$. \\
Therefore after all one has 

\begin{align*}
\sum\limits_{i=1}^{M-k}\int_{i-1}^{i}& \prod_{q=0}^{k}  \int_{i+q-1}^{i+q}1_{\left\lbrace \sum\limits_{j=1}^{m+s_{q}-1}w_{n}^{M,j}<u\leq\sum\limits_{j=1}^{m+s_{q}}w_{n}^{M,j} \right\rbrace}du\\ &=1_{\left\lbrace s_k \leq \phi_{{n}}(k)\right\rbrace } \sum\limits_{i\geq 1} \prod_{q=1}^{k}  \int_{i+q-1}^{i+q}1_{\left\lbrace \sum\limits_{j=1}^{m+s_{q}-1}w_{n}^{M,j}<u\leq\sum\limits_{j=1}^{m+s_{q}}w_{n}^{M,j} \right\rbrace}du.
\end{align*}

We can now apply the change of variable $u:= u'+ \left\lfloor \sum\limits_{j=1}^{m-1}w_{n}^{M,j}\right \rfloor  $ in each of the above integrals so to obtain

\begin{align}\nonumber
&\sum\limits_{i\geq 1} \prod_{q=0}^{k}  \int_{i+q-1}^{i+q}1_{\left\lbrace \sum\limits_{j=1}^{m+s_{q}-1}w_{n}^{M,j}<u\leq\sum\limits_{j=1}^{m+s_{q}}w_{n}^{M,j} \right\rbrace}du\\\nonumber
& = \sum\limits_{i\geq 1} \prod_{q=0}^{k}  \int_{i+q-1-\left\lfloor \sum\limits_{j=1}^{m-1}w_{n}^{M,j}\right \rfloor}^{i+q-\left\lfloor \sum\limits_{j=1}^{m-1}w_{n}^{M,j}\right \rfloor}1_{\left\lbrace u_{n}^{M,m-1}+\sum\limits_{j=m}^{m+s_{q}-1}w_{n}^{M,j} <u'\leq u_{n}^{M,m-1}+\sum\limits_{j=m}^{m+s_{q}}w_{n}^{M,j} \right\rbrace}du'\\ \label{qqq}
&= \sum\limits_{i\geq 1} \prod_{q=0}^{k}  \int_{i+q-1}^{i+q}1_{\left\lbrace u_{n}^{M,m-1}+\sum\limits_{j=m}^{m+s_{q}-1}w_{n}^{M,j} <u'\leq u_{n}^{M,m-1}+\sum\limits_{j=m}^{m+s_{q}}w_{n}^{M,j} \right\rbrace}du'
\end{align}
where to obtain the last equality we use the fact that if $i< 1$, $\int_{i-1}^{i}1_{\left\lbrace u_{n}^{M,m-1}+\sum\limits_{j=m}^{m-1}w_{n}^{M,j} <u'\leq u_{n}^{M,m-1}+\sum\limits_{j=m}^{m}w_{n}^{M,j} \right\rbrace}du'=0$. We observe that if the quantity inside the sum is different from zero, then in particular $$\int_{i-1}^{i}1_{\left\lbrace u_{n}^{M,m-1} <u'\leq u_{n}^{M,m-1}+w_{n}^{M,m} \right\rbrace}du' \neq 0.$$
	This implies that $i-1<u_{n}^{M,m-1}+w_{n}^{M,m} \leq 1+\dfrac{\bar{g}_{n}}{\underline{g}_{n}}$ and so $i\leq \left\lceil 1+\dfrac{\bar{g}_{n}}{\underline{g}_{n}}\right\rceil$. In conclusion, (\ref{qqq}) can be rewritten as 
\begin{align}\label{qqq2}
\sum\limits_{i=1}^{\left\lceil 1+\bar{g}_{n} /\underline{g}_{n}\right\rceil} \prod_{q=0}^{k}  \int_{i+q-1}^{i+q}1_{\left\lbrace u_{n}^{M,m-1}+\sum\limits_{j=m}^{m+s_{q}-1}w_{n}^{M,j} <u'\leq u_{n}^{M,m-1}+\sum\limits_{j=m}^{m+s_{q}}w_{n}^{M,j} \right\rbrace}du'.
\end{align}

We denote (\ref{qqq2}) by $ \psi_{s_{1}:s_{k}}\left(u_{n}^{M,m-1},w_{n}^{M,m},w_{n}^{M,m+1},\cdots ,w_{n}^{M,m+s_{k}} \right)$ 
where $\psi_{s_{1}:s_{k}}:\mathbb{R}^{s_k+2}\rightarrow \mathbb{R}$.\\
Therefore if $M\geq \max\left(1+k,1+\phi_{{n}}(k) \right) $ we have obtained that 

\begin{align}
&\mathbb{E}\left(\tilde\eta_{n+1}^{k,M}(f) \right)\label{objective1}\\
&=\sum_{ s_k = 0}^{\phi_{{n}}(k)}
\sum_{0\leq s_1\leq s_2\leq\cdots \leq s_k} \frac{1}{M}\sum_{m=1}^{M-s_{k}}\mathbb{E}\Bigg( f\left( X_{n}^{M,m},X_{n}^{M,m+s_{1}}, \cdots,X_{n}^{M,m+s_{k}} \right)\label{firstpart}\\
&\phantom{===========}\times\psi_{s_{1}:s_{k}}\left(u_{n}^{M,m-1},w_{n}^{M,m},w_{n}^{M,m+1},\cdots ,w_{n}^{M,m+s_{k}} \right)\Bigg). \label{secondpart}
\end{align}

We can therefore apply Conjecture \ref{conj2} with $t=s_k$, $h\left(x_0,x_1,\cdots,x_{s_k} \right)= f\left(x_0,x_{s_1},x_{s_2},\cdots,x_{s_k} \right)$, $\psi\left(u,y_0,\cdots,y_{s_k} \right) = \psi_{s_{1}:s_{k}}\left(u,y_0,\cdots,y_{s_k}\right)$.

\begin{align}
\left| \mathbb{E}\left(\tilde\eta_{n+1}^{k,M}(f) \right) -\sum_{ s_{k} = 0}^{\phi_{{n}}(k)} \mathbb{E}\left( \bar\eta_{n}^{s_{k},M}(T_{n}^{k\rightarrow s_k }f)\right)\right|  \underset{M\rightarrow \infty}{\longrightarrow} 0 \label{finalstep}
\end{align}
where $T_{n}^{k\rightarrow s_k }:\mathcal{B}_b\left(\mathbb{R}^{d(k+1)} \right) \rightarrow \mathcal{B}_b\left(\mathbb{R}^{d(s_k+1)} \right) $ is defined by
 \begin{align*}
 	 &T_{n}^{k\rightarrow s_k }f\left(x_0,x_1,\cdots,x_{s_{k}} \right)=\sum_{0\leq s_1\leq s_2\leq\cdots \leq s_{k}} f\left( x_0,x_{s_1}, \cdots,x_{s_{k}} \right)\int_{0}^{1}\psi_{s_{1}:s_{k}}\left(u,\tilde{g}_{n}(x_{0}),\tilde{g}_{n}(x_{1}),\cdots ,\tilde{g}_{n}(x_{s_k})\right)du
 \end{align*}
 and where given $0\leq s_k \leq \phi_{{n}}(k)$ and $0\leq s_1\leq s_2\leq\cdots \leq s_k$:
 \begin{align*}
 	&\psi_{s_{1}:s_{k}}\left(u,y_{0},y_{1},\cdots,y_{s_k}\right) =\sum\limits_{i=1}^{\left\lceil 1+\bar{g}_{n} /\underline{g}_{n}\right\rceil}\prod_{q=0}^{k}  \int_{i+q-1}^{i+q}1_{\left\lbrace u+\sum\limits_{j=0}^{s_{q}-1}y_{j} <u'\leq u+\sum\limits_{j=0}^{s_{q}}y_{j} \right\rbrace}du'.
 \end{align*}
We now observe that by defining $T_{n}^{k}:\mathcal{B}_b\left(\mathbb{R}^{d(k+1)} \right)\rightarrow \mathcal{B}_b\left(\mathbb{R}^{d(\phi_{{n}}(k)+1)} \right)$  by  
\begin{align*}
	T_{n}^{k}f\left(x_0, \cdots, x_{\phi_{{n}}(k)} \right)  =\sum\limits_{ s_k = 0}^{\phi_{{n}}(k)} T_{n}^{k\rightarrow s_k }f\left(x_0,\cdots,x_{s_k} \right),
\end{align*} one has
\begin{align}\nonumber
&	\left|\sum_{ s_{k} = 0}^{\phi_{{n}}(k)}  \bar\eta_{n}^{s_{k},M}(T_{n}^{k\rightarrow s_k }f) -  \bar\eta_{n}^{\phi_{{n}}(k),M}(T_{n}^{k }f)\right|\\\nonumber
	&=\frac{1}{M}\left|\sum_{ s_{k} = 0}^{\phi_{{n}}(k)}\left(\sum_{i=1}^{M-s_k}\left( T_{n}^{k\rightarrow s_k }f\right)\left( X_{n}^{M,i},\cdots,X_{n}^{M,i+s_k}\right) -  \sum_{i=1}^{M-\phi_{{n}}(k)}\left( T_{n}^{k\rightarrow s_k }f\right)\left( X_{n}^{M,i},\cdots,X_{n}^{M,i+s_k}\right) \right) \right| \\\nonumber
	&\leq  \frac{1}{M}\sum_{ s_{k} = 0}^{\phi_{{n}}(k)}\left\| T_{n}^{k\rightarrow s_k }f\right\|_{\infty}  \left( \phi_{{n}}(k) - s_k\right)  \underset{M\rightarrow \infty}{\longrightarrow} 0. 
\end{align}
 Combining the above estimate with (\ref{finalstep}) we obtain that 
\begin{align}
	\left| \mathbb{E}\left(\tilde\eta_{n+1}^{k,M}(f) \right) -  \mathbb{E}\left( \bar\eta_{n}^{\phi_{{n}}(k),M}(T_{n}^{k }f)\right)\right| \underset{M\rightarrow \infty}{\longrightarrow} 0. \label{finalstepfin}
\end{align}
 
In conclusion, combining (\ref{from-mut-to-sel}) with (\ref{finalstepfin}) by taking $f=P^{\otimes (k+1)}_{n+1}h $, we obtain  that

\begin{align*}
		\left|\mathbb{E}\left(\bar\eta_{n+1}^{k,M}(h) \right)- \mathbb{E}\left( \bar\eta_{n}^{\phi_{{n}}(k),M}(\mathcal{T}_{n}^{k }h)\right) \right|\underset{M\rightarrow \infty}{\longrightarrow} 0 
\end{align*}
where $\mathcal{T}_{n}^{k } = T_{n}^{k }P^{\otimes (k+1)}_{n+1}$
and this concludes the proof.
\end{proof}

\section{Numerical Results}\label{numericalresults}
In what follows we fix $d=1$ and given a sequence $\left(W_n \right)_{n\geq 1} $ of independent real-valued random variables distributed according to the uniform law on $\left(0,1\right)$, let  $Z_{n+1}=Z_n+W_{n+1}$ for $n\geq 0$ with  $Z_0$ distributed according to the uniform law on $\left(0,1\right)$ $\left(\eta(dx) = 1_{\left[ 0,1\right]  }\left(x\right)dx \right)$. Thus in this case the transition kernel is given by $P(x,dy)=1_{\left[x,x+1 \right] }(y)dy$. Moreover we will fix $g_n \left(x \right) = f\left(x \right) = e^x$ $\forall n\geq 0$, $\forall x\in \mathbb{R}$.\\

In this section, we are first going to numerically verify the Conjecture \ref{conj1} in the case $n=1$ and the Conjecture \ref{conj2} in the case $n=1$ and $n=2$.\\
In the second place, we will test the two conjectures together  by directly studying the asymptotic variance. As done in the theory, we study numerically the asymptotic behaviour of following expression:

\begin{align}
	V_{n+1}^M\left(f\right) = \textrm{Var}\left( \sqrt{M}\dfrac{\bar{\eta}^{0,M}_{n}\left(g_nf \right)}{\bar{\eta}^{0,M}_{n}\left(g_n\right)} \right)+\mathbb{E}\left( \textrm{Var}\left(\dfrac{1}{\sqrt{M}}\sum_{m=1}^{M}f(Y_{n+1}^{M,m}) \mathrel{\Big|}\mathcal{F}^{n}\right)\right). \label{splitnum}
\end{align}
We will consider separately the case  $n=0$ and the case $n=1$. We recall that the first case has been fully studied in Section \ref{secasymptoticvariance} without the need to introduce any conjecture.
\subsection{Verification of the Conjectures}
\subsubsection{Conjecture \ref{conj1}}
We recall the notation
\begin{align*}
f_1:=g_1\left( \bar{\eta}_1^0\left(g_1\right)f-\bar{\eta}_1^0\left(g_1f\right)\right).
\end{align*}

Since we have fixed $g_n \left(x \right) = f\left( x\right) = e^x$ $\forall n\geq 0$, $\forall x\in \mathbb{R}$, one has $\bar{\eta}_0^0\left(g_0\right)=e-1$, $\bar{\eta}_1^0\left(g_1\right)=\frac{e^2-1}{2}$,  $\bar{\eta}_1^0\left(g_1f\right)= \frac{1}{6}(e^3-1)(e+1)$. Moreover $Pf_1(x)=\bar{\eta}_1^0\left(g_1\right)\frac{e^{2(x+1)}-e^{2x}}{2}-\bar{\eta}_1^0\left(fg_1\right)\left(e^{x+1}-e^x \right) $ and  $Pf_1^2(x)=\left( \bar{\eta}_1^0\left(g_1 \right) \right)^2 \frac{e^{4(x+1)}-e^{4x}}{4} +\left( \bar{\eta}_1^0\left(fg_1\right)\right)^2\frac{e^{2(x+1)}-e^{2x}}{2} - 2 \bar{\eta}_1^0\left(g_1\right)\bar{\eta}_1^0\left(fg_1\right) \frac{e^{3(x+1)}-e^{3x}}{3}$ and
\begin{align*}
\bar{\eta}_{0}^0\left(g_{0}\left( P f_1^2-\left(P f_1 \right)^2 \right) \right) =& \left( \bar{\eta}_1^0\left(g_1\right)\right) ^2\left(\frac{e^5-1}{5} \right) \left(\frac{e^2-1}{2} \right)+ \left( \bar{\eta}_1^0\left(g_1f\right)\right)^2\left(\frac{e^3-1}{3} \right)\left(\frac{e^2-1}{2}-\left( e-1\right)^2  \right)\\
&- \bar{\eta}_1^0\left(g_1\right) \bar{\eta}_1^0\left(g_1f\right)\left(\frac{e^4-1}{4} \right) \left(\frac{2}{3}\left(e^3-1 \right)-\left(e^2-1 \right) \left(e-1 \right)   \right). 
\end{align*}    \\
We first observe that given a sequence $\left(T_i \right)_{i\geq 1}$ of square integrable i.i.d. random variables, by using the delta method it is possible to prove that as $M$ goes to infinity the following convergence in distribution holds:
\begin{align}\label{conf_int_var}
	\sqrt{M}\left(\left( \frac{1}{M}\sum\limits_{ i = 1}^{M}T_i^2 - \left( \frac{1}{M}\sum\limits_{ i=1}^{M}T_i\right)^2 \right) - \textrm{Var}\left(T_1 \right) \right) \overset{d}{\Longrightarrow}\mathcal{N}\left(0, \textrm{Var}\left(\left(T_1 - \mathbb{E}\left(T_1 \right)  \right)^2 \right)  \right) .
\end{align}
The general strategy will be the following: we fix $M=10000$ and

\begin{enumerate}
\item we simulate $n_1=10^7$ independent samples $T_j$ of $\sqrt{M}\dfrac{\bar{\eta}^{0,M}_{1}\left(g_1f \right)}{\bar{\eta}^{0,M}_{1}\left(g_1 \right)}= \sqrt{M}\frac{\sum\limits_{i=1}^{M}\left( g_1f\right)\left(  X_1^{M,i}\right) }{\sum\limits_{i=1}^{M}g_1\left(  X_1^{M,i}\right)}$    and we compute, by using (\ref{conf_int_var}), the estimator

\begin{align} 
\hat{v}_{1}:= \frac{1}{n_1}\sum\limits_{ i = 1}^{n_1}T_i^2 - \left( \frac{1}{n_1}\sum\limits_{ i=1}^{n_1}T_i\right)^2 \label{estimvar}
\end{align}

of  $\textrm{Var}\left( \sqrt{M}\dfrac{\bar{\eta}^{0,M}_{1}\left(g_1f \right)}{\bar{\eta}^{0,M}_{1}\left(g_1 \right)} \right)$ with relative $95\%$ confidence interval $\left[ a_{1}, b_{1}\right] $  given by 

\begin{align}
&a_1 :=  \hat{v}_{1} -\frac{1.96}{\sqrt{n_1}} \sqrt{\frac{1}{n_1}\sum\limits_{i=1}^{n_1}\left(T_i-\frac{1}{n_1}\sum\limits_{ j = 1}^{n_1} T_j\right)^4 - \left( \frac{1}{n_1}\sum\limits_{i=1}^{n_1}\left(T_i-\frac{1}{n_1}\sum\limits_{ j = 1}^{n_1} T_j\right)^2\right) ^2 }\label{civar1}\\
&b_1 := \hat{v}_{1} +\frac{1.96}{\sqrt{n_1}} \sqrt{\frac{1}{n_1}\sum\limits_{i=1}^{n_1}\left(T_i-\frac{1}{n_1}\sum\limits_{ j = 1}^{n_1} T_j\right)^4 - \left( \frac{1}{n_1}\sum\limits_{i=1}^{n_1}\left(T_i-\frac{1}{n_1}\sum\limits_{ j = 1}^{n_1} T_j\right)^2\right) ^2 }\label{civar2}.
\end{align}
\item we simulate $n_1=10^7$ independent samples $H_j$ of $\dfrac{1}{\sqrt{M}}\sum\limits_{m=1}^{M}Pf_1(Y_1^{M,m})$  and we compute the estimator $\hat{v}_{2}$
of  $V_{1}^M\left(Pf_1\right) = \textrm{Var}\left(\dfrac{1}{\sqrt{M}}\sum\limits_{m=1}^{M}Pf_1(Y_1^{M,m})\right)$ with relative $95\%$ confidence interval $\left[ a_{2}, b_{2}\right] $  with $\hat{v}_{2},\,a_{2},$ and $b_{2}$ respectively defined as in  (\ref{estimvar}), (\ref{civar1}) and (\ref{civar2}).\\
 Therefore 
 \begin{align*}
 	\frac{\hat{v}_{2}}{\left( \bar{\eta}_1^0\left(g_1\right)\right)^4}+\frac{\bar{\eta}_{0}^0\left(g_{0}\left( P_1 f_1^2-\left(P f_1 \right)^2 \right) \right) }{\left( \bar{\eta}_1^0\left(g_1\right)\right)^4\bar{\eta}_{0}^0\left(g_{0}\right)} 
 \end{align*}
 is an estimator of $\frac{V_{1}^M\left(Pf_1\right)}{\left( \bar{\eta}_1^0\left(g_1\right)\right) ^4 }+\frac{\bar{\eta}_{0}^0\left(g_{0}\left( P_1 f_1^2-\left(P f_1 \right)^2 \right) \right) }{\left( \bar{\eta}_1^0\left(g_1\right)\right)^4\bar{\eta}_{0}^0\left(g_{0}\right)} $ with relative $95\%$ confidence interval 
 \begin{align*}
 	 \left[ \frac{a_{2}}{\left( \bar{\eta}_1^0\left(g_1\right)\right) ^4}+\frac{\bar{\eta}_{0}^0\left(g_{0}\left( P f_1^2-\left(Pf_1 \right)^2 \right) \right) }{\left( \bar{\eta}_1^0\left(g_1\right)\right)^4\bar{\eta}_{0}^0\left(g_{0}\right)}, \frac{b_{2}}{\left( \bar{\eta}_1^0\left(g_1\right)\right) ^4}+\frac{\bar{\eta}_{0}^0\left(g_{0}\left( P_1 f_1^2-\left(P f_1 \right)^2 \right) \right) }{\left( \bar{\eta}_1^0\left(g_1\right)\right)^4\bar{\eta}_{0}^0\left(g_{0}\right)}\right].
 \end{align*}
 \item we check that $\hat{v}_{1}$ is close to $\frac{\hat{v}_{2}}{\left( \bar{\eta}_1^0\left(g_1\right)\right)^4}+\frac{\bar{\eta}_{0}^0\left(g_{0}\left( P_1 f_1^2-\left(P f_1 \right)^2 \right) \right) }{\left( \bar{\eta}_1^0\left(g_1\right)\right)^4\bar{\eta}_{0}^0\left(g_{0}\right)}$ as expected
\end{enumerate}
\subsubsection*{RESULTS}
\begin{center}	
	\begin{tabular}{|c|c|c|}
		\hline
		n=1	& value & CI (95\%) \\
		\hline
		$\hat{v}_{1}$	&  2.8021446 & [ 2.7775619 , 2.8267273 ] \\
		\hline
		$\frac{\hat{v}_{2}}{\left( \bar{\eta}_1^0\left(g_1\right)\right)^4}+\frac{\bar{\eta}_{0}^0\left(g_{0}\left( P_1 f_1^2-\left(P f_1 \right)^2 \right) \right) }{\left( \bar{\eta}_1^0\left(g_1\right)\right)^4\bar{\eta}_{0}^0\left(g_{0}\right)}$	&  2.7932862
		& [ 2.7831461 , 2.8034263 ]\\
		\hline
	\end{tabular}
\end{center}	

\subsubsection{Conjecture \ref{conj2}}
Let $ n \in \left\lbrace  1,2\right\rbrace$ and $t\in \left\lbrace  1,2\right\rbrace $. Moreover let $h\left(x_0,\cdots,x_t \right)=  x_0+\cdots+x_t $   and   $\psi\left(u_0,w_0,\cdots,w_{t+1} \right)=  u_0+w_0+\cdots+w_{t+1}$ . \\
The general strategy will be the following: we fix $M=10000$ and 
\begin{enumerate}
	\item we simulate $n_1=10^5$ independent samples $T_j$ of $\frac{1}{M}\sum\limits_{m=1}^{M-t} h\left(X_n^{M,m},\cdots, X_n^{M,m+t} \right)\psi\left(u_{n}^{M,m-1},w_n^{M,m},\cdots,w_n^{M,m+t} \right)$
and we compute the standard estimator 
	\begin{align}\label{meanest}
	\hat{v}_{1}:= \frac{1}{n_1} \sum\limits_{j=1}^{n_1}T_j
	\end{align}
		of $\frac{1}{M}\sum\limits_{m=1}^{M-t} \mathbb{E}\left( h\left(X_n^{M,m},\cdots, X_n^{M,m+t} \right)\psi\left(u_{n}^{M,m-1},w_n^{M,m},\cdots,w_n^{M,m+t} \right)\right) $ with relative $95\%$ confidence interval $\left[ a_{1}, b_{1}\right] $  given by 
		
	\begin{align}\label{cimean1}
	&a_{1}:=\hat{v}_{1} - \frac{1.96}{\sqrt{n_1}}*\sqrt{\frac{1}{n_1} \sum\limits_{j=1}^{n_1}\left(T_j-\hat{v}_{1}  \right)^2 }\\\label{cimean2}
	& b_{1}:= \hat{v}_{1} + \frac{1.96}{\sqrt{n_1}}*\sqrt{\frac{1}{n_1} \sum\limits_{j=1}^{n_1}\left(T_j-\hat{v}_{1}  \right)^2 }
	\end{align}
	\item we simulate  $n_1=10^5$ independent samples $H_j$ of
	\begin{align*}
		\frac{1}{M}\sum_{m=1}^{M-t} h\left(X_n^{M,m},\cdots, X_n^{M,m+t} \right)\psi\left(U,\tilde{g}_n(X_n^{M,m}),\cdots,\tilde{g}_n(X_n^{M,m+t})\right) 
	\end{align*} 
	where $U$ is a random variable uniformly distributed on $\left(0,1\right)$ independent of $\mathcal{F}^n$ and $\tilde{g}_n= \frac{g_n}{\eta_{n}^0\left( g_n\right) }$ $\left(\eta_{1}^0\left( g_1\right)= \frac{e^2-1}{2},\eta_{2}^0\left( g_2\right)= \frac{e^3-1}{3}\right)$. We then compute the standard estimator $\hat{v}_{2}$ of 
		\begin{align*}
	\frac{1}{M}\sum_{m=1}^{M-t} \mathbb{E}\left( h\left(X_n^{M,m},\cdots, X_n^{M,m+t} \right)\psi\left(U,\tilde{g}_n(X_n^{M,m}),\cdots,\tilde{g}_n(X_n^{M,m+t})\right) \right) 
	\end{align*} 
	
	with relative $95\%$ confidence interval $\left[ a_{2}, b_{2}\right] $  with $\hat{v}_{2}, a_{2}, b_{2}$ respectively defined as in (\ref{meanest}), (\ref{cimean1}) and (\ref{cimean2}).
	\item we check that $\hat{v}_{1}$ is close to $\hat{v}_{2}$ as expected
\end{enumerate}
\subsubsection*{RESULTS}
\subsubsection*{n=1}
 \begin{minipage}[c]{0.5\textwidth}

		\begin{tabular}{|c|c|c|}
			\hline
			t=1	& value & CI (95\%) \\
			\hline
			$\hat{v}_{1}$	&  5.7510732 
			& [ 5.7509322 , 5.7512143 ] \\
			\hline
			$\hat{v}_{2}$	& 5.7509738  
			& [ 5.7470892 , 5.7548583 ] \\
			\hline
		\end{tabular}
		
	\end{minipage}
	\begin{minipage}[c]{0.5\textwidth}
		
		\begin{tabular}{|c|c|c|}
			\hline
			t=2	& value & CI (95\%) \\
			\hline
			$\hat{v}_{1}$	&  11.8853516
			& [ 11.8850608 , 11.8856425 ]  
			\\
			\hline
			$\hat{v}_{2}$	&  11.8835799 
			&  [ 11.8777417 , 11.8894182 ]\\
			\hline
		\end{tabular}
	\end{minipage}\\

\subsubsection*{n=2}

 \begin{minipage}[c]{0.5\textwidth}

	\begin{tabular}{|c|c|c|}
		\hline
		t=1	& value & CI (95\%) \\
		\hline
		$\hat{v}_{1}$	&  9.2154319
		&  [ 9.2152408 , 9.215623 ]   \\
		\hline
		$\hat{v}_{2}$	&  9.2150201
		& [ 9.2087834 , 9.2212568 ]    \\
		\hline
	\end{tabular}
	
\end{minipage}
\begin{minipage}[c]{0.5\textwidth}
	
	\begin{tabular}{|c|c|c|}
		\hline
		t=2	& value & CI (95\%) \\
		\hline
		$\hat{v}_{1}$	&  19.0901421
		&  [ 19.0897513 , 19.0905329 ]\\
		\hline
		$\hat{v}_{2}$	& 19.0895089  
		& [ 19.0801501 , 19.0988676 ] \\
		\hline
	\end{tabular}
\end{minipage}\\

\subsection{Focus on the Variance}

\subsubsection{case $n=0$}
We recall the notation $\bar\eta_0^0(h):=\eta(h)$ $\forall h\in\mathcal{B}_b\left(\mathbb{R} \right)$ and 
$f_0:=g_0\left( \bar{\eta}_0^0\left(g_0\right)f-\bar{\eta}_0^0\left(g_0f\right)\right).$
 \\
By (\ref{purpose}) and by observing that $\textrm{Var}\left( \sqrt{M}\dfrac{\bar{\eta}^{0,M}_{0}\left(g_0f \right)}{\bar{\eta}^{0,M}_{0}\left(g_0\right)} \right)=\textrm{Var}\left( \sqrt{M}\dfrac{\bar{\eta}^{0,M}_{0}\left(\frac{g_0}{\eta\left(g_0 \right)}f \right)}{\bar{\eta}^{0,M}_{0}\left(\frac{g_0}{\eta\left(g_0 \right)}\right)} \right)$, one has
\begin{align*}
	\left|\textrm{Var}\left( \sqrt{M}\dfrac{\bar{\eta}^{0,M}_{0}\left(g_0f \right)}{\bar{\eta}^{0,M}_{0}\left(g_0\right)} \right) -\frac{\eta\left(f_0^2 \right)}{\eta\left(g_0 \right)^4}  \right|\underset{M\rightarrow \infty}{\longrightarrow} 0  
\end{align*}
where with the choices made $\bar{\eta}_0^0\left(g_0\right)=  e-1$, $\bar{\eta}_0^0\left(g_0f\right)= \frac{e^2-1}{2}$,   $\eta\left(f_0^2 \right)= \left( \bar{\eta}_0^0\left(g_0\right)\right)^2 \frac{e^4-1}{4}+ \left(\bar{\eta}_0^0\left(g_0f\right) \right)^2\frac{e^2-1}{2}-2\bar{\eta}_0^0\left(g_0\right)\bar{\eta}_0^0\left(g_0f\right)\frac{e^3-1}{3} $.

The application of Theorem \ref{mainresult} with $g = \tilde{g}_0 = \frac{g_0}{\bar{\eta}_0^0\left(g_0\right)}$ gives
\begin{align*}
\left|\mathbb{E}\left( \textrm{Var}\left(\dfrac{1}{\sqrt{M}}\sum_{m=1}^{M}f(Y_1^{M,m}) \mathrel{\Big|}\mathcal{F}^{0}\right)\right) - \sum\limits_{k= 0}^{\phi_{0}(0)}\mathbb{E}\left(F_{k}\right) \right|  \underset{M\rightarrow \infty}{\longrightarrow} 0 
\end{align*}

	with $\left( F_{k}\right)_{k\in \mathbb{N}} $  given by 
\begin{equation}
F_{k} = \begin{cases}
f^2(X_0^{M,1})\beta_{0}\left( U_1,\tilde{g}_0(X_0^{M,1})\right) &\,k=0\\
-f(X_0^{M,1})f(X_0^{M,k+1})\beta_{1}\left( U_1,\tilde{g}_0(X_0^{M,1}),\sum\limits_{\ell=2}^{k}\tilde{g}_0(X_0^{M,\ell}),\tilde{g}_0(X_0^{M,k+1})\right) &\, k>0
\end{cases}
\end{equation}  
where $U_1\sim \mathcal{U}(0,1)$ is independent of $X_0^{M,1},\cdots,X_0^{M,k+1}$ and $\beta_{0}$ and $\beta_{1}$ are respectively defined in (\ref{beta0}) and (\ref{beta1}). By using that when $U_1$ is uniformly distributed  on $\left(0,1 \right) $, $\left\lbrace U_1 + r\right\rbrace $ is uniformly distributed on $\left(0,1 \right) $ for each $r\geq 0$, we can apply the Freezing Lemma to rewrite $\mathbb{E}\left(F_{k}\right)$ $\forall k=0,\cdots,\phi_{0}(0)$ :

\begin{align}\label{definizione_sigma2}
\sum\limits_{k= 0}^{\phi_{0}(0)}\mathbb{E}\left(F_{k}\right) = \sum\limits_{k= 0}^{\phi_{0}(0)}\mathbb{E}\left(\phi_{k,\tilde{g}_0}\left(X_0^{M,1},\cdots,X_0^{M,k+1} \right) \right)
\end{align}

for measurable functions $\phi_{k,\tilde{g}_0}:\mathbb{R}^{\left( k+1\right) }\rightarrow \mathbb{R}$ given by 
\begin{align}
\phi_{0,\tilde{g}_0}(x)= \frac{1}{3}\left( f^2( x)\left( 1-1_{\left\lbrace \tilde{g}_0(x)<1\right\rbrace }\left(  1-\tilde{g}_0(x)\right) ^3 \right)\right)\label{phi0}
\end{align}
and for each $k\geq1$

\begin{align}
&\phi_{k,\tilde{g}_0}(x)= \medmath{\frac{1}{2}\left(f( x_{1})f( x_{k+1})1_{\left\lbrace \sum\limits_{i=2}^{k}\tilde{g}_0(x_i)<1\right\rbrace }1_{\left\lbrace \sum\limits_{i=1}^{k+1}\tilde{g}_0(x_i)<1\right\rbrace }\tilde{g}_0(x_1)\tilde{g}_0(x_{k+1})\left( 2-2\sum\limits_{i=2}^{k}\tilde{g}_0(x_i)-(\tilde{g}_0(x_1)+\tilde{g}_0(x_{k+1}))\right)\right)}\nonumber \\
&\phantom{=}\medmath{+\frac{1}{6}\left( f( x_{1})f( x_{k+1})\cdot 1_{\left\lbrace \sum\limits_{i=2}^{k}\tilde{g}_0(x_i)<1\right\rbrace } 1_{\left\lbrace \sum\limits_{i=1}^{k+1}\tilde{g}_0(x_i)\geq  1\right\rbrace }\left(1-\sum\limits_{i=2}^{k}\tilde{g}_0(x_i) \right)^3\right)} \nonumber\\
&\phantom{=}\medmath{-\frac{1}{6}\left( f( x_{1})f(x_{k_1+1})\cdot 1_{\left\lbrace \sum\limits_{i=2}^{k}\tilde{g}_0(x_i)<1\right\rbrace } 1_{\left\lbrace \sum\limits_{i=1}^{k+1}\tilde{g}_0(x_i)\geq  1\right\rbrace }\left( \left(1-\sum\limits_{i=1}^{k}\tilde{g}_0(x_i) \right)^3 1_{\left\lbrace \sum\limits_{i=1}^{k}\tilde{g}_0(x_i)<1\right\rbrace }+\left(1-\sum\limits_{i=2}^{k+1}\tilde{g}_0(x_i) \right)^3 1_{\left\lbrace \sum\limits_{i=2}^{k+1}\tilde{g}_0(x_i)<1\right\rbrace }\right)\right)}.\label{phik}
\end{align}

Thus using numerical methods we are going to check what we already know theoretically that is

\begin{align*}
	\left|V_{1}^M\left(f\right) - \frac{\eta\left(f_0^2 \right)}{\eta\left(g_0 \right)^4} -\sum\limits_{k= 0}^{\phi_{0}(0)}\mathbb{E}\left(\phi_{k,\tilde{g}_0}\left(X_0^{M,1},\cdots,X_0^{M,k+1} \right) \right) \right| \underset{M\rightarrow \infty}{\longrightarrow} 0 .
\end{align*}

The general strategy will be the following: we fix $M=10000$ and 
\begin{enumerate}
		\item we simulate $n_1=10^7$ independent samples $T_j$ of $\dfrac{1}{\sqrt{M}}\sum\limits_{m=1}^{M}f(Y_1^{M,m})$  and we compute the estimator $\hat{v}_{1}$ 
		 of  $V_{1}^M\left(f\right) = \textrm{Var}\left(\dfrac{1}{\sqrt{M}}\sum\limits_{m=1}^{M}f(Y_1^{M,m})\right)$ with relative $95\%$ confidence interval $\left[ a_{1}, b_{1}\right] $ with $\hat{v}_{1}, a_{1}, b_{1}$ respectively defined as in (\ref{estimvar}), (\ref{civar1}) and (\ref{civar2}).

	Therefore $\hat{v}_{1}-\frac{\eta\left(f_0^2 \right)}{\eta\left(g_0 \right)^4} $ is an estimator of $V_{1}^M\left(f\right) - \frac{\eta\left(f_0^2 \right)}{\eta\left(g_0 \right)^4}$ with relative $95\%$ confidence interval $\left[ a_{1}-\frac{\eta\left(f_0^2 \right)}{\eta\left(g_0 \right)^4}, b_{1}-\frac{\eta\left(f_0^2 \right)}{\eta\left(g_0 \right)^4}\right]$.
	\item We simulate $n_2= 10^5$ independent samples $Z_j$  of   $\sum\limits_{k= 0}^{\phi_{0}(0)}\phi_{k,\tilde{g}_0}\left(X_0^{M,1},\cdots,X_0^{M,k+1} \right)$ and we compute the standard estimator 	$\hat{v}_{2}$ 
	of $\sum\limits_{k= 0}^{\phi_{0}(0)}\mathbb{E}\left(\phi_{k,\tilde{g}_0}\left(X_0^{M,1},\cdots,X_0^{M,k+1} \right) \right) $  with relative $95\%$ confidence interval $\left[ a_{2}, b_{2}\right] $  with $\hat{v}_{2}, a_{2}, b_{2}$ respectively defined as in (\ref{meanest}), (\ref{cimean1}) and (\ref{cimean2}).
	\item We check that $\hat{v}_{1}-\frac{\eta\left(f_0^2 \right)}{\eta\left(g_0 \right)^4} $ is close to $\hat{v}_{2}$ as expected.
\end{enumerate}	
\subsubsection*{RESULTS}
\begin{center}
	\begin{tabular}{|c|c|c|}
		\hline
		n=0	& value & CI (95\%) \\
		\hline
		$\hat{v}_{1}-\frac{\eta\left(f_0^2 \right)}{\eta\left(g_0 \right)^4}$	& 0.07943  & [ 0.079127 , 0.079733 ] \\
		\hline
		$\hat{v}_{2}$	& 0.0793412
		& [ 0.0790773 , 0.0796051 ] \\
		\hline
	\end{tabular}
\end{center}
\subsubsection{case $n=1$}

In this case the asymptotic behaviour of the first term of the right-hand side  of (\ref{splitnum}) is given by Conjecture \ref{conj1}

\begin{align*}
	\left|\textrm{Var}\left( \sqrt{M}\dfrac{\bar{\eta}^{0,M}_{1}\left(g_1f \right)}{\bar{\eta}^{0,M}_{1}\left(g_1 \right)} \right) - \frac{V_{1}^M\left(Pf_1\right)}{\left( \bar{\eta}_1^0\left(g_1\right)\right) ^4 }-\frac{\bar{\eta}_{0}^0\left(g_{0}\left( P f_1^2-\left(P f_1 \right)^2 \right) \right) }{\left( \bar{\eta}_1^0\left(g_1\right)\right)^4\bar{\eta}_{0}^0\left(g_0\right)}\right|   \underset{M\rightarrow \infty}{\longrightarrow} 0 
\end{align*}

and by (\ref{conv1}) the asymptotic behaviour of the second term of the right-hand side  of (\ref{splitnum}) is given by

\begin{align}
&\left| \mathbb{E}\left( \textrm{Var}\left(\dfrac{1}{\sqrt{M}}\sum_{m=1}^{M}f(Y_{2}^{M,m}) \mathrel{\Big|}\mathcal{F}^{1}\right)\right) -\sum_{k=0}^{\phi_{1}(0)}
\mathbb{E}\left(\bar{\eta}_1^{k,M} \left( \tilde{f}_k\right) \right)  \right| \underset{M\rightarrow \infty}{\longrightarrow} 0 
\end{align}

where for $k=0,\cdots,\phi_{1}(0)$,
\begin{align*}
	\tilde{f}_k\left(x_0,\cdots,x_k \right)= f\left(x_0 \right)f\left(x_k \right) \int_{0}^{1}\bar\beta_k\left(u,\tilde{g}_{1}\left( x_0 \right) ,\cdots,\tilde{g}_{1}\left( x_k \right) \right) du
\end{align*}
 
 with $\bar\beta_k\left(u,y_0,\cdots,y_k \right): = \beta_{0}(u,y_0)1_{\left\lbrace k=0\right\rbrace }- \beta_{1}\left( u,y_0,\sum\limits_{\ell=1}^{k-1}y_{\ell},y_k\right) 1_{\left\lbrace k\neq 0\right\rbrace }$.
We recall that $\beta_{0}$ and $\beta_{1}$ are respectively defined in (\ref{beta0}) and (\ref{beta1}) and  $\tilde{g}_1= \frac{g_1}{\bar{\eta}_1^0\left(g_1\right) } $ .\\
We now observe that 
\begin{align}
&\sum_{k=0}^{\phi_{1}(0)}
\mathbb{E}\left(\bar{\eta}_1^{k,M} \left( \tilde{f}_k\right) \right)  = \sum_{k=0}^{\phi_{1}(0)} \mathbb{E}\left(\bar{\eta}_1^{k,M}\left( \phi_{k,\tilde{g}_1}\right) \right) \label{rew_rhs}
\end{align}
where the $\phi_{k,\tilde{g}_1}:\mathbb{R}^{(k+1)}\rightarrow\mathbb{R}$ are measurable functions defined in (\ref{phi0}) and (\ref{phik}).\\ 

Thus using numerical methods we are going to check that 
\begin{align*}
	\left|V_{2}^M\left(f\right)- \frac{V_{1}^M\left(Pf_1\right)}{\left( \bar{\eta}_1^0\left(g_1\right)\right) ^4 }-\frac{\bar{\eta}_{0}^0\left(g_{0}\left( P f_1^2-\left(P f_1 \right)^2 \right) \right) }{\left( \bar{\eta}_1^0\left(g_1\right)\right)^4\bar{\eta}_{0}^0\left(g_0\right)}-\sum_{k=0}^{\phi_{1}(0)} \mathbb{E}\left(\bar{\eta}_1^{k,M}\left( \phi_{k,\tilde{g}_1}\right) \right) \right|  \underset{M\rightarrow \infty}{\longrightarrow} 0
\end{align*}

so that Conjecture $1$ and Conjecture $2$ are numerically verified at the same time.\\

The general strategy will be the following:
\begin{enumerate}

	\item 	\begin{enumerate}
	\item we simulate $n_1=10^7$ independent samples $T_j$ of $\dfrac{1}{\sqrt{M}}\sum\limits_{m=1}^{M}f(Y_2^{M,m})$  and we compute the estimator $\hat{v}_{1,1}$
	of  $\textrm{Var}\left(\dfrac{1}{\sqrt{M}}\sum\limits_{m=1}^{M}f(Y_2^{M,m})\right)$ with relative $95\%$ confidence interval $\left[ a_{1,1}, b_{1,1}\right] $  with $\hat{v}_{1,1},a_{1,1},$ and $b_{1,1}$ respectively defined as in  (\ref{estimvar}), (\ref{civar1}) and (\ref{civar2}).
	\item we simulate $n_1=10^7$ independent samples $H_j$ of $\dfrac{1}{\sqrt{M}}\sum\limits_{m=1}^{M}Pf_1(Y_1^{M,m})$ independent of $\left( T_j\right)_{1\leq j \leq n_1}$. We compute the estimator $\hat{v}_{1,2}$
	of  $\textrm{Var}\left(\dfrac{1}{\sqrt{M}}\sum\limits_{m=1}^{M}Pf_1(Y_1^{M,m})\right)$ with relative $95\%$ confidence interval $\left[ a_{1,2}, b_{1,2}\right] $  with $\hat{v}_{1,2},\,a_{1,2},$ and $b_{1,2}$ respectively defined as in  (\ref{estimvar}), (\ref{civar1}) and (\ref{civar2}).
	\item we compute the estimator $\hat{v}_{1,1} - \frac{\hat{v}_{1,2}}{\left( \bar{\eta}_1^0\left(g_1\right)\right) ^4}-\frac{\bar{\eta}_{0}^0\left(g_{0}\left( P f_1^2-\left(P f_1 \right)^2 \right) \right) }{\left( \bar{\eta}_1^0\left(g_1\right)\right)^4\bar{\eta}_{0}^0\left(g_0\right)}$ of  
	\begin{align*}
V_{2}^M\left(f\right)- \frac{V_{1}^M\left(Pf_1\right)}{\left( \bar{\eta}_1^0\left(g_1\right)\right) ^4 }-\frac{\bar{\eta}_{0}^0\left(g_{0}\left( P f_1^2-\left(P f_1 \right)^2 \right) \right) }{\left( \bar{\eta}_1^0\left(g_1\right)\right)^4\bar{\eta}_{0}^0\left(g_0\right)}
	\end{align*}
	with $90\%$ confidence interval 
	\begin{align*}
		\left[ a_{1,1}-\frac{b_{1,2}}{\left( \bar{\eta}_1^0\left(g_1\right)\right) ^4 }-\frac{\bar{\eta}_{0}^0\left(g_{0}\left( P f_1^2-\left(P f_1 \right)^2 \right) \right) }{\left( \bar{\eta}_1^0\left(g_1\right)\right)^4\bar{\eta}_{0}^0\left(g_0\right)},b_{1,1}- \frac{a_{1,2}}{\left( \bar{\eta}_1^0\left(g_1\right)\right) ^4 }-\frac{\bar{\eta}_{0}^0\left(g_{0}\left( P f_1^2-\left(P f_1 \right)^2 \right) \right) }{\left( \bar{\eta}_1^0\left(g_1\right)\right)^4\bar{\eta}_{0}^0\left(g_0\right)}\right].
	\end{align*}
	\end{enumerate}
	 \item We simulate $n_2=10^5$ independent samples $Z_j$  of  $\sum\limits_{k=0}^{\phi_{1}(0)} \bar{\eta}_1^{k,M}\left( \phi_{k,\tilde{g}_1}\right) $ and we compute the standard estimator $\hat{v}_{2}$ of  $\sum\limits_{k=0}^{\phi_{1}(0)} \mathbb{E}\left(\bar{\eta}_1^{k,M}\left( \phi_{k,\tilde{g}_1}\right) \right)$  with relative $95\%$ confidence interval $\left[ a_{2}, b_{2}\right] $  with $\hat{v}_{2}, a_{2}, b_{2}$ respectively defined as in (\ref{meanest}), (\ref{cimean1}) and (\ref{cimean2}).

	\item We check that $\hat{v}_{1,1} - \frac{\hat{v}_{1,2}}{\left( \bar{\eta}_1^0\left(g_1\right)\right) ^4}-\frac{\bar{\eta}_{0}^0\left(g_{0}\left( P f_1^2-\left(P f_1 \right)^2 \right) \right) }{\left( \bar{\eta}_1^0\left(g_1\right)\right)^4\bar{\eta}_{0}^0\left(g_0\right)}$ is close to $\hat{v}_{2}$ as expected.   
	\end{enumerate}

\subsubsection*{RESULTS} 

\begin{center}	
	\begin{tabular}{|c|c|c|}
		\hline
		n=1	& value & CI  \\
		\hline
		$\hat{v}_{1,1} - \frac{\hat{v}_{1,2}}{\left( \bar{\eta}_1^0\left(g_1\right)\right) ^4}-\frac{\bar{\eta}_{0}^0\left(g_{0}\left( P f_1^2-\left(P f_1 \right)^2 \right) \right) }{\left( \bar{\eta}_1^0\left(g_1\right)\right)^4\bar{\eta}_{0}^0\left(g_0\right)}$	& 0.4729737   & [ 0.4690806 , 0.4768669 ] \\
		\hline
		$\hat{v}_{2}$	&  0.4725217
		& [ 0.4724719 , 0.4725714 ] \\
		\hline
	\end{tabular}
\end{center}


\begin{thebibliography}{2}
	\bibitem{evans}{M. Evans and T. Swartz. Methods for approximating integrals in Statistics
		with special emphasis on Bayesian integration problems. Statist. Sci., 10, 254–272, 1995.}
	 \bibitem{casella}{C. P. Robert and G.Casella. Monte Carlo Statistical Methods. Springer, 2nd ed., 2004.}
\bibitem{cappemoulines}{O. Capp\'{e}, E. Moulines and T. Ryd\'{e}n. Inference in Hidden Markov Models. Springer, 2005.}
\bibitem{mayne}{J. Handschin and D. Mayne. Monte Carlo techniques to estimate the conditionnal expectation in multi-stage non-linear filtering. Int. J. Control, 9, 547–559, 1969. }
\bibitem{handschin}{J. Handschin. Monte Carlo techniques for prediction and filtering of non-
	linear stochastic processes. Automatica, 6, 555–563, 1970. }
	\bibitem{kitagawa}{G. Kitagawa. Monte Carlo filter and smoother for non-gaussian state space models. Journal of Computational and Graphical Statistics, 5 (1), 1-25, 1996. }
	\bibitem{Carpenter}{J. Carpenter, P. Clifford and P. Fearnhead. Improved particle filter for nonlinear problems. IEE Proceedings Radar, Sonar Navigation, 146 (1), 2-7, 1999. }
	\bibitem{Douc}{R. Douc, O. Cappé and E. Moulines. Comparison of resampling schemes for particle filtering. Proceedings of the 4th International Symposium on Image and Signal Processing and Analysis, 64-69, 2005.}
	\bibitem{GordonSalmond}{N. Gordon, D. Salmond and A. Smith. Novel approach to nonlinear/non-Gaussian Bayesian state estimation. IEE Proceedings F (Radar and Signal Processing), 140 (2), 107-113, 1993.}
	\bibitem{Doucet}{A. Doucet, N. de Freitas and N.J. Gordon. Sequential Monte Carlo Methods in Practice. Springer-Verlag, 2001.}
		\bibitem{ristic}{B. Ristic, M. Arulampalam and A. Gordon. Beyond Kalman Filters: Particle Filters for Target Tracking. Artech House, 2004.}
	\bibitem{Liu}{J. Liu and R. Chen. Sequential Monte-Carlo methods for dynamic systems. Journal of the Royal Statistical Society B, 93, 1032–1044, 1998.}	
	\bibitem{Fearnhead}{P. Fearnhead. Sequential Monte Carlo methods in filter theory. PhD thesis, University of Oxford, 1998.}
	\bibitem{Kunsch}{H. R. K\"{u}nsch. Recursive Monte-Carlo filters: algorithms and theoretical analysis. The Annals of Statistics, 33 (5), 1983–2021, 2005.}
	\bibitem{NChopin}{N. Chopin. Central limit theorem for sequential monte carlo methods and its application to bayesian
		inference. The Annals of Statistics, 32 (6), 2385–2411, 2004. }
	\bibitem{Proschan}{K. Joag-Dev and F. Proschan. Negative association of random variables with applications. The Annals of Statistics, 11 (1), 286-295, 1983.}
	\bibitem{chopin2}{N. Chopin, S.S. Singh, T. Soto and M. Vihola. On resampling schemes for particle filters with weakly informative observations. ArXiv, 2022.}
	\bibitem{Ecuyer}{P. L'Ecuyer and C. Lemieux. Variance reduction via lattice rules. Management Science, 46 (9), 1214-1235, 2000.}
			\bibitem{chopin}{M. Gerber, N. Chopin and N. Whiteley. Negative association, ordering and convergence of resampling methods. The Annals of Statistics, 47 (4), 2236–2260, 2019.}
			
	\bibitem{billingsley}{P. Billingsley. Probability and Measure. Wiley series in probability and mathematical statistics:
			probability and mathematical statistics. A Wiley-Interscience Publication, 3rd edn.,
			1995.}
		\bibitem{gamgam}{H. E. Akyuz and H. Gamgam. Robust confidence intervals for the difference of two independent population variances. Hacet. J. Math. Stat., 49 (1), 478-493, 2020.}
		\bibitem{delmoral}{P. Del Moral. Feynman-Kac formulae. Genealogical and interacting particle systems with applications. Probability and its Applications. Springer, 2004.}
		\bibitem{Bonett}{D. G. Bonett. Approximate confidence interval for standard deviation of non normal distributions. Computational Statistics and Data Analysis, 50, 775 – 782, 2006.}
		\bibitem{chopinandpapas}{N. Chopin and O. Papaspiliopoulos. An introduction to sequential Monte Carlo.
		Springer Series in Statistics. Springer, 2020.}
		\bibitem{CLTfract}{R. Flenghi and B. Jourdain. Convergence to the uniform distribution of vectors of partial sums modulo one with a common factor. ArXiv:2308.01874, 2023.}
\end{thebibliography}
\end{document}